\documentclass[12pt]{amsart}
\usepackage{aliascnt} 
\usepackage{amscd} 
\usepackage{amsfonts} 
\usepackage{amsmath} 
\usepackage{amssymb}
\usepackage{amsthm}
\usepackage{appendix}
\usepackage{bm} %
\usepackage{color}
\usepackage{chngpage}
\usepackage[shortlabels]{enumitem}
\usepackage{fancyhdr}
\usepackage{float}
\usepackage{fullpage}
\usepackage{graphics} 
\usepackage{graphicx}
\usepackage[CJKbookmarks=true]{hyperref}
\usepackage[all]{hypcap}
\usepackage{indentfirst}
\usepackage{latexsym}
\usepackage{mathrsfs}
\usepackage{mathtools}
\usepackage{multirow}
\usepackage{pstricks}
\usepackage{rotating}
\usepackage{soul} 
\usepackage{tikz}

\usepackage{url}
\usepackage{xcolor}
\usepackage[all,poly,knot]{xy}
\theoremstyle{plain}
\newtheorem{theorem}{Theorem}[section]

\newtheorem*{acknowledgement*}{\protect \acknowledgementname}
\providecommand{\acknowledgementname}{Acknowledgement}
\newaliascnt{setup}{theorem}
\newtheorem{setup}[setup]{Setup}
\aliascntresetthe{setup}

\newaliascnt{question}{theorem}
\newtheorem{question}[question]{Question}
\aliascntresetthe{question}

\newaliascnt{lemma}{theorem}
\newtheorem{lemma}[lemma]{Lemma}
\aliascntresetthe{lemma}

\newaliascnt{assumption}{theorem}

\aliascntresetthe{assumption}

\newaliascnt{conjecture}{theorem}
\newtheorem{conjecture}[conjecture]{Conjecture}
\aliascntresetthe{conjecture}

\newaliascnt{proposition}{theorem}
\newtheorem{proposition}[proposition]{Proposition}
\aliascntresetthe{proposition}

\newaliascnt{corollary}{theorem}
\newtheorem{corollary}[corollary]{Corollary}
\aliascntresetthe{corollary}

\newaliascnt{problem}{theorem}

\aliascntresetthe{problem}

\theoremstyle{definition}
\newaliascnt{definition}{theorem}
\newtheorem{definition}[definition]{Definition}
\aliascntresetthe{definition}

\newaliascnt{example}{theorem}

\aliascntresetthe{example}

\theoremstyle{remark}
\newaliascnt{remark}{theorem}
\newtheorem{remark}[remark]{Remark}
\aliascntresetthe{remark}

\newaliascnt{claim}{theorem}

\aliascntresetthe{claim}

\newaliascnt{fact}{theorem}

\aliascntresetthe{fact}

\newaliascnt{notation}{theorem}
\newtheorem{notation}[notation]{Notation}
\aliascntresetthe{notation}

\newaliascnt{remarks}{theorem}

\aliascntresetthe{remarks}

\newcommand{\frakp}{{\mathfrak p}}

\newcommand{\bA}{{\mathbb A}}

\newcommand{\bC}{{\mathbb C}}
\newcommand{\bD}{{\mathbb D}}
\newcommand{\bE}{{\mathbb E}}
\newcommand{\bF}{{\mathbb F}}

\newcommand{\bL}{{\mathbb L}}

\newcommand{\bN}{{\mathbb N}}

\newcommand{\bP}{{\mathbb P}}
\newcommand{\bQ}{{\mathbb Q}}

\newcommand{\bV}{{\mathbb V}}
\newcommand{\bW}{{\mathbb W}}

\newcommand{\bZ}{{\mathbb Z}}
\newcommand{\mA}{{\mathcal A}}

\newcommand{\mC}{{\mathcal C}}
\newcommand{\mD}{{\mathcal D}}
\newcommand{\mE}{{\mathcal E}}
\newcommand{\mF}{{\mathcal F}}
\newcommand{\mG}{{\mathcal G}}
\newcommand{\mH}{{\mathcal H}}

\newcommand{\mL}{{\mathcal L}}
\newcommand{\mM}{{\mathcal V}}

\newcommand{\mO}{{\mathcal O}}
\newcommand{\mP}{{\mathcal P}}

\newcommand{\mT}{{\mathcal T}}
\newcommand{\mU}{{\mathcal U}}
\newcommand{\mV}{{\mathcal V}}

\newcommand{\mX}{{\mathcal X}}
\newcommand{\mY}{{\mathcal Y}}
\newcommand{\mZ}{{\mathcal Z}}

\newcommand{\sH}{{\mathscr H}}

\DeclareMathOperator\rmd{d}
\DeclareMathOperator\End{End}

\DeclareMathOperator\Fil{\mathrm{Fil}}

\DeclareMathOperator\Gal{Gal}
\DeclareMathOperator\GL{GL}

\DeclareMathOperator\Gr{Gr}

\DeclareMathOperator\mHom{\mathcal Hom}
\DeclareMathOperator\id{id}

\DeclareMathOperator\Spec{Spec}
\DeclareMathOperator\Spf{Spf}

\DeclareMathOperator\tr{tr}
\newcommand{\dR}{\mathrm{dR}}

\newcommand{\et}{\mathrm{et}}
\newcommand{\MF}{\mathcal{MF}}

\newcommand{\HIG}{\mathrm{HIG}}
\DeclareMathOperator{\FIsoc}{F-Isoc}

\DeclareMathOperator{\Crys}{Crys}
\DeclareMathOperator{\FCrys}{F-Crys}

\newcommand{\Fp}{{\mathbb{F}_p}}
\newcommand{\Fq}{{\mathbb{F}_q}}

\newcommand{\Qbar}{{\overline{\mathbb Q}}}
\newcommand{\Qp}{{\mathbb{Q}_p}}

\newcommand{\Qlbar}{{\overline{\mathbb{Q}}_{\ell}}}

\numberwithin{equation}{section}
\definecolor{backgroundcolor}{RGB}{210,210,210} 
\definecolor{wordscolor}{RGB}{50,50,50} 

\def\FIsoch{{\FIsoc^{\dagger{1\over 2}}_{\lambda} }}
\def\FIsochf{{\FIsoc^{\dagger{1\over 2}}_{\lambda,f} }}
\def\HDFh{\mathrm{HDF}^{{1\over 2}}_{\lambda}}
\def\High{{\HIG^{{\rm gr}\,{1\over 2}}_{\lambda}}}
\def\Loch{{\mathrm{LOC}^{\ell,{1\over 2}}_{\lambda}}}

\def\MdRh{{M^{{1\over 2}}_{{\rm dR}\,\lambda}}}

\def\MFh{\MF^{{1\over 2}}_{\lambda}}
\def\MFhf{\MF^{{1\over 2}}_{\lambda,f}}

\def\PHDFh{\mathrm{PHDF}^{{1\over 2}}_{\lambda}}
\def\PHDFhf{\mathrm{PHDF}^{{1\over 2}}_{\lambda,f}}
\def\PHigh{{\mathrm{PHIG}^{{\rm gr}\,{1\over 2}}_{\lambda}}}
\def\PHighf{{\mathrm{PHIG}^{{\rm gr}\,{1\over 2}}_{\lambda,f}}}
\def\Flow{\mathrm{Flow}}
\def\Frob{\mathrm{Frob}}
\def\EK{{\mathbb E}}

\def\FF{\mathrm{FF}}

\def\BBn{{(\mathbb P^1_{W_n(k)}, D_{W_n(k)})/W_n(k)}}
\def\bark{{\overline k}_0}
\def\barFp{{\overline{\mathbb F}_p}}

\def\END{\bibliographystyle{alpha} \bibliography{abelian_scheme_GL.bib} \end{document}}

\begin{document}

\title{Constructing families of abelian varieties of $\text{GL}_2$-type over $4$-punctured complex projective line via $p$-adic Hodge theory and Langlands correspondence and application to algebraic solutions of Painleve VI equation}

\author{Jinbang Yang}
\email{\href{mailto:yjb@mail.ustc.edu.cn}{yjb@mail.ustc.edu.cn}}
\address{School of Mathematical Sciences, University of Science and Technology of China, Hefei, Anhui 230026, PR China}

\author{Kang Zuo}
\email{\href{mailto:zuok@uni-mainz.de}{zuok@uni-mainz.de}}
\address{School of Mathematics and Statistics, Wuhan University, Luojiashan, Wuchang, Wuhan, Hubei, 430072, P.R. China.; Institut f\"ur Mathematik, Universit\"at Mainz, Mainz 55099, Germany}

\maketitle

\begin{abstract}
we construct infinitely many non-isotrivial families of abelian varieties of $\text{GL}_2$-type over four punctured projective lines with bad reduction of type-$(1/2)_\infty$ via $p$-adic Hodge theory and Langlands correspondence. They lead to algebraic solutions of Painleve VI equation. Recently Lin-Sheng-Wang proved the conjecture on the torsionness of zeros of Kodaira-Spencer maps of those type families. Based on their theorem we show the set of those type families of abelian varieties is {\sl exactly} parameterized by torsion sections of the universal family of elliptic curves modulo the involution. After our paper submitted in arXiv \cite{YaZu23b}, Lam-Litt gave a totally new construction of those abelian schemes by applying Katz’s middle convolution \cite{LaLi23}.
\end{abstract}

\tableofcontents

\section{\bf Introduction}

Let $R$ be a commutative ring with identity and let $X$ be a scheme over $R$. A \emph{family of varieties} over $X$ of dimension $g$ is a flat morphism $\pi\colon Y\rightarrow X$ of finite type with geometric fibers that are pure, $g$-dimensional, connected, and reduced. For any $x\in X$, the fiber of $\pi$ over $x$ is denoted by $Y_x\coloneqq\pi^{-1}(x)$. An \emph{abelian scheme} (or a \emph{smooth family of abelian varieties}) over $X$ is a smooth, projective family of varieties $f\colon A\rightarrow X$ along with a section $s\colon X\rightarrow A$ such that the fiber $(A_x,s_x)$ forms an abelian variety for each $x\in X$. A \emph{family of abelian varieties} is a projective family of varieties $f\colon A\rightarrow X$ along with a section $s\colon X\rightarrow A$ such that, for some nonempty open dense subset $U\subseteq X$, the restricted family $f\mid_U\colon A\times_XU\rightarrow U$ together with the restricted section $s\mid_U$ form an abelian scheme.

An abelian scheme $A$ over a scheme $X$, together with a polarization $\mu$, is said to be \emph{of $\GL_2$-type} if there exists a number field $\EK$ of degree $\dim_X A$ such that the ring of integers $\mO_\EK$ can be embedded into the endomorphism ring $\End_{\mu}(A/X)$. If we want to emphasize the role of $\EK$, we call $A$ \emph{of $\GL_2(\EK)$-type}. Similarly, a family of abelian varieties over $X$ is said to be \emph{of $\GL_2$-type} if its restriction to the smooth locus is of $\GL_2$-type.

Consider a $\GL_2(\EK)$-type family of abelian varieties $f\colon A\rightarrow X$. Let $D$ denote the discriminant locus and $X^0$ denote the smooth locus, which is the complement of $D$ in $X$. We define $\Delta$ as the inverse image of $D$ under the structure morphism $f$ and $A^0$ as the complement of $\Delta$ in $A$. Then we obtain an abelian scheme $f^0\colon A^0\rightarrow X^0$.

For $R= \bC$ we consider the Betti-local system
\[\bV= R^1_\mathrm{B}f^0_* \bZ_{A^0}\]
attached to $f^0$, which is a $\bZ$-local system over the base $X^0$. Since $f$ is of $\GL_2(\EK)$-type, the action of $\mO_\EK$ on $f$ induces an action of $\EK$ on the $\Qbar$-local system $\bV\otimes\Qbar$. Taking the $\EK$-eigen sheaves decomposition
\[\bV\otimes\Qbar= \bigoplus_{i=1}^g \bL_i.\]
Then these $\bL_i$'s are $\Qbar$-local systems of rank-$2$ over $X^0$ and defined over the ring of integers of some number field. On the other hand, consider the logarithmic de Rham bundle attached to the family of abelian varieties $f$ and denote
\[(V,\nabla)= R^1_\dR f_*\Big(\Omega^*_{A/X}(\log\Delta),\rmd\Big).\]
On this de Rham bundle, there is a canonical filtration satisfying Griffiths transversality given by relative differential $1$-forms
\[E^{1,0}\coloneqq R^0f_* \Omega^1_{A/X}(\log\Delta))\subset V.\]
Taking the grading with respect to this filtration, one gets a logarithmic graded Higgs bundle, which is called Kodaira-Spencer map attached to $f$
\begin{equation} \label{equ:KS}
(E,\theta)\coloneqq(E^{1,0} \oplus E^{0,1},\theta)\coloneqq \Gr_{E^{1,0}} (V,\nabla) = \Big(R^0f_*\Omega^1_{A/X}(\log\Delta)\oplus R^1f_*\mO_A,\Gr(\nabla)\Big).
\end{equation}
Since $f$ is of $\GL_2(\EK)$-type, one also gets an $\EK$-eigen decomposition of the Higgs bundle
\begin{equation} \label{equ:decomp_Higgs}
(E,\theta)=\bigoplus_{i=1}^{g} (E,\theta)_i.
\end{equation}
Under Hitchin-Simpson's non-abelian Hodge theory, these eigensheaves $\{(E,\theta)_i\}_{i=1,\cdots,g}$ are just those Higgs bundles correspond to the local systems $\{\bL_i\}_{i=1,\cdots,g}$.

Those local systems and Higgs bundles above are examples of motivic local systems and motivic Higgs bundles, sometimes also called \emph{coming from geometry origin}. Simpson had found a characterization for a rank-$2$ local system to be motivic.
\begin{theorem}[Simpson\cite{Sim92}] \label{thm_Simpson}
A rank-2 local system $\bL$ over a smooth complex quasi-projective curve $U$ is an eigen sheaf of
an abelian scheme of $\GL_2$-type over $U$ if and only if the following two conditions hold:
\begin{enumerate}[$(1).$]
\item $\bL$ is defined over the ring of integers of some number field, and
\item for each element $\sigma \in \mathrm{Gal}(\overline{\bQ}/\bQ)$
the Higgs bundle corresponding to the Galois conjugation $\bL^\sigma$ is again graded.
\end{enumerate}
\end{theorem}
\begin{conjecture}[Simpson] \label{conj_Simpson}
A rigid local system is motivic.
\end{conjecture}
Simpson and Corlette \cite{KeSi08} proved that \autoref{conj_Simpson} holds in the rank-$2$ case. In fact, they showed that a rank-2 rigid local system does satisfy these two properties required in Theorem 1.1. Another crucial point is the construction of the polarization from the harmonic metric on the local system. Simpson's conjecture for rank-3 case has be proven by Langer-Simpson \cite{LaSi18} for cohomologically rigid local systems. The conjecture predicts that any rigid local system $\bL$ should enjoy all properties of motivic local systems. For example,
\begin{itemize}
\item its corresponding filtered de Rham bundle is isomorphic to the underlying filtered de Rham bundle of some Fontaine-Faltings modules at almost all places, and
\item if $\bL$ is in addition cohomologically rigid, then it is defined over the ring of integers of some number field.
\end{itemize}
Those two properties have been verified by Esnault-Groechenig recently \cite{EsGr18,EsGr20}.

We propose a program on searching for loci of motivic Higgs bundles in moduli spaces of semistable Higgs bundles with trivial Chern classes on a given smooth complex quasi-projective variety $X-D$, though the dimensions of moduli spaces could be positive.

Motived by an observation of Kontsevich based on Langlands correspondence over function field of characteristic-$p$ we ask the locus of those arithmetic periodic Higgs bundles in the moduli space over complex numbers:
\begin{conjecture}
\begin{enumerate}
\item There exists a family of self maps $\phi[n]$ on the moduli space of Higgs bundle parametrized by $n\in\bN$ such that
\begin{itemize}
\item their are additive in the following sense
\[\phi[m] \circ \phi[n] = \phi[m+n] =\phi[n] \circ \phi[m];\]
\item the modulo $p$ reduction is birationally equivalent to the self map induced by Higgs-de Rham flow.
\end{itemize}
\item A Higgs bundle is motivic if and only if it is torsion under the map $\phi[n]$ for some $n$.
\end{enumerate}
\end{conjecture}

In this note, we make the first step towards to this program by taking $X$ as the complex projective line $\bP^1_{\bC}$ and $D$ as the $4$ punctures $\{0,1,\infty,\lambda\}$. In this case the moduli space of rank-2 semistable Higgs bundles on $\bP^1$ with prescribed parabolic structure on $4$-punctures always has positive dimension.
Our goal is looking for the locus of rank-$2$ motivic graded Higgs bundles over $(\bP^1_{\bC},\{0,1,\infty,\lambda\})$
with prescribed parabolic structure at four punctures $\{0,1,\lambda,\infty\}.$

Beauville \cite{Bea82} has shown that there exist exactly 6 non-isotrivial families of elliptic curves over $\bP^1_{\bC}$ with semistable reductions over $\{0,1,\infty,\lambda_i\}$ for ${1\leq i\leq 6}$. All of them are modular curves of certain mixed level structures. The same statement has also shown by Viehweg-Zuo for families of higher dimension abelian varieties on $\bP^1$ with semistable reduction on $4$-punctures.
So except Beauville's example any non-isotrivial smooth families of abelian varieties over $\bP^1\setminus\{0,1,\infty,\lambda\}$ of $\GL_2$-type must have non-semistable reduction at some point in $\{0,1,\infty,\lambda\}$. In this case, the some eigenvalues of the local monodromies of motivic local system must be roots of unity other than $1$.

We consider the Legendre family of elliptic curves over $\bP^1$ defined by the equation
\[ y^2=x(x-1)(x-\lambda),\quad \lambda \in \bP^1-\{0,\,1,\infty\}.\]
The family has semistable reduction over $\{0,1\}$ and potentially semistable reduction
over $\{\infty\}$ with local monodromy around $\infty$ of eigenvalues $\{e^{2i\pi \over 2},\,e^{2i\pi \over 2}\}$.
Such a family is said to have bad reduction at discriminant locus of type-$(1/2)_\infty$. We are motivated by this example to search for more families of elliptic curves/abelian varieties over $\bP^1$ with bad reduction on $4$-punctures of type-$(1/2)_\infty$.

In the paper \cite{SYZ22} we studied rank-2 $p$-adic graded Higgs bundles on $4$-punctured $\bP^1$ with prescribed parabolic structure on punctures of type-$(1/2)_\infty$. Our motivation for this study was Simpson's theorem on rank-$2$ motivic Higgs over complex number field. Specifically, we aimed to find motivic Higgs bundles that are graded Higgs bundles from Fontaine-Faltings modules.

By Fontaine-Faltings' work on crystalline local systems and the work by Lan-Sheng-Zuo on Higgs-de Rham flow \cite{LSZ19}, a motivic Higgs bundles must be periodic points of the self map of Higgs-de Rham flow. We shall point out, the notion of Higgs-de Rham flow has been already introduced in a unpublished paper \cite{ShZu12} by M. Sheng and K. Zuo for the category of sub Higgs bundles in graded Higgs bundles arising from Fontaine-Faltings modules. Though the main object is the category of sub Higgs bundles in a given graded Higgs bundle from Fontaine-Faltings module, the lifting of inverse Cartier transform on the category of the category of sub Higgs bundles over $W_{n}(k)$ which are periodic (modulo $p^{n-1}$) has been originally constructed in this paper).

In \cite{SYZ22} one has found the explicit expression of the self map. By identifying the moduli space $\HIG^{\text{gr} {1\over 2}}_\lambda $ of Higgs bundles on $\bP^1_k$ with parabolic structure on the punctures $\{0,1,\lambda,\infty\}$ of type-$(1/2)_\infty$ with $\bP^1_k$ then the self map on $\High$ is a polynomial map of degree $p^2$ composed with the Frobenius map, See \cite[appendix A]{SYZ22}.

More mysterious things happen, we define the elliptic curve $C_\lambda$ associated to a $4$-punctured $\bP^1$ as the double cover
\[ \pi: C_\lambda\to \bP^1\] ramified on $\{0,1,\lambda,\infty\}$
and choosing $\infty$ as the origin for group law, we have examined the formula for the self-map for primes $2< p< 50$ and found that the self map on $\HIG^{\text{gr} {1\over 2}}_\lambda =\bP^1$ coincides with the multiplication by $p$ map on the elliptic curve $C_\lambda$ via $\pi$. Consequently if $(E,\theta)$ is period, (i.e. it is the grading of a $p$-torsion Fontaine-Faltings module) if and only if the zero of the Higgs field $(\theta)_0\in \bP^1(\bF_q)$ is the image of a torsion point in $C_\lambda$. See the more detailed discussions after \autoref{thm_main_painleve} and \autoref{conj:SYZ}.

Given an abelian scheme $A$ over $\mO[1/N]$, where $\mO$ is the ring of integers of some number field $K$ and $N$ is a positive integer. Pink's theorem \cite{Pin04} implies that a point $z\in A(K)$ is a torsion point if and only the order of the modulo $\frakp$ reduction $z\pmod{\frakp} \in A(k_{\frakp})$ is bounded above by some number which is independent of the choice of the finite place $\frakp$. It motivates us to make the following conjecture (in a talk held in Lyon by the second named author in April 2018).
\begin{conjecture}\label{Conj_2}
A complex semistable parabolic graded Higgs bundle of degree $0$ on the projective line with $4$-punctures $\{0,1,\lambda,\infty\}$
of parabolic type-$(1/2)_\infty$ is motivic if and only if the zero of the Higgs field $(\theta)_0$ is a torsion point in $C_\lambda$.
\end{conjecture}
In particular, \autoref{Conj_2} implies that there exist infinitely many rank-2 motivic Higgs bundles on any complex $4$-punctured $\bP^1$ of parabolic type-$(1/2)_\infty$.

J. Lu, X. Lv and J. Yang have found $26$ (classes of) families of complex elliptic curves on $\bP^1$ with bad reductions on $\{0,1,\lambda,\infty\}$ such that the zero of Kodaira-Spencer maps are torsion of order $1$, $2$, $3$, $4$ and $6$, by applying Voisin's result on Jacobian ring and computer program, see the families in \autoref{sec_appendix_A}.

In a recent preprint \cite{LSW22}, Lin-Sheng-Wang have solved one direction of \autoref{Conj_2}.
\begin{theorem}[Lin-Sheng-Wang] \label{thm_LSW_torsion}
If $(E,\theta)$ is a rank-2 motivic Higgs bundle on a $4$-punctured $\bP^1$ over $\bC$ with parabolic structure on the $4$-punctures of type-$(1/2)_\infty$ then the zero the Higgs field $(\theta)_0$ is the image of a torsion point in $C_\lambda$.
\end{theorem}
Actually they have solved \autoref{conj:SYZ} on the property of the torsioness of zeros of Higgs fields of graded Higgs bundles come from $p$-torsion Fontaine-Faltings modules. Combining this characteristic $p$ result with the Pink's theorem mentioned above they have obtained \autoref{thm_LSW_torsion}.

Our first result in this paper shows the existence part claimed in \autoref{Conj_2}.
\begin{theorem}[\autoref{thm_main_motivic_torsion_description}]\label{thm_main}
A complex semistable parabolic graded Higgs bundle of degree $0$ on the projective line with $4$-punctures $\{0,1,\lambda,\infty\}$
of parabolic type-$(1/2)_\infty$ is motivic if the zero of the Higgs field $(\theta)_0$ is a the image of torsion point in $C_\lambda$.
\end{theorem}

\begin{remark} For given $4$-punctured complex projective line $(\bP^1,\{ 0,\,1,\,\lambda,\,\infty\})$,\autoref{thm_main} implies that there exists infinitely many non-isotrivial $\GL_2$-type families of abelian varieties over $\bP^1$ with the discriminant locus contained in $\{0,1,\infty,\lambda\}$ and whose associated rank-$2$ eigen local systems are of type-$(1/2)_\infty$.
\end{remark}

Let $M_{0,n}$ denote the moduli space of isomorphism classes of $n$-marked projective line (i.e. projective line with $n$-ordered distinct marked points). Let $S_{0,n}$ denote the total space of the universal family of $n$-marked projective line with structure morphism
\[p_n\colon S_{0,n}\rightarrow M_{0,n}.\]
Then $\bP^1\setminus\{0,1,\infty\}$ is naturally isomorphic to $M_{0,4}$ by sending $\lambda$ to isomorphic class of the projective line with $4$-marked points $\{0,1,\infty,\lambda\}$. Once we identify $M_{0,4}$ with $\bP^1\setminus\{0,1,\infty\}$
\[M_{0,4} = \bP^1\setminus\{0,1,\infty\},\]
then $S_{0,4}= \bP^1\setminus\{0,1,\infty\} \times \bP^1$ is an algebraic surface, the structure morphism $p_4$ is given by $p_4(\lambda,z)=\lambda$ and $4$ structure sections are $\sigma_0(\lambda)=(\lambda,0)$, $\sigma_0(\lambda)=(\lambda,1)$, $\sigma_0(\lambda)=(\lambda,\infty)$, $\sigma_0(\lambda)=(\lambda,\lambda)$.

The following result is the heart part in our paper.
\begin{theorem}[\autoref{thm_main_Higg_mod_p_to_largest_family}]
\label{thm_mainII}
Let $L$ be a number field and let $\lambda_0 \in M_{0,4}(L)$. Assume $\frakp$ is a finite place such that $\lambda_0$ is a $\frakp$-adic integer and $\lambda_0$ is supersingular at $\frakp$ in the sense \autoref{def_supersingular}. For any
$(\overline E,\overline\theta)\in \HIG_{\lambda_0}^{{\rm gr}{1\over2}}(\overline{k}_{\frakp})$, denote by $(E,\theta)$ the unique motivic lifting in $\HIG_{\lambda_0}^{{\rm gr}{1\over2}}(\overline{\bQ})$ and denote by $f_{\lambda_0}$ the family constructed in\autoref{thm_family_from_W_to_number_field}. Then there exists a finite \'etale covering $\widetilde M_{0,4}\to M_{0,4}$ (depending on $(\overline{E},\overline{\theta})$) such that $f_{\lambda_0}$ can be extended\footnote{In other words, there exists a point $\widehat{\lambda_0}$ in the preimage of $\lambda_0$ under $\widetilde M_{0,4}\to M_{0,4}$ with
\[f\mid_{S_{0,4}\times_{M_{0,4}}\{\widehat{\lambda_0}\}} \cong  f_{\lambda_0}.\]} to an abelian scheme
\[f\colon A\to \widetilde S_{0,4}= S_{0,4}\times_{M_{0,4}} \widetilde M_{0,4}\]
of $\text{GL}_2$-type,
with bad reduction on the four punctures such that the local monodromies of $f$ around $\{0,1,\lambda\}$ are unipotent and around $\{\infty\}$ is quasi-unipotent with all eigenvalues being $-1$.
\end{theorem}

It is clear that each rank-2 eigen sheaf $\widetilde{\bL}$ associated to $f$ is a local system on $\widetilde S_{0,4}$ arising from isomonodromy deformation of an eigen sheaf $\bL_{\lambda_0}$ associated to the family of abelian varieties restricted to the fiber over $\lambda_0$
\[f_{\lambda_0} \colon A_{\lambda_0}\to \bP^1\]
with bad reduction on $\{0,1,\lambda,\infty\}$ of type-$(1/2)_\infty$.
\begin{corollary} \label{thm_main_painleve}
Let $f\colon A\to\widetilde S_{0,4}$ be a family given in \autoref{thm_mainII}. Then all rank-$2$ eigen local systems associated to the family $f$ are algebraic solutions of Painleve VI equation of the type-$(1/2)_\infty$.
\end{corollary}

For given $\lambda\in\bP^1\setminus\{0,1,\infty\}$, any family $f_\lambda\colon A_\lambda\rightarrow\bP^1$ in \autoref{thm_main} has semistable reduction over $\{0,1,\lambda\}$ and potentially semistable reduction over $\infty$. Thus the eigen Higgs bundles $(E,\theta)_i$ (constructed in \eqref{equ:decomp_Higgs}) associated to this family have the following form
\begin{equation} \label{equ:Higgs_form}
E_i=\mO\oplus\mO(-1),\qquad\theta_i\colon\mO\xrightarrow{\neq0} \mO(-1) \otimes\Omega^1_{\bP^1}(\log\{0,1,\infty,\lambda\})
\end{equation}
and are endowed with natural parabolic structures on the punctures $\{0,1,\infty,\lambda\}$ of type-$(1/2)_\infty$.
Here type-$(1/2)_\infty$ parabolic structures means that the parabolic structures at $0$, $1$ and $\lambda$ are trivial and the parabolic filtration at $\infty$ is
\[\left(E_{i}\mid_\infty\right)_\alpha=\left\{\begin{array}{cc}
E_{i}\mid_\infty & 0\leq\alpha\leq1/2,\\
0 & 1/2 <\alpha < 1.\\
\end{array}\right.\]
Let $\High$ denote the moduli space of rank-2 semi-stable graded Higgs bundles over $\bP^1$ with the parabolic structure on $\{0,1,\infty,\lambda\}$ of type-$(1/2)_\infty$ and with parabolic degree $0$. Then any Higgs bundle $(E,\theta)\in\High$ is parabolic stable and has the form as in \eqref{equ:Higgs_form}.

In view of $p$-adic Hodge theory, a Higgs bundle $(E,\theta)$ over the Witt ring $W(\Fq)$ realized by a family of abelian varieties over $W(\Fq)$ of $\GL_2(\EK)$-type has to be the grading of an $\EK$-eigen sheaf of the Fontaine-Faltings module
attached to the family of abelian varieties. Hence, by Lan-Sheng-Zuo functor, the graded Higgs bundle $(E,\theta)$ is \emph{periodic} on $\High$ over $W(\Fq)$ under the map induced by Higgs-de Rham flow.

One identifies the moduli space $\High$ with the projective line $\bP^1$ by sending $(E,\theta)$ to the zero locus of the Higgs map $(\theta)_0\in\bP^1$
\[\High= \bP^1.\]
Let $C_{\lambda}$ be the elliptic curve defined by the Weierstrass function $y^2=z(z-1)(z- \lambda)$, which is just the double cover of the projective line ramified on $\{0,1,\infty,\lambda\}$
\[\pi\colon C_\lambda\to\bP^1.\]
\begin{conjecture} [Sun-Yang-Zuo \cite{SYZ22}] \label{conj:SYZ}
The self-map $\phi$ induced by Higgs-de Rham flow on ${\High}\otimes {\Fq}$ comes from multiplication-by-$p$ map on the elliptic curve $C_\lambda\otimes\Fq$. In other words, the following diagram commutes
\[\xymatrix{
& C_\lambda\otimes {\Fq} \ar[d]_{\pi} \ar[r]^{[p]} & C_\lambda\otimes {\Fq} \ar[d]^{\pi} & \\
\High \otimes {\Fq}\ar@/_12pt/[rrr]_{\phi} \ar@{=}[r] & \bP^1_{\Fq} \ar[r] & \bP^1_{\Fq} \ar@{=}[r] & \High\otimes {\Fq} \\
}\]
\end{conjecture}

The conjecture implies two things:
\begin{enumerate}[$(1).$]
\item a Higgs bundle $(E,\theta)$ is $f$-periodic under the map $\phi$ if and only if the two points in $\pi^{-1}(\theta)_0$ are both torsion in $C_\lambda$ and of order $p^f\pm1$.
\item for a prime $p>2$ and assume $C_\lambda$ is supersingular then $\phi_\lambda(z)=z^{p^2}$. Hence, any Higgs bundle $(E,\theta)\in \High(\overline{\bF}_{q})$ is periodic.
\end{enumerate}

The \autoref{conj:SYZ} has been checked by Sun-Yang-Zuo for $p<50$. Very recently it has been proved by
Lin-Sheng-Wang and becomes a theorem.
\begin{theorem} [Lin-Sheng-Wang \cite{LSW22}] \label{thm_LSW}
\autoref{conj:SYZ} holds true.
\end{theorem}
\autoref{thm_mainII} combined with \autoref{thm_LSW} lead us to prove \autoref{thm_main} the part of the existence of
rank-2 motivic Higgs bundles in terms of torsioness of zeros of Higgs fields claimed in \autoref{Conj_2}.
\begin{remark}
\autoref{thm_main} and \autoref{thm_mainII} implies that $\widetilde M_{0,4}$ is a moduli curve. It looks very interesting to such kind of properties on modularity appeared already as modular forms in the work by C. S. Lin and C. L. Wang on Painleve VI and Lame equations \cite{LiWa10}
\end{remark}

Given a semi-stable Higgs bundle $(E,\theta)$ with trivial Chern classes on a smooth scheme $\mX$ over the ring of integers of some number field. Then for almost all finite places $\frakp$ the reduction $(E,\theta)\pmod{\frakp}$ is semistable. Thus, the Higgs bundle $(E,\theta)\pmod{\frakp}$ is preperiodic under the Higgs-de Rham flow. We take the length $\ell_{(E,\theta)\pmod{\frakp}}$ of the periodicity of $(E,\theta)\pmod{\frakp}$ at $p$.
\begin{corollary}
A Higgs bundle $(E,\theta)\in \High(\overline{\bQ})$ is motivic if and only if
the set of preperiodic lengths $\{ \ell_{(E,\theta)\pmod{\frakp}} \}_{\frakp}$ is bounded above.
\end{corollary}
\begin{conjecture} A semistable Higgs bundle with trivial Chern classes on a smooth scheme $\mX$ over the ring of integers of some number field is motivic if and only if the set of preperiodic lengths is bounded above.
\end{conjecture}
Consider an $n$-marked projective line $(\bP^1,\{x_1,\cdots,x_n\}=:D),\,n\geq 4$.. Then the moduli space
$\HIG_D^{\text{gr},{1\over 2}}$ of rank-2 semistable graded Higgs bundles on $\bP^1$ of degree zero with parabolic structure of type-$(1/2)_{x_n}$ contains a component isomorphic to $\bP^{n-3}$ of the maximal dimension.
In \cite{SYZ22} we showed that $\bP^{n-3}(\overline {\bF}_p)$ contains a dense set $\HIG_D^{\text{per},{1\over 2}}( \overline {\bF}_p)$ of periodic Higgs bundles.
\begin{question}
\begin{enumerate}[$(1).$]
\item Does the self-map of Higgs-de Rham on $\HIG_D^{\text{gr},{1\over 2}}$ come from a multiplication by $p$ map on an abelian variety associated the $n$-punctured projective line?
\item Can we find motivic Higgs bundles in $\HIG_D^{\text{per},{1\over 2}}( \overline {\bF}_p)$?
Can they be characterized by torsion points on the possible existing abelian variety?
\end{enumerate}
\end{question}

\begin{acknowledgement*} Our program on searching for loci of motivic Higgs bundles in moduli spaces is highly inspired by Simpson's conjecture and question on motivic local systems. We thank him for explaining on his theorem on rank-2 motivic local systems.

We thank Ariyan Javanpeykar for discussing on the relation between the self map and the multiplication by $p$ map on supersingular elliptic curves.

After a lecture by the first named author held at University Mainz in November 2017, Duco Van Straten introduced Kontsevich's proposal on a characterization of rank-2 $\ell$-adic local systems on $4$-puctured $\bP^1$ as fixed points of possibly existing additive maps on moduli spaces. It made us seriously to consider a possible connection between $p$-adic theory and $\ell$-adic theory. We thank him for very useful conversations.

We thank warmly Raju Krishnamoorthy for his constant support to our work. We have learned from him about Deligne's $p$-to$\ell$ companion, techniques of $F$-isocrystals and Kato's work.

We thank Tomoyuki Abe for answering our question on the compatibility between local and global Langlands correspondence in several email exchanges.

We thank Mao Sheng for discussing on the solution by Lin-Sheng-Wang of \autoref{conj:SYZ} on torsioness of Kodaira-Spencer map, Certainly their theorem plays a very important role in our paper.

We thank Hongjie Yu for showing us his beautiful solution of Deligne's conjecture on counting numbers $\ell$-adic local systems on punctured projective lines in terms of numbers of parabolic graded Higgs bundles over finite fields. His theorem is very crucial in our paper.

During the preparation of this paper we also got a lot of benefits of sharing ideas and working knowledge from the following people. The authors warmly thank Chung Pan Mok, Ruiran Sun, Chin-Lung Wang , Junyi Xie,
Shing-Tung Yau.

J. Yang is supported by National Natural Science Foundation of China Grant No. 12201595, the Fundamental Research Funds for the Central Universities and CAS Project for Young Scientists in Basic Research Grant No. YSBR-032.
\end{acknowledgement*}

\section*{\bf Structure of the paper}

\subsection{ The idea}

The underlying principle behind the proof is very simple, the so-called isomonodromy deformation for motivic local systems over mixed characteristic.
Let's first look at the situation over complex numbers. We assume, there exists a family of abelian varieties
$f_{\lambda_0}\colon A_{\lambda_0}\to \bP^1_\bC$ of $\GL_2(\EK)$-type over complex projective line with bad reduction on $\{0,1,\lambda_0,\infty\}$ of type-(1/2).
Then the filtered logarithmic de Rham bundle decomposes as $\EK$-eigen sheaves
\[(V,\nabla,E^{1,0})=:( R^1_\dR f_* \Omega^*_{A_{\lambda_0}/\bP^1}(\log \Delta),d),R^0f_* \Omega^1_{A_{\lambda_0}/\bP^1}(\log \Delta))=\bigoplus_{i=1}^g(V,\nabla,E^{1,0})_i,\]
where each eigen sheaf has the form
\[(V,\nabla,E^{1,0})_i\simeq (\mO\oplus \mO(-1),\nabla_i,\mO). \]
Consider the universal family of $4$-punctured lines
\[S_{0,4}|_{\hat U_{\lambda_0}}\to \hat U_{\lambda_0}\]
over a formal neighborhood $\hat U_{\lambda_0}\subset M_{0\,4}$ of $\lambda_0$.
Then by forgetting the Hodge filtration the de Rham bundle extends to a de Rham bundle
$(V,\nabla)_{S_{0,4}|_{\hat U_{\lambda_0}}}$
over $S_{0,4}|_{\hat U_{\lambda_0}}$. It is known the family of abelian varieties extends over $S_{0,4}|_{{\hat U_{\lambda_0}}}$
if and only if the Hodge filtration $E^{1,0}$ extends to a sub bundle in the de Rham bundle $(V,\nabla)_{S_{0,4}|_{\hat U_{\lambda_0}}}$. Using the $\EK$-eigen sheave decomposition we see that the obstruction for extending the Hodge filtration $E^{1,0}=\bigoplus_{i=1}^g \mO$
lies in
$\bigoplus_{i=1}^g H^1(\bP^1,\mO(-1))=0$. Hence, the family of abelian varieties $f_{\lambda_0}$
extends over the base $S_{0,4}|_{\hat U_{\lambda_0}}$. A standard argument on the moduli space of period mappings from curves with fixed genus shows that this formal extension leads an algebraic extension on $S_{0,4}$.

Back to the situation over mixed characteristic, along the diagram below. We like to construct a family of abelian varieties on $\bP^1_{\bF_q}$ with bad reduction on $4$-punctures of type-$(1/2)_\infty$ and such that the Hodge filtration can be lifted as a sub bundle in the Dieudann\'e module attached to this family over characteristic-zero and then a type of Grothendieck-Messing-Kato logarithmic deformation theorem for log classifying mapping applies.
\begin{equation*}
\xymatrix@C=3cm @R=1.5cm{
\left\{\text{Higgs}\atop \text{bundles}\right\}&\left\{\text{periodic}\atop \text{Higgs bundles}\right\} \ar@{_(->}[l]&\\
\left\{\text{Crystalline}\atop \text{representations}\right\}
\ar@{<->}[r]^-{\text{Fontaine-Faltings}}_-{p\text{\rm-adic RH}}
& \left\{\text{Fontaine-Faltings}\atop \text{modules}\right\}
\ar@{^(->}[r]^-{\text{forgetting Hodge}}_-{\text{filtration,}\otimes \bQ_p}
\ar@{<->}[u]_{\text{Higgs-de Rham flow}}^{\text{by Lan-Sheng-Zuo}}
& \left\{\text{Overconvergent}\atop \text{$F$-isocrystal}\right\}
\ar@{<->}[d]^{\text{Deligne's } p-\ell}_{\text{companion by Abe}} \\
&& \left\{\text{$\ell$-adic}\atop \text{representations}\right\}
& \\
}
\end{equation*}

\subsection{Technique steps}

\subsubsection{\bf A bijection from the set of parabolic graded semi stable Higgs bundles over $\bF_q$ to the set of $\ell$-adic local systems over $\bF_q$ via Abe's theorem on Deligne's $p$-to-$\ell$ companion.}

\begin{itemize}
\item In \autoref{sec_main_para}, we recall the notion of parabolic objects from \cite{YaZu23c} and classifying results of rank-$2$ de Rham bundles and Higgs bundles on $\bP^1$ with parabolic structures on $4$-punctures of type-$(1/2)_\infty$ in \autoref{thm_ClassfyR2PdE} and \autoref{thm_ClassfyR2PHiggs}. For supersingular $\lambda\in W(k)$, we show that every Higgs bundle on $\bP^1_k$ of type-$(1/2)_\infty$ is periodic and lifts uniquely as a periodic Higgs bundle on $\bP^1_{W(k)}$. In other words, there is an natural injection
\[\High(k) \hookrightarrow [\MFh(W(k))]\]
from the set of Higgs bundles on $\bP^1_k$ with parabolic structure on $\{0,1,\overline\lambda,\infty\}$ of type-$(1/2)_\infty$
to the set of Fontaine-Faltings modules on $\bP^1_W(k)$ with parabolic structure on $\{0,1,\lambda,\infty\}
$ of type-$(1/2)_\infty$ modulo an equivalent relation. The main result in this section is \autoref{mthm_PHIGk2PHIGW}.

\item In \autoref{sec_main_F_Isoc}, we consider the set of rank-2 overconvergent $F$-isocrystals over $4$-punctured projective line $(\bP^1,\{0,1,\lambda,\infty\}$ over $k$ with given exponents. In \autoref{mthm_FF2Isoc_classes} and \autoref{thm_PHIG_k_to_F_Isoc_k}, we show an injective map from
$\HIG^{\text{gr} {1\over 2}}_{\overline \lambda}(k) $ to a set of overconvergent $F$-isocrystals with given exponent and of trivial determinants.
\item In \autoref{sec_main_p_to_l_bijections}, we make use of Abe's theorem on Deligne's $p$-to-$\ell$ companion to the set of overconvergent $F$-isocrystals with given exponent. Composed with the injective maps obtained in the previous sections, in \autoref{mthm_HIGk2Lock}, the $p$-to-$\ell$ companion induces an injective map
\[ \High(k) \hookrightarrow \Loch(\bark)^{\Frob_k},\]
the latter consists of rank-$2$ geometric local systems on $(\bP^1\setminus \{0,1,\lambda,\infty\})_{\bark}$ of local monodromy around the punctures of type-$(1/2)_\infty$ and stabilized by $\Frob_k$.
\item By Yu's formula\autoref{thm_Yu} for numeric Simpson correspondence, the above injective map is actually bijective, \autoref{thm_genuineSimCorr}
\begin{equation} \label{eq_bij_Hig_k_to_loc_bark}
\High(k) \xrightarrow{\ \simeq\ } \Loch(\bark)^{\Frob_k}.
\end{equation}
As a consequence, the trace field of any local system in $\Loch(\bark)^{\Frob_k}$ is unramified at $p$.
\end{itemize}

\subsubsection{\bf Constructing families of abelian varieties over $\Fq$ by Drinfeld's work on Langlands correspondence over characteristic $p$ and lifting Hodge filtrations characteristic zero.}

Given a local system $\bL\in \Loch(k_2')$ fixed by the Frobenius $\text{Frob}_k$ with cyclotomic determinant. By applying Drinfeld's result \autoref{thm_Drinfeld_GL2}, in \autoref{sec_main_family_over_k}, we find a family of abelian varieties of $\GL_2(\EK)$-type
\[f'\colon A' \to (\bP^1-\{0,1,\lambda,\infty\})_{k'}=:U_{k_2'}\]
such that $\bL$ appears as an eigen $\ell$-adic local system and all other eigen local systems are located in $\Loch(k_2')$ and fixed by the Frobenius $\text{Frob}_k$ with cyclotomic determinant.

Consider the Dieudonn\'e crystal\footnote{By Kato's \autoref{thm_Kato_equivalent}, we identify the underlying crystal with its realization over the formal completion of $U_{W(k)}$.}
\[(V,\nabla,\Phi,\mV)\]
attached to $f'$, which is automatically overconvergent. After extending the coefficient from $\Qp$ to $\bQ_{p^f}$, the $\GL_2(\EK)$-structure induces an $\EK$-eigen sheaves decomposition of overconvergent $F$-isocrystals with $\bQ_{p^f}$-coefficients
\[(V,\nabla,\Phi)_{\bQ_p}=\bigoplus _{i=1}^g(V,\nabla,\Phi)_{i}\]
where $(V,\nabla,\Phi)_{i}$ has cyclotomic determinant. By construction of the bijection \eqref{eq_bij_Hig_k_to_loc_bark}, each $(V,\nabla,\Phi)_i$ has an integral extension, which underlies a Fontaine-Faltings module with a $\bZ_{p^f}$-endomorphism structure. As a consequence, there exists an isogeny of the Dieudonn\'e crystal of $f'$, which carries a Hodge filtration $\Fil$. According to equivalence of the Dieudonn\'e functor, from this isogeny Dieudonn\'e crystal, one gets a $p$-isogeny
\begin{equation} \label{eq_family_char_p}
f\colon A\to \bP^1_{k_2'}
\end{equation}
of the original family of abelian varieties such that the isogeny Dieudonn\'e crystal is isomorphic to that attached to $f$, see in \autoref{thm_construction_family_mod_p}. In this case, the Hodge filtration $E^{1,0}_{f}$ attached to $f$ coincides with $\Fil\otimes k$. Thus $E^{1,0}_{f}$ lifts to characteristic zero.

\subsubsection{\bf Lifting families of abelian varieties from characteristic $p$ to characteristic zero by Grothendieck-Messing-Kato logarithmic deformation theorem}

The goal in \autoref{sec_main_lifting} is to lift the family $f$ in \autoref{eq_family_char_p} from characteristic $p$ to characteristic $0$. Our idea is to lift the ``classifying mapping'' attached to this family. Because once a family is obtained by pulling back of some universal family along a classifying mapping, then to lift the given family is equivalent to lift the classifying mapping.

To get a classifying mapping attached to our family $f$, we need to choose a good moduli spaces, add a level structure and a principal polarization structure on the family.

For the moduli space, we take the fine arithmetic moduli space $\mA_{8g,1,3}=: \mX^0$ of principle polarized abelian varieties with level-3 exists over $\bZ[e^{{2i \pi \over 3}},1/3]$. The advantage is that the moduli space and universal family has good compactifications by a theorem due to Faltings-Chai \cite{FaCh90}.

The strategy for adding level structure is pulling back along some finite covering mapping \autoref{sec_subsub_level}, and that for adding principal polarization structure is to utilize the Zarkin's trick \autoref{sec_subsub_Zarhin}. After these proceeding, one gets new family
\[f^{(4,4)}\colon A^{4,4} \rightarrow C_k\setminus D_k\]
which carries a principle polarization and a full $3$-level structure. By the universal property of moduli space $\mA_{8g,1,3}$, one obtains a classifying mapping \autoref{thm_classifying_mapping_k}
\[\overline{\varphi}_k \colon C_k \rightarrow \overline{\mA}_{8g,1,3}.\]

Moreover, the Hodge filtration attached to this family can be lifted to characteristic $0$ by the discuss in \autoref{sec_main_family_over_k}. In \autoref{sec_sub_polarization_and_Fil}, we show that the polarization is compatible with the lifting Hodge filtration \autoref{thm_polarization_and_filtration}. Then the classifying mapping $\overline{\varphi}_k$ lifts to a mapping
\[\overline{\psi}_{W(k)} \colon C_{W(k)} \rightarrow \overline{\mA}_{8g,1,3}\]
by applying the main result \autoref{thm_main_comparing_obstruction} in \autoref{sec_main_compare_obstructions}, which identifies the obstruction of lifting the Hodge filtration with that of lifting the classifying mapping via Faltings-Chai universal Kodaira-Spencer map \cite{FaCh90}.

By the rigidity of $\overline{\varphi}_{W(k)}$ (see \cite[Section 4]{KYZ22}), the family is actually defined over some number field. By using Weil restriction and Simpson's \autoref{thm_Simpson}, we show $\overline{\varphi}_{W(k)}$ splits out a family of abelian varieties of $\GL_2$-type over the projective line such that the given Higgs bundle $(E,\theta)$ appears as an eigen sheaf \autoref{thm_main_Higg_mod_p_to_largest_family}.

\section*{Conventions, Notation, and Terminology}
For convenience, we explicitly state conventions and notations. These are in full force unless otherwise stated.
\begin{itemize}
\item Let $k_0$ be a finite field with cardinality.
\item Let $\lambda\in W(k_0)$ be an element satisfying $\lambda\neq 0,1\pmod{p}$. By abusing notation, we sometimes using $\lambda$ to stand for $\lambda\pmod{p} \in k_0$.
\item For any finite extension $k$ of $k_0$, let $k_n$ denote the field extension of $k$ of degree $n$ for any $n\geq1$.
\item For any finite extension $k$ of $k_0$, assume it has cardinality $p^h$, the $h$-iteration of the absolute Frobenius on $\bP^1_{k}$ is a morphism of $k$-schemes preserves the divisor $\{0,1,\lambda,\infty\}$. We denote it by
\[\Frob_k\colon (\bP_k^1,\{0,1,\lambda,\infty\}) \rightarrow (\bP_k^1,\{0,1,\lambda,\infty\}).\]
By abusing notation, we also use $\Frob_k$ to stand for its base change to $k_n$ or $\bark$
\[ \Frob_k\colon (\bP_{k_n}^1,\{0,1,\lambda,\infty\}) \rightarrow (\bP_{k_n}^1,\{0,1,\lambda,\infty\}),\]
\[ \Frob_k\colon (\bP_{\bark}^1,\{0,1,\lambda,\infty\}) \rightarrow (\bP_{\bark}^1,\{0,1,\lambda,\infty\}).\]
\end{itemize}

\section {\bf Parabolic Fontaine-Faltings Modules and parabolic Higgs-de Rham Flows} \label{sec_main_FF_and_Higgs_de_Rham}

In this section, the aim is to establish the following bijective and injective maps
\begin{equation} \label{equ_main_sec_FFHDF}
\xymatrix@C=1.5cm{
\PHighf(W(k)) \ar[d]_{1:1}^{{\autoref{mthm_PHIGf2PHIG}\atop \autoref{mthm_HIG2PHIG}} \atop \autoref{mthm_PHIGk2PHIGW}} & [\PHDFhf(W(k))] \ar[l]_{1:1}^{\autoref{mthm_PHDFf2PHIGf}} \ar@{>->}[r]_{\autoref{mthm_PHDFf2MFf}}& [\MFh(W(k'))_{\bZ_{p^f}}] \ar@{>->}[d]_{\autoref{thm_cyclDeterminant_FFM}} \\
\High(k)\ar@{>..>}[rr] && [\MFh(W(k'_2))_{\bZ_{p^f}}^{\rm cy}] \\}
\end{equation}
for supersingular $\lambda$(\autoref{def_supersingular}), where $k'$ is a finite extension of $k$ containing $\bF_{p^f}$, $k_2'$ is the extension of $k'$ of degree $2$, and
\begin{itemize}
\item $\High(k)$ is the set of all isomorphic classes of rank-$2$ stable graded parabolic Higgs bundles $(E,\theta)$ of degree zero on $(\bP^1_k,D_k)/k$ with all parabolic weights being zero at $\{0,1,\lambda\}$ and with all parabolic weights being $1/2$ at $\infty$ (\autoref{notation_HIG});
\item $\PHigh(W(k))$ is the set of periodic Higgs bundles contained in $\High(W(k))$ (\autoref{notation_PHIG});
\item $[\PHDFhf(W(k))]$ is the set of periodic Higgs-de Rham flows with Higgs terms contained in $\PHigh(W(k))$ modulo an equivalence (\autoref{notation_PHDF_classes});
\item $[\MFh(W(k'))_{\bZ_{p^f}}]$ is the set of Fontaine-Faltings modules with $\bZ_{p^f}$-endomorphism structures, such that all eigen components of the corresponding graded Higgs bundles are contained in $\High(W(k))$, modulo an equivalence defined in \autoref{def_diffByConstantFFM}.
\item $[\MFh(W(k'_2))_{\bZ_{p^f}}^{\rm cy}]$ is the subset of $[\MFh(W(k'_2))_{\bZ_{p^f}}]$ coming from Fontaine-Faltings module with cyclotomic determinant.
\end{itemize}

Throughout this subsection, we will free use the terminology and notation for parabolic structure summarized in \cite{YaZu23c}. One can also find definitions in \cite{IySi07} and \cite{KrSh20}.

\subsection{\bf Parabolic de Rham bundles and parabolic Higgs bundles} \label{sec_main_para} In this section, we recall some parabolic objects from \cite{YaZu23c}. For the purposes of our application, we only focus on the following special spaces $(Y,D_Y)/S$:
\begin{enumerate}[$(1).$]
\item $S=\Spec(K)$, where $K$ is a field of characteristic $0$;
\item $S=\Spec(W_m(k))$, where $W_m(k)$ is a ring of truncated Witt vectors with coefficients in a finite field $k$;
\item $S=\Spec(\mO_K)$, where $K$ is an unramified $p$-adic number field.
\item $S=\Spf(\mO_K)$, where $K$ is an unramified $p$-adic number field.
\end{enumerate}
For a smooth curve $Y$ over $S$ (or a smooth formal curve over $S$ if $S$ is a formal scheme), we define the reduced divisor $D_Y$ by $n$ $S$-sections $x_i\colon S\rightarrow Y$, $i=1,\cdots,n$, that do not intersect with each other. We denote by $U_Y := Y - D_Y$ and by $j_Y$ the open immersion $j_Y\colon U_Y\rightarrow Y$. The irreducible components of $D_Y$ are denoted by $D_{Y,i}$, $i=1,2,\cdots,n$, and we have $D_Y = \bigcup_{i=1}^n D_{Y,i}$. We set $\Omega_{Y/S}^1$ to be the sheaf of relative $1$-forms and $\Omega_{Y/S}^1(\log D_Y)$ to be the sheaf of relative $1$-forms with logarithmic poles along $D_Y$.

By the smoothness of $Y$ over $S$, both of these sheaves are line bundles over $Y$. The following definitions are inspired by \cite{IySi07}.

\begin{definition}[{\cite[Definition 1.21]{YaZu23c}}]
A \emph{parabolic de Rham bundle} $(V,\nabla)=\{(V_\alpha,\nabla_\alpha)\}$ over $(Y,D_Y)/S$ is parabolic vector bundle $V=\{V_\alpha\}$ together with integrable connections $\nabla_\alpha$ having logarithmic pole along $D_Y$ such that the inclusions $V_\alpha\hookrightarrow V_\beta$ preserves the connections.
We call $\nabla:=\{\nabla_\alpha\}$ a \emph{parabolic connection} on the parabolic vector bundle $V$.
\end{definition}

Recall that a logarithmic $p$-connection on a vector bundle $V$ over $(Y,D_Y)/S$ is an $\mO_S$-linear mapping
\[\nabla \colon V \rightarrow V\otimes \Omega_{Y/S}(\log D_Y)\]
satisfying, for any local section $s \in \mO_Y$ and any local section $v\in V$
\[\nabla(sv) = p v\otimes \rmd s + s\nabla(v)\]
We note that the multiplication of a connection with $p$ is always a $p$-connection and if $p$ is invert in $\mO_S(S)$, then all $p$-connections are coming from this way.
Similarly, one can defines \emph{parabolic $p$-connections} on parabolic vector bundles.

\begin{definition}[{\cite[Definition 1.29]{YaZu23c}}]
A \emph{parabolic Higgs bundle} $(E,\theta)=\{(E_\alpha,\theta_\alpha)\}$ over $(Y,D_Y)$ is
\begin{itemize}
\item a parabolic vector bundle $E=\{E_\alpha\}$, together with
\item integrable Higgs fields $\theta_\alpha$ having logarithmic pole along $D_Y$
\end{itemize}
such that the inclusions $E_\alpha\hookrightarrow E_\beta$ preserves the Higgs fields.

A parabolic Higgs bundle $(E,\theta)$ is called \emph{graded}, if there is a grading structure $Gr$ on $E$ satisfying decomposition of the underlying parabolic vector bundle $E$
\[\theta(\Gr^\ell E) \subset \Gr^{\ell -1}E\otimes_{\mO_X} \Omega^1_{X/S}(\log D).\]
\end{definition}

The definitions of Fontaine-Faltings modules and Higgs-de Rham flows are extended to parabolic versionsI, in \cite[Definition 2.27, Definition 2.29]{YaZu23c}. For more basic properties of parabolic objects see \cite{IySi07,KrSh20,YaZu23c}.

We recall some classifying result of parabolic objects of small rank from \cite{YaZu23c}. In the rest of this subsection, we take $X=\bP^1_S$ as the projective line over $S$ and take $D=D_S\subset \bP_S^1$ as the divisor given by $4$ $S$-points $\{0,1,\infty,\lambda\}$. Denote by $D_i$ the reduce and irreducible divisor given by the point $x$ for any $x\in\{0,1,\infty,\lambda\}$.

\begin{notation} \label{notation_MIC}
Denote by \emph{$\MdRh(S)$} the set of all isomorphic classes of rank-$2$ stable parabolic de Rham bundles $(V,\nabla)$ of degree zero on $(\bP^1_S,D_S)/S$ with all parabolic weights being zero at $\{0,1,\lambda\}$ and with all parabolic weights being $1/2$ at $\infty$.
\end{notation}

\begin{proposition}[{\cite[Proposition 1.36]{YaZu23c}}] \label{thm_ClassfyR2PdE}
Let $(V,\nabla)$ be a parabolic de Rham bundle in $\MdRh(S)$. Then
\begin{enumerate}[$(1).$]
\item
the parabolic de Rham bundle $(V,\nabla)$ has the form
\[(\mL\oplus \mL^{-1},\nabla).\]
where $\mL=\mO(\frac12(\infty))$.
\item
if we take the \emph{parabolic Hodge line bundle} as $\mL$, then the associated graded parabolic Higgs field is nonzero and is of form
\[\theta\colon \mL\rightarrow \mL^{-1}\otimes \Omega^1_{X/S}(\log D).\]
In particular, the graded parabolic Higgs bundle $(\mL\oplus \mL^{-1},\theta)$ is stable and is of degree zero.
\end{enumerate}
\end{proposition}

\begin{notation} \label{notation_HIG}
Denote by \emph{$\High(S)$\label{High}} the set of all isomorphic classes of rank-$2$ stable graded parabolic Higgs bundles $(E,\theta)$ of degree zero on $(\bP^1_S,D_S)/S$ with all parabolic weights being zero at $\{0,1,\lambda\}$ and with all parabolic weights being $1/2$ at $\infty$.
\end{notation}

\begin{proposition}[{\cite[Proposition 1.37]{YaZu23c}}] \label{thm_ClassfyR2PHiggs}
Let $(E,\theta)$ be a graded parabolic Higgs bundle in $\High(S)$. Then
\[E = \mL\oplus \mL^{-1},\]
where $\mL=\mO(\frac12(\infty))$ and the parabolic Higgs field is nonzero and is of form
\[\theta\colon \mL\rightarrow \mL^{-1}\otimes \Omega^1_{X/S}(\log D).\]
\end{proposition}

As a consequence, we have the following description of the Higgs bundle via the zero of the corresponding Higgs field.
\begin{corollary} \label{thm_Higgs0ProjLine} Any parabolic Higgs bundle $(E,\theta)\in \High(S)$ is uniquely determined by $(\theta)_0 \in \bP^1_S(S)$, the zero of the Higgs field $\theta$. One has a natural bijection induced by taking zeros
\[\High(S) \xrightarrow[(E,\theta)\mapsto (\theta)_0]{1:1} \bP^1_S(S).\]
\end{corollary}

\subsection{Parabolic Higgs-de Rham flows over projective line}

Let $k$ be a finite field with cardinality $q=p^{h}$. Let $\lambda \in W(k)$ such that $\lambda\pmod{p}\neq 0,1\in k$. Denote the formal projective line over $W(k)$ and a divisor on it
\[\mP^1_{W(k)}:=\bP^1_{W(k)}\times_{\Spec(W(k))}\Spf(W(k)) \text{ \ and \ } \mD_{W(k)}=\{0,1,\lambda,\infty\}\subset\mP^1_{W(k)}.\]
By modulo $p^n$, one gets logarithmic pair $(\mP^1_{W_n(k)},\mD_{W_n(k)})/W_n(k)$. In this section, we will study some periodic Higgs-de Rham flows over $(\mP^1_{W_n(k)},\mD_{W_n(k)})/W_n(k)$ and over $(\mP^1_{W(k)},\mD_{W_n(k)})/W(k)$. To reduce the repetition of writing, we sometimes use $W_\infty(k)$ to stand for $W(k)$.

We first recall that, for any $n\in\{\infty,1,2,\cdots\}$, \[\High(W_n(k))\coloneqq\High(\Spf(W_n(k)))\]
is the set of all isomorphic classes of rank-$2$ stable graded parabolic Higgs bundles $(E,\theta)$ of degree zero on $(\mP^1_{W_n(k)},\mD_{W_n(k)})/W_n(k)$ with all parabolic weights being zero at $\{0,1,\lambda\}$ and with all parabolic weights being $1/2$ at $\infty$.

\subsubsection{Parabolic Higg-de Rham flows initialed with given parabolic Higgs bundles in $\High(k)$}

We first construct a parabolic Higgs-de Rham flow initialed with a parabolic Higgs bundle in $\High(k)$.

\begin{lemma} \label{thm_Higgs2HDF_k}
Let $(E,\theta) \in \High(k)$. Then there is a unique (up to an isomorphism) parabolic Higgs-de Rham flow
\[\Flow = \left\{
(E,\theta)_{0},
(V,\nabla,\Fil)_{0},
(E,\theta)_{1},
(V,\nabla,\Fil)_{1},
\cdots\right\},\]
initialed with $(E,\theta)_0=(E,\theta)$, such that Higgs terms $(E,\theta)_{i}$ are contained in $\High(k)$ for all $i\geq 0$.
Moreover $(V,\nabla,\Fil)_{i}\in\MdRh(k)$ for all $i\geq 0$.
\end{lemma}

\begin{proof}
By \autoref{thm_ClassfyR2PHiggs}, $(E,\theta)$ has the form
\[\theta\colon \mL\rightarrow \mL^{-1}\otimes \Omega^1_{\bP^1_k/k}(\log D_k)\]
with a single zero $(\theta)_0\in \bP^1_k(k)$. Then taking the inverse Cartier, one gets a parabolic de Rham bundle $(V,\nabla)_0$, which is stable and of degree $0$. Hence it is contained in $\MdRh(k)$.

To make the graded Higgs bundle contained in $\High(\barFp)$ the Hodge filtration must be given by $\mL$, see \autoref{thm_ClassfyR2PdE}. Taking the grading of $(V,\nabla)_0$ with respect to the Hodge filtration, one gets a graded parabolic Higgs bundle contained in $\High(k)$;

From above, the first filtered de Rham term and the second Higgs term both exist and are uniquely determined by the first Higgs term.

Repeating the above procedure, one then get the unique parabolic Higgs-de Rham flow initialed with $(E,\theta)$:

\begin{equation*}
\xymatrix@C=2mm@R=5mm{
& {\scriptstyle ( V,\nabla,\Fil)_0} \ar[rd]|{\text{Gr}}
&
& {\scriptstyle ( V,\nabla,\Fil)_1} \ar[rd]|{\text{Gr}}
&
& {\scriptstyle ( V,\nabla,\Fil)_2} \ar[rd]|{\text{Gr}}
&
& {\scriptstyle \cdots}
\\
{\scriptstyle ( E,\theta)_0} \ar[ru]|{\mC^{-1}}
&
& {\scriptstyle ( E,\theta)_1} \ar[ru]|{\mC^{-1}}
&
& {\scriptstyle ( E,\theta)_2} \ar[ru]|{\mC^{-1}}
&
& {\scriptstyle ( E,\theta)_3} \ar[ru]|{\mC^{-1}}
&
& {\scriptstyle \quad \cdots}
}
\end{equation*}
\end{proof}

\begin{notation} \label{notation_PHIG}
Denote by $\PHighf(k)$ the set of Higgs bundle $(E,\theta)_0$ which is $f$-periodic. I.e. there is an isomorphism between $(E,\theta)_0$ and the $f$-th Higgs term $(E,\theta)_f$. And denote
\[\PHigh(k) = \bigcup_{f} \PHighf(k).\]
\end{notation}

\begin{proposition} \label{mthm_PHIGf2PHIG}
If $(\#k+1)!\mid f$, then
\[\PHigh(k) = \PHighf(k)\]
\end{proposition}
\begin{proof}
By \autoref{thm_Higgs0ProjLine},
we know $\#\High(k)$ has cardinality $\#k+1$. Thus the periodicity of any periodic Higgs bundle in $\High(k)$ is smaller than or equal to $\#k+1$.
\end{proof}

\begin{remark}
Although the Higgs-de Rham flow exists and must be preperiodic due to the finiteness of $\High(k)$, there is some freedom in the choice of the position of repeating part and the period mapping. For example, if $(E,\theta)_e\cong (E,\theta)_{e+f}$, then we always have
\[(E,\theta)_i\cong (E,\theta)_{i+kf},\qquad \text{for any $i\geq e$ and any $k>0$}.\]
So the cycle nodes can be chosen at $i,i+kf$, and the period mapping can be chosen to be any isomorphism between $(E,\theta)_i$ and $(E,\theta)_{i+kf}$.
\end{remark}

There is a theoretical way to find periodic Higgs bundles. Under the natural bijection $\High(k) \simeq \bP^1_k(k)$ in \autoref{thm_Higgs0ProjLine}, Sun-Yang-Zuo have shown that the self-map
\[\phi\coloneqq \text{Gr}\circ
\mC^{-1}\colon \High(k)\rightarrow \High(k)\]
is induced by an endomorphism of $\bP_{k_0}^1$ give by a rational function of form $\phi(z)=\psi(z^p)$, where $\psi$ is a rational function of degree $p$. To find periodic Higgs-de Rham flow one only need to find periodic points of the map $\phi$. In particular, we obtain
\begin{proposition}
The number of $f$-periodic Higgs bundles in $\High(\bark)$ is $p^{2f}+1$.
\end{proposition}
We take then the elliptic curve $C_{\lambda}$ over the field $k_0$ defined by the Weierstrass equation $y^2=z(z-1)(z- \lambda)$. Modulo involution on the elliptic curve induces natural double cover
\[\pi\colon C_\lambda\to \bP^1_{k_0}\] ramified on $\{0,1,\infty,\lambda\}$ and $\infty$ as the origin for the group law. Sun-Yang-Zuo have asked the following conjecture.
\begin{conjecture}
The self-map $\phi$ comes from multiplication map by $p$ on the associated elliptic curve $C_\lambda$ over $k_0$. In other words, the following diagram commutes
\[\xymatrix{
& C_\lambda \ar[d]_{\pi} \ar[r]^{[p]} & C_\lambda\ar[d]^\pi & \\
M_{Higg,\lambda}^{gr}\ar@/_12pt/[rrr]_{\phi} \ar@{=}[r] & \bP^1_{k_0} \ar[r]^{\phi} & \bP^1_{k_0} \ar@{=}[r] & M_{Higg,\lambda}^{gr} \\
}\]
\end{conjecture}

The conjecture implies two things:
\begin{enumerate}[$(1).$]
\item a Higgs bundle $(E,\theta)$ is $f$-periodic under the map $\phi$ if and only if the two points in $\pi^{-1}(\theta)_0$ are both torsion in $C_\lambda$ and of order $p^f\pm1$.
\item for a prime $p>2$ and assume $C_\lambda$ is supersingular then $\phi_\lambda(z)=z^{p^2}$. Hence, any Higgs bundle $(E,\theta)\in {\High}(\bark)$ is periodic.
\end{enumerate}

The Conjecture has been checked by Sun-Yang-Zuo for $p<50$. Very recently it has been proved by
Lin-Sheng-Wang \cite{LSW22} and becomes a theorem.
\begin{theorem} [Lin-Sheng-Wang]
\autoref{conj:SYZ} holds true.
\end{theorem}

\begin{corollary} \label{mthm_HIG2PHIG}
If $C_\lambda$ is supersingular, then any Higgs bundle $(E,\theta)\in {\High}(\bark)$ is periodic.
\end{corollary}

\subsubsection{Parabolic Higgs-de Rham flows initialed with given parabolic Higgs bundles in $\High(W_n(k))$}

In this subsubsection, we take $n\in\{\infty,1,2,\cdots\}$. We show that the is at most one parabolic Higgs-de Rham flow initialed with a given parabolic Higgs bundle in $\High(W_n(k))$.

\begin{definition} \label{thm_Higgs2HDF_Wn}
Let $(E,\theta) \in \High(W_n(k))$. A parabolic Higgs-de Rham flow
\[\Flow = \left\{
(\overline{V},\overline{\nabla},\overline{\Fil})_{-1},
(E,\theta)_{0},
(V,\nabla,\Fil)_{0},
(E,\theta)_{1},
(V,\nabla,\Fil)_{1},
\cdots\right\},\]
over $\BBn$ is called initialed\footnote{We note that when $n=1$, the $-1$-th term $(V,\nabla,\Fil)_{-1}$ is vacuous and $(E,\theta)_{0}$ is indeed the leading term. This is why we call $0$-th term the initial one for general $n$.} with $(E,\theta)$, if there is an isomorphism between $(E,\theta)$ and $(E,\theta)_0$.
\end{definition}

Due to the uniqueness of the Hodge filtration in \autoref{thm_ClassfyR2PdE}, we may repeat the proof for \autoref{thm_Higgs2HDF_k} and get following result.
\begin{lemma} \label{thm_Higgs2HDF_W}
Let $(\overline{V},\overline{\nabla},\overline{\Fil})_{-1} \in \MdRh(W_{n-1}(k))$ and $(E,\theta)_{0}\in \High(W_n(k))$ with $\Gr(\overline{V},\overline{\nabla},\overline{\Fil}) = (E,\theta)_{0}\pmod{p^{n-1}}$. Then there exists a unique (up to an isomorphism) parabolic Higgs-de Rham flow
\[\Flow = \left\{
(\overline{V},\overline{\nabla},\overline{\Fil})_{-1},
(E,\theta)_{0},
(V,\nabla,\Fil)_{0},
(E,\theta)_{1},
(V,\nabla,\Fil)_{1},
\cdots\right\},\]
initialed with $(E,\theta)_0$, and the $-1$-th de Rham term being $(\overline{V},\overline{\nabla},\overline{\Fil})_{-1}$ such that Higgs terms $(E,\theta)_{i}$ are contained in $\High(W_n(k))$ for all $i\geq 0$.
Moreover $(V,\nabla,\Fil)_{i}\in\MdRh(W_n(k))$ for all $i\geq 0$.
\end{lemma}

\begin{lemma}
Up to an isomorphism, there is at most one periodic parabolic Higgs-de Rham flow initialed with $(E,\theta)\in \High(W_n(k))$.
\end{lemma}

\begin{proof}
Suppose $(\Flow,\psi)$ and $(\Flow',\psi')$ be two $f$-periodic flows initialed with $(E,\theta)$, denote by $(\Flow^{(n)},\psi^{(n)})$ and $(\Flow'^{(n)},\psi'^{(n)})$ their modulo $p^n$ reductions. By the uniqueness in \autoref{thm_Higgs2HDF_k}, we may identify $\Flow^{(1)}$ and $\Flow'^{(1)}$. By shifting the isomorphism on the $f-1$-th de Rham terms via the periodic maps, one gets an isomorphism between the $-1$-th de Rham terms in the flow $\Flow^{(2)}$ and $\Flow'^{(2)}$, By uniqueness in \autoref{thm_Higgs2HDF_W}, we may identify $\Flow^{(2)}$ and $\Flow'^{(2)}$. Inductively, one can identify $\Flow^{(n)}$ and $\Flow'^{(n)}$ for all $n$.
\end{proof}

\subsubsection{An equivalence on the set of isomorphic classes of periodic Higgs-de Rham flows.}

Let $(\Flow,\psi)$ be an $f$-periodic Higgs-de Rham flow with
\[\Flow = \left\{
(\overline{V},\overline{\nabla},\overline{\Fil})_{-1},
(E,\theta)_{0},(V,\nabla,\Fil)_{0},(E,\theta)_{1},(V,\nabla,\Fil)_{1},\cdots\right\},\]
and $\psi\colon \Flow[f]\cong \Flow$. By shifting the index, one gets isomorphisms of flows
\[\psi[k]\colon \Flow[f+k] \rightarrow \Flow[k],\quad \text{for all $k\geq0$}.\]
For any $k\geq1$, there is a natural isomorphism
\[\psi^k:=\psi[(k-1)f]\circ\cdots\circ\psi[f]\circ\psi\colon \Flow[kf]\rightarrow\Flow\]
Thus one gets a periodic flow $(\Flow,\psi^k)$ for any $k\geq1$.

\begin{definition} \label{def_diffByConstant}
Let $(\Flow_1,\psi_1)$ and $(\Flow_2,\psi_2)$ be two $f$-periodic Higgs-de Rham flows over $\BBn$. We call they are \emph{differed by a constant}, if
\begin{itemize}
\item there exists an isomorphism of the underlying flows, and
\item once we identify the flows via the isomorphism, there exists a unit $u\in W_n(k)^\times$ such that
\[\psi_1 = u\cdot \psi_2.\]
\end{itemize}
\end{definition}

\begin{notation} \label{notation_PHDF}
Let \emph{$\PHDFh(W_n(k))$} be the set of isomorphic classes of periodic Higgs-de Rham flows (Higgs-de Rham flows with periodic mappings) with all Higgs terms are contained in $\High(W_n(k))$ and all de Rham terms are contained in $\MdRh(W_n(k))$. Denote by \emph{$\PHDFhf(W(k))$} the subset consists of $f$-periodic flows.
Denote
\[\PHDFh(W(k)) := \varprojlim_n \PHDFh(W_n(k)) \quad \text{and} \quad \PHDFhf(W(k)) := \varprojlim_n \PHDFhf(W_n(k)).\]
\end{notation}

\begin{lemma}
Two periodic Higgs-de Rham flows in $\PHDFhf(W_n(k))$ are differed by a constant if and only if they have isomorphic initial terms.
\end{lemma}
\begin{proof}
The ``only if'' part is trivial. Now, we consider the ``if'' part and assume the two flow have isomorphic initial terms. By \autoref{thm_Higgs2HDF_k}, there is an isomorphism between the underlying flows. We may identify this two flows. Then there are two periodic mappings on this common flow. We need to show this two mappings are differed by a unit in the sense \autoref{def_diffByConstant}. This follows that the modulo $p$ reduction of all Higgs terms appeared in the flow are stable.
\end{proof}

As a consequence, differed by a constant is an equivalent relations on $\PHDFh(W_n(k))$.
\begin{notation} \label{notation_PHDF_classes}
Denote by \emph{$[\PHDFh(W_n(k))]$} the set of all equivalent classes. Similarly we denote the notation \emph{$[\PHDFhf(W_n(k))]$}, \emph{$[\PHDFh(W(k))]$}, \emph{$[\PHDFhf(W(k))]$} and \emph{$[\PHDFh(W(k))]$}. Then we have the following result.
\end{notation}

\begin{corollary} \label{mthm_PHDFf2PHIGf}
Taking initial terms induces bijection
\[ [\PHDFhf(W_n(k))] \xrightarrow{1:1} \PHighf(W_n(k)).\]
Taking inverse limits, one gets an bijection
\[ [\PHDFhf(W(k))] \xrightarrow{1:1} \PHighf(W(k)).\]
\end{corollary}

\subsection{Parabolic Fontaine-Faltings modules over projective line}
\subsubsection{Fontaine-Faltings modules associated to periodic flows in $\PHDFh(W(k))$}\label{sec_perHiggs2FFM} \label{sec_PHDF2MF}

In this subsubsection, we give the construction of Lan-Sheng-Zuo equivalent functor in parabolic setting. The main result in this subsection is an bijection
\[\PHDFhf(W(k')) \xrightarrow{1:1} \MFh(W(k'))_{\bZ_{p^f}}\]
between a set of periodic Higgs-de Rham flow and a set of Fontaine-Faltings modules established in \autoref{thm_parabolicLSZ}.

Let $(\Flow,\psi)$ be an $f$-periodic flow in $\PHDFhf(W(k))$ with
\[\Flow = \left\{
(\overline{V},\overline{\nabla},\overline{\Fil})_{-1},
(E,\theta)_{0},(V,\nabla,\Fil)_{0},(E,\theta)_{1},(V,\nabla,\Fil)_{1},\cdots\right\}\]
and $\psi=\{\overline{\varphi}_{-1},\psi_{0},\varphi_{0},\psi_{1},\varphi_{1},\cdots\}$. By adding up all filtered de Rham terms appeared in the repeating part, one gets a parabolic de Rham bundle of rank $2f$
\begin{equation}\label{equ_dRtermsSum}
(V,\nabla,\Fil):= ( V,\nabla,\Fil)_0 \oplus ( V,\nabla,\Fil)_1 \oplus \cdots \oplus ( V,\nabla,\Fil)_{f-1}.
\end{equation}
We defined an isomorphism $\varphi\colon C^{-1} \circ \overline{Gr} ( V,\nabla,\Fil) \to ( V,\nabla,\Fil)$ by
\begin{equation*} \footnotesize
\xymatrix@C=2mm@R=1cm{
C^{-1} \circ \overline{Gr} ( V,\nabla,\Fil) \ar[d]^{\varphi} \ar@{}[r]|{=}
& C^{-1} \circ \overline{Gr} ( V,\nabla,\Fil)_0 \ar[dr]^{\id} \ar@{}[r]|-{\oplus}
& C^{-1} \circ \overline{Gr} ( V,\nabla,\Fil)_1 \ar[dr]^(0.7){\id} \ar@{}[r]|-{\oplus}
& \cdots \ar[dr]^{\id} \ar@{}[r]|-{\oplus}
& C^{-1} \circ \overline{Gr} ( V,\nabla,\Fil)_{f-1} \ar[dlll]^{\varphi_0}\\
( V,\nabla,\Fil) \ar@{}[r]|{=}
&( V,\nabla,\Fil)_0 \ar@{}[r]|-{\oplus}
&( V,\nabla,\Fil)_1 \ar@{}[r]|-{\oplus}
&\cdots \ar@{}[r]|-{\oplus}
&( V,\nabla,\Fil)_{f-1} \\
}
\end{equation*}
Then tuple
\begin{equation} \label{eq_HDF2FFM}
(V,\nabla,\Fil,\varphi)
\end{equation}
forms a parabolic Fontaine-Faltings module.

\begin{remark} In order to construct a correspondence between periodic Higgs-de Rham flows and Fontaine-Faltings modules. We need overcome one obstacle. By shifting the flow $i$-times, we get another $f$-periodic flow $(\Flow[i],\psi[i])$. From above construction, we can see that these two flows corresponding to isomorphic Fontaine-Faltings module. Hence the wanted correspondence is not injective.
In order to get an injective one, one needs to add endomorphism structures on such a Fontaine-Faltings module, such that different periodic flows corresponds to different Fontaine-Faltings modules with endomorphism structures.
\end{remark}

In the following, we construct natural $\bZ_{p^f}$-endomorphism structures
\[\iota_j \colon \bZ_{p^f}\rightarrow \End \Big((V,\nabla,\Fil,\varphi)\Big),\quad j\in \bZ.\]
which can be used to distinguish direct summands $(V_i,\nabla_i,\Fil_i)$ of the underlying de Rham bundle $(V,\nabla,\Fil)$.
\begin{lemma} \label{thm_ConstEndStructure} Suppose $\bF_{p^f}\subseteq k$.
For any $j\in\bZ$, any $a\in\bZ_{p^f}$ and any local section $v_i\in V_i$, set
\[\iota_j(a)(v_i):=\sigma^{i+j}(a)\cdot v_i.\]
Then $\iota_j$ is an $\bZ_{p^f}$-endomorphism structure on $(V,\nabla,\Fil,\varphi)$.
\end{lemma}
\begin{proof}
Since $\nabla$ is $W(k)$-linear and $\Fil$ consists of sub-$W(k)$-modules, $\iota_j$ indeed gives an $\bZ_{p^f}$-endomorphism on $(V,\nabla,\Fil)$. Next, one only need to show $\iota_j$ preserves the Frobenius structure $\varphi$ in the Fontaine-Faltings module. In other words, we need to check the following diagram commutes for any $a\in\bZ_{p^f}$
\begin{equation*}
\xymatrix{
F^*\widetilde{V} \ar[r]^{\varphi} \ar[d]_{\id\otimes \iota_j(a)} & V \ar[d]^{\iota_j(a)}\\
F^*\widetilde{V} \ar[r]^{\varphi} & V\\
}
\end{equation*}
For any local section $v_{i\ell}\in\Fil^\ell V_i$, we have $\varphi(1\otimes[v_{i\ell}])\in V_{i+1}$. Thus
\[\iota_j(a)\circ\varphi(1\otimes[v_{i\ell}]) = \sigma^{i+1+j}(a)\cdot\varphi(1\otimes[v_{i\ell}]).\]
On the other hand, one has
\[\varphi\circ (\id\otimes \iota_j(a)) (1\otimes[v_{i\ell}]) = \varphi(1\otimes\sigma^{i+j}(a)\cdot [v_{i\ell}])=\sigma^{i+j+1}(a)\cdot \varphi(1\otimes [v_{i\ell}]),\]
where the last equality follows the $\sigma$-semilinearity of $\varphi$. Thus the Lemma follows.
\end{proof}

\begin{definition} Suppose $\bF_{p^f}\subseteq k$.
\begin{enumerate}[$(1).$]
\item Let $(V,\nabla,\Fil,\iota)$ be a filtered parabolic de Rham bundle $V$ with an $\bF_{p^f}$-endomorphism structure $\iota$ over $(\bP^1_{W(k)},D_{W(k)})$. Then the filtration can be restricted on $V^{\iota=\sigma^i}$.

We call the sub parabolic de Rham bundle \[(V^{\iota=\id},\nabla\mid_{V^{\iota=\id}},\Fil_{V^{\iota=\id}}) =: (V,\nabla,\Fil)^{\iota=\id}\]
\emph{the $i$-th eigen component of $(V,\nabla,\Fil,\iota)$}. If $i=0$, then we call it \emph{the identity component of $(V,\nabla,\Fil,\iota)$.}
\item Let $(E,\theta,\iota)$ be a parabolic Higgs bundle $V$ with an $\bF_{p^f}$-endomorphism structure $\iota$ over $(\bP^1_{W(k)},D_{W(k)})$. Then the Higgs field can be restricted on $E^{\iota=\sigma^i}$.
We call the sub parabolic Higgs bundle \[(E^{\iota=\id},\theta\mid_{E^{\iota=\id}}) =: (E,\theta)^{\iota=\id}\]
\emph{the $i$-th eigen component of $(E,\theta,\iota)$}. If $i=0$, then we call it \emph{the identity component of $(E,\theta,\iota)$.}
\end{enumerate}

\end{definition}

By direct calculation, one has following result.
\begin{lemma} For any $j\in\bZ$, one has
\[\iota_{j+1} = \iota_j\circ\sigma \quad\text{and}\quad \iota_j=\iota_{j+f}.\]
For any $i=0,\cdots,f-1$, the direct summand $V_i$ is the identity component of $(V,\iota_j)$ if and only if $f\mid i+j$. In particular, by taking the identity components from different endomorphism structures, we can pick out different direct summands of the filtered de Rham bundle.
\end{lemma}

Taking grading of the underlying filtered de Rham bundle, one gets endomorphism structures on the graded Higgs bundle, still denoted by $\iota_j$ by abusing notion,
\[\iota_j\colon \bZ_{p^f} \rightarrow \End(E,\theta),\]
where $(E,\theta) = (E,\theta)_0 \oplus (E,\theta)_1 \oplus \cdots \oplus (E,\theta)_{f-1}$. By direct calculation, one has following result.
\begin{lemma} For any $j\in \bZ$, any $i\in\{0,1,\cdots,f-1\}$ and any local section $v_i\in E_i$
\[\iota_j(a)(v_i):=\sigma^{i+j-1}(a)\cdot v_i.\]
In particular, the Higgs bundle $(E_0,\theta_0)$ is the identity component of $(E,\theta,\iota_1)$.
\end{lemma}

\begin{notation} \label{notation_MF}
For any integer $f$ such that $\bF_{p^f}\subseteq k$, denote by \emph{$\MFh(W(k))_{\bZ_{p^f}}$} the set of all isomorphic classes of Fontaine-Faltings module with an $\bZ_{p^f}$-endomorphism structure such that all eigen components\footnote{The condition $\bF_{p^f}\subseteq k$ ensure that we can take eigen components.} of the corresponding filtered de Rham bundles are contained $\MdRh(W(k))$ and all eigen components of the corresponding graded Higgs bundles are contained in $\High(W(k))$.
\end{notation}

By the above construction of the parabolic version of Lan-Sheng-Zuo's equivalent functor, one gets following bijection, whose proof is the same as the original version of Lan-Sheng-Zuo.
\begin{proposition}\label{thm_parabolicLSZ} Let $k'$ be a finite extension of $k$ containing $\bF_{p^f}$. Then one has a bijection
\[\PHDFhf(W(k')) \xrightarrow{1:1} \MFh(W(k'))_{\bZ_{p^f}}\]
sending an $f$ periodic flow $(\Flow,\psi)$ to $(V,\nabla,\Fil,\varphi,\iota_1)$, where $(V,\nabla,\Fil,\varphi)$ is given in \eqref{eq_HDF2FFM} and $\iota_1$ is given in \autoref{thm_ConstEndStructure}.
\end{proposition}

By base change from $k$ to $k'$, one gets the natural embedding
\[\PHDFhf(W(k)) \hookrightarrow \PHDFhf(W(k')).\]
\begin{corollary}
Let $k'$ be a finite extension of $k$ containing $\bF_{p^f}$. Then the restriction of bijection induces an injection
\begin{equation}\label{eq_PHDFHfW2MFhf}
\PHDFhf(W(k)) \hookrightarrow \MFh(W(k'))_{\bZ_{p^f}}.
\end{equation}
\end{corollary}

\subsubsection{An equivalence relation on the set of isomorphic classes of Fontaine-Faltings modules in $\MFh$.} \label{subsec_HDF2FFM}

Let $k'$ be a finite field extension of $k$ containing $\bF_{p^f}$. We recall the definition of \emph{constant Fontaine-Faltings module} from \cite[Definition 2.6, Section 2.1.8]{YaZu23c}.

\begin{definition} \label{def_diffByConstantFFM}
Let $M$ and $M' \in \MFh(W(k'))_{\bZ_{p^f}}$. We call they are \emph{differed by a constant}, if there exists a constant Fontaine-Faltings module $M^\circ \in \MF_{[0,0]}^{\varphi}(W(k'))_{\bZ_{p^{f}}}$ of rank $1$ such that
\[M' = M \otimes M^\circ.\]
Clearly, differed by a constant is an equivalent relation on the set $\MFh(W(k'))_{\bZ_{p^f}}$. Denote by \emph{$[\MFh(W(k'))_{\bZ_{p^f}}]$} the set of all equivalent classes.
\end{definition}

\begin{remark} \label{rmk_FFotimesConstFF}
Suppose $M=(V,\nabla,\Fil,\varphi,\iota)$ and $M^\circ=(V^\circ,\varphi^\circ,\iota^\circ)$. Denote by $(V,\nabla,\Fil)_i$ the $i$-th eigen component of $(V,\nabla,\Fil,\iota)$ and denote by $V^\circ_i\simeq W(k')$ the $i$-th eigen component of $V^\circ$. Then
One the direct summand $(V,\nabla,\Fil)_i\otimes_{W(k')} V^\circ_j\simeq (V,\nabla,\Fil)_i$ of $(V,\nabla,\Fil)\otimes_{W(k)} V^\circ$, the action of $\iota$ and $\iota'$ are coincide if and only if $i=j$. Thus we can see that
the underlying filtered de Rham bundle of $(V,\nabla,\Fil)\otimes_{W(k)} V^\circ$ is
\[\bigoplus_{i=0}^{f-1} (V,\nabla,\Fil)_i\otimes_{W(k')} V^\circ_i \]
which is isomorphic to $(V,\nabla,\Fil)$ once we fixed a basis $e_i$ for each $V^\circ_i$, the map is given by $v_i\otimes e_i \mapsto v_i$.
We also decompose the Frobenius structure $\varphi$
\begin{equation*} \footnotesize
\xymatrix@C=2mm@R=1cm{
C^{-1} \circ \overline{Gr} ( V,\nabla,\Fil) \ar[d]^{\varphi} \ar@{}[r]|{=}
& C^{-1} \circ \overline{Gr} ( V,\nabla,\Fil)_0 \ar[dr]^{\varphi_0} \ar@{}[r]|-{\oplus}
& C^{-1} \circ \overline{Gr} ( V,\nabla,\Fil)_1 \ar[dr]^(0.7){\varphi_1} \ar@{}[r]|-{\oplus}
& \cdots \ar[dr]^{\varphi_{f-2}} \ar@{}[r]|-{\oplus}
& C^{-1} \circ \overline{Gr} ( V,\nabla,\Fil)_{f-1} \ar[dlll]^{\varphi_{f-1}}\\
( V,\nabla,\Fil) \ar@{}[r]|{=}
&( V,\nabla,\Fil)_0 \ar@{}[r]|-{\oplus}
&( V,\nabla,\Fil)_1 \ar@{}[r]|-{\oplus}
&\cdots \ar@{}[r]|-{\oplus}
&( V,\nabla,\Fil)_{f-1} \\
}
\end{equation*}
Suppose $\varphi_i^\circ(e_i) = a_ie_{i+1}$, then
we see that for any $v_i\in \Fil^\ell V_i$
\[\varphi_{\rm tot}(1\otimes_{\Phi} [v_i]\otimes e_i) = \varphi_i(1\otimes_{\Phi} [v_i]) \otimes \varphi_i^\circ(e_i) = a_i \cdot \varphi(1\otimes_{\Phi} [v_i]) \otimes e_{i+1}.\]
If we identify $(V,\nabla,\Fil)_i\otimes_{W(k')} V^\circ_i$ with $(V,\nabla,\Fil)_i$ by sending $v_i\otimes e_i$ to $v_i$, then the Frobenius structure on $M'$ can be describe as
\begin{equation*} \footnotesize
\xymatrix@C=2mm@R=1cm{
C^{-1} \circ \overline{Gr} ( V,\nabla,\Fil) \ar[d]^{\varphi'} \ar@{}[r]|{=}
& C^{-1} \circ \overline{Gr} ( V,\nabla,\Fil)_0 \ar[dr]^{a_0\varphi_0} \ar@{}[r]|-{\oplus}
& C^{-1} \circ \overline{Gr} ( V,\nabla,\Fil)_1 \ar[dr]^(0.7){a_1\varphi_1} \ar@{}[r]|-{\oplus}
& \cdots \ar[dr]^{a_{f-2}\varphi_{f-2}} \ar@{}[r]|-{\oplus}
& C^{-1} \circ \overline{Gr} ( V,\nabla,\Fil)_{f-1} \ar[dlll]^{a_{f-1}\varphi_{f-1}}\\
( V,\nabla,\Fil) \ar@{}[r]|{=}
&( V,\nabla,\Fil)_0 \ar@{}[r]|-{\oplus}
&( V,\nabla,\Fil)_1 \ar@{}[r]|-{\oplus}
&\cdots \ar@{}[r]|-{\oplus}
&( V,\nabla,\Fil)_{f-1} \\
}
\end{equation*}
Moreover, if we choose the basis suitable, we may even make $a_0=a_1=\cdots=a_{f-2}=1$. In this case, the only map need to change is $a_{f-1}$.
\end{remark}

\begin{lemma} \label{mthm_PHDFf2MFf}
The injection in \eqref{eq_PHDFHfW2MFhf} induces another one
\begin{equation} \label{eq_PHDFfW2MFfW}
[\PHDFhf(W(k))] \hookrightarrow [\MFh(W(k'))_{\bZ_{p^f}}].
\end{equation}
\end{lemma}

\begin{proof}
Let
\[(\Flow,\psi) = (\left\{
(\overline{V},\overline{\nabla},\overline{\Fil})_{-1},
(E,\theta)_{0},(V,\nabla,\Fil)_{0},(E,\theta)_{1},(V,\nabla,\Fil)_{1},\cdots\right\},\psi)\]
and
\[(\Flow',\psi') = (\left\{
(\overline{V},\overline{\nabla},\overline{\Fil})'_{-1},
(E,\theta)'_{0},(V,\nabla,\Fil)'_{0},(E,\theta)'_{1},(V,\nabla,\Fil)'_{1},\cdots\right\},\psi')\]
be two $f$-periodic flows. Let $M=(V,\nabla,\Fil,\varphi,\iota)$ be the associated Fontaine-Faltings module of $(\Flow,\psi)$. Then \begin{equation}\label{equ_dRtermsSum}
(V,\nabla,\Fil):= ( V,\nabla,\Fil)_0 \oplus ( V,\nabla,\Fil)_1 \oplus \cdots \oplus ( V,\nabla,\Fil)_{f-1}.
\end{equation}
and
\begin{equation*} \footnotesize
\xymatrix@C=2mm@R=1cm{
C^{-1} \circ \overline{Gr} ( V,\nabla,\Fil) \ar[d]^{\varphi} \ar@{}[r]|{=}
& C^{-1} \circ \overline{Gr} ( V,\nabla,\Fil)_0 \ar[dr]^{\id} \ar@{}[r]|-{\oplus}
& C^{-1} \circ \overline{Gr} ( V,\nabla,\Fil)_1 \ar[dr]^(0.7){\id} \ar@{}[r]|-{\oplus}
& \cdots \ar[dr]^{\id} \ar@{}[r]|-{\oplus}
& C^{-1} \circ \overline{Gr} ( V,\nabla,\Fil)_{f-1} \ar[dlll]^{\varphi_0}\\
( V,\nabla,\Fil) \ar@{}[r]|{=}
&( V,\nabla,\Fil)_0 \ar@{}[r]|-{\oplus}
&( V,\nabla,\Fil)_1 \ar@{}[r]|-{\oplus}
&\cdots \ar@{}[r]|-{\oplus}
&( V,\nabla,\Fil)_{f-1} \\
}
\end{equation*}
We also have similar diagram for $M'$. Then by remark \autoref{rmk_FFotimesConstFF}, $\varphi_0$ and $\varphi_0'$ are differed by a unit in $W(k')$. This means the original flows are differed by a constant.
\end{proof}

\subsubsection{representative element with cyclotomic determinant}
Let $k'$ be a finite field extension of $k$ containing $\bF_{p^f}$.
Denote by $k_2'$ the field extension of $k$ of degree $2$.

\begin{definition} \label{notation_MF_cy}
We say that a Fontaine-Faltings module \emph{has cyclotomic determinant} if its determinant is the cyclotomic Fontaine-Faltings module\cite[Definition 2.10]{YaZu23c}. Denote by \emph{$\MFh(W(k'))_{\bZ_{p^f}}^{\rm cy}$} the subset of $\MFh(W(k'))_{\bZ_{p^f}}$ consisting of elements with cyclotomic determinant and denote by \emph{$[\MFh(W(k'))_{\bZ_{p^f}}^{\rm cy}]$} the image of $\MFh(W(k'))_{\bZ_{p^f}}^{\rm cy}$ in $[\MFh(W(k'))_{\bZ_{p^f}}]$.
\end{definition}

\begin{proposition} \label{thm_MFfW2MFWcy}
Let $M \in \MFh(W(k'))_{\bZ_{p^f}}$. There exists a constant Fontaine-Faltings module $M^\circ\in\MF^\varphi_{[0,0]}(W(k'_2))_{\bZ_{p^f}}$ of rank $1$, such that $M\otimes M^\circ \in \MFh(W(k'_2))_{\bZ_{p^f}}^{\rm cy}$.
\end{proposition}

\begin{lemma}
The determinant of any object in $\MFh(W(k'))_{\bZ_{p^f}}$ is constant and contained in $\MF^{\varphi}_{[1,1]}(W(k'))_{\bZ_{p^f}}$.
\end{lemma}
\begin{proof} Let $(V,\nabla,\Fil,\varphi,\tau)\in\MFh(W(k'))_{\bZ_{p^f}}$.
The associated Higgs-de Rham flow is of form
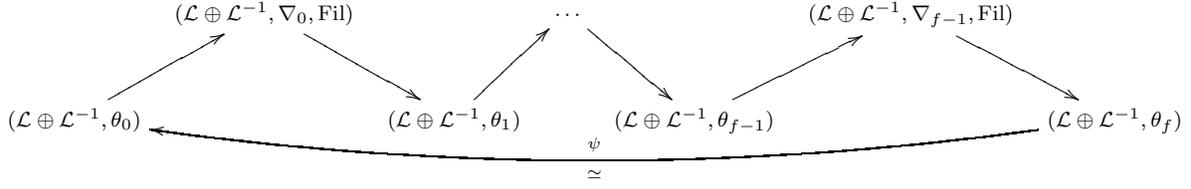
\begin{equation*} \tiny
\xymatrix@C=2mm{
&(\mL\oplus \mL^{-1},\nabla_0,\Fil) \ar[dr] && \cdots \ar[dr]&&
(\mL\oplus \mL^{-1},\nabla_{f-1},\Fil)\ar[dr] &
\\
(\mL\oplus \mL^{-1},\theta_0) \ar[ur] &&
(\mL\oplus \mL^{-1},\theta_1) \ar[ur] &&
(\mL\oplus \mL^{-1},\theta_{f-1}) \ar[ur] &&
(\mL\oplus \mL^{-1},\theta_f) \ar@/^16pt/[llllll]^{\simeq}_\psi \\
}
\end{equation*}
Taking determinant one gets
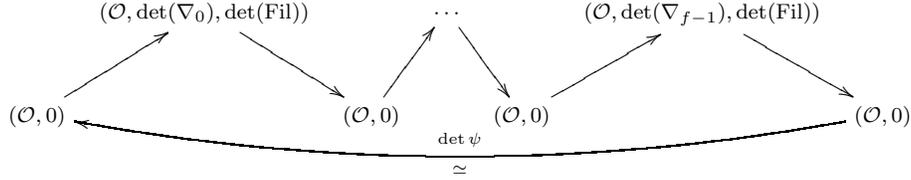
\begin{equation*} \tiny
\xymatrix@C=2mm{
&(\mO,\det(\nabla_0),\det(\Fil)) \ar[dr] && \cdots \ar[dr]&&
(\mO,\det(\nabla_{f-1}),\det(\Fil))\ar[dr] &
\\
(\mO,0) \ar[ur] &&
(\mO,0) \ar[ur] &&
(\mO,0) \ar[ur] &&
(\mO,0) \ar@/^16pt/[llllll]^{\simeq}_{\det\psi} \\
}
\end{equation*}
we note that the determinant of the Higgs field is trivial because the Higgs field is graded.

Write $\det(\nabla_i) = \rmd + \omega_i$. Due to the existence of Frobenius structure, one has (we denote $\omega_{f}:=\omega_0$)
\[\omega_{i+1} = F^*\omega_i.\]
Thus all $\omega_i=0$. This is because $p^\alpha\mid \omega$ implies $p^{\alpha+1}\mid F^*\omega$. In particular, the eigen component of the underlying de Rham bundles are all trivial. Thus it is constant.
\end{proof}

\begin{proof}[Proof of \autoref{thm_MFfW2MFWcy}]
By the lemma, we gets $\det(M)$ is constant and contained in $\MF^{\varphi}_{[1,1]}(W(k'))$. According \cite[Corollary 2.12]{YaZu23c}, there exists a constant $M^\circ\in\MF^{\varphi}_{[0,0]}(W(k'))$ such that
\[\det(M)\otimes M^\circ\otimes M^\circ = M_{cy}.\]
Since $M$ is of rank $2$, the determinant of $M\otimes M^\circ$ is cyclotomic.
\end{proof}

\begin{corollary} \label{thm_cyclDeterminant_FFM}
The base change from $W(k')$ to $W(k'_2)$ induces an injection
\[[\MFh(W(k'))_{\bZ_{p^f}}] \hookrightarrow [\MFh(W(k'_2))_{\bZ_{p^f}}^{\rm cy}].\]
\end{corollary}
\begin{proof}
The only thing we have to check is the injectivity. Suppose two Fontaine-Faltings module $M,M'\in \MFh(W(k'))_{\bZ_{p^f}}$ are differed by a constant in $\MFh(W(k'_2))_{\bZ_{p^f}}$. Then by \autoref{rmk_FFotimesConstFF}, we may identify the underlying filtered de Rham bundles with endomorphism structure. Then only the $f-1$-th eigen components of the Frobenius structures are differed by a unit $u\in W(k'_2)^\times$. But both Fontaine-Faltings module are contained in $\MFh(W(k'))_{\bZ_{p^f}}$, so the unit $u$ must contained in $W(k')$. In other word, they are contained in the same class in $[\MFh(W(k'))_{\bZ_{p^f}}]$.
\end{proof}

\subsection{Frobenius action}
\subsubsection{The Frobenius action on $\High(\barFp)$} \label{sec_FrobActionHiggs}

Recall $k$ is a finite field with cardinality $q=p^h$ containing $k_0\ni \lambda$. By extension the coefficient, we may embedding $\High(k)$ into $\High(\bark)$.

Let $\Frob\colon \bP_{\bark}^1 \rightarrow \bP_{\bark}^1$ be the Frobenius endomorphism, i.e., the base change to $\bark$ of the morphism induced by the map $a \mapsto a^p$ on $\bP^1_{\bF_p}$. The pullback functor induces natural map
\[\Frob^*\colon\High(\bark) \rightarrow \HIG_{\Frob^{-1}(\lambda)}^{\mathrm{gr}{1\over2}}(\bark).\]
Denote $\Frob_{k_0}=\Frob^{h_0}$ and $\Frob_k=\Frob^{h}$. Since $\Frob^{h}(\lambda)=\lambda$, one gets a bijective endomapping
\[\Frob_k^*\colon\High(\bark) \rightarrow \High(\bark).\]
In our case, the mapping is easy to describe: if the zero of the Higgs field $\theta$ is $(\theta)_0=:a$, then the zero of the Higgs field $\Frob_k^*(\theta)$ is $\Frob_k^{-1}(a)$. In particular, one have following result.
\begin{lemma}
$\High(k) = \Big(\High(\bark)\Big)^{\Frob_k^*}$.
\end{lemma}
\begin{proof}
Since $(E,\theta)\in\High(k)$ if and only if $a:=(\theta)_0\in k$. This is also equivalent to $\Frob^{h}(a)=a$. Now the Lemma follows the description of the action of $\Frob$ on $\High(\bark)$.
\end{proof}

\subsection{Lifting of parabolic Higgs-de Rham flows}
\subsubsection{Lifting the periodic parabolic Higgs bundles}

In this section, we lift those periodic Higgs bundles in $\High(S_1)$ to periodic ones in $\High(\overline{S})$ inductively, where $S=\Spec(W(\overline{k}))$.

Let $(\overline{E},\overline{\theta})_0$ be an $f$-periodic Higgs bundle in $\High(S_1)$ with the corresponding flow
\begin{equation*}
\xymatrix@C=2mm@R=5mm{
& {\scriptstyle (\overline V,\overline\nabla,\overline {E}^{1,0})_0} \ar[rd]|{\text{Gr}}
&
& {\scriptstyle (\overline V,\overline\nabla,\overline E^{1,0})_1} \ar@{.>}[rd]|{\text{Gr}}
& {\scriptstyle \quad \cdots} \ar@{}[d]|{\quad\cdots}
&
& {\scriptstyle \cdots} \ar@{}[d]|{\cdots}
& {\scriptstyle (\overline V,\overline\nabla,\overline E^{1,0})_{f-1}} \ar[rd]|{\text{Gr}}
&
\\
{\scriptstyle (\overline E,\overline\theta)_0} \ar[ru]|{\mC^{-1}}
&
& {\scriptstyle (\overline E,\overline\theta)_1} \ar[ru]|{\mC^{-1}}
&
& {\scriptstyle \quad \cdots}
& {\scriptstyle \cdots}
& {\scriptstyle \cdots} \ar[ru]|{\mC^{-1}}
&
& {\scriptstyle (\overline E,\overline\theta)_{f}},\ar@/^3pc/[llllllll]|{\simeq}_{\psi}
}
\end{equation*}

From now on we identify $(\overline E,\overline\theta)_{f}$ with $(\overline E,\overline\theta)_{0}$ via the isomorphism $\psi$.

\textbf{Lifting over $S_2$.} Choose a lifting $(E,\theta)_0$ of $(\overline{E},\overline{\theta})_0$ in $\High(S_2)$. By running the Higgs-de Rham flow over $S_2$, we gets
\begin{equation*}\tiny
\xymatrix@W=10mm@C=-3mm@R=5mm{
&& \cdots \ar@{}[d]|\cdots && \cdots \ar@{}[d]|\cdots \ar[dr]|{\Gr}
&& \cdots \ar@{}[d]|\cdots && \cdots \ar@{}[d]|\cdots \ar[dr]|{\Gr}
&& \cdots \ar@{}[d]|\cdots && \cdots \ar@{}[d]|\cdots \ar[dr]|{\Gr}
&& \cdots
\\
& (E,\theta)_{0} \ar[ur]^{\mC^{-1}_2} \ar@{..>}[dd]
&\cdots&\cdots& \cdots & (E,\theta)_{f} \ar[ur]^{\mC^{-1}_2} \ar@{..>}[dd] &\cdots&\cdots& \cdots & (E,\theta)_{2f} \ar[ur]^{\mC^{-1}_2} \ar@{..>}[dd] &\cdots&\cdots& \cdots & (E,\theta)_{3f} \ar[ur]^{\mC^{-1}_2} \ar@{..>}[dd]
\\
(\overline V,\overline\nabla,\overline E^{1,0})_{f-1} \ar[dr]|{\Gr}
&\\
&(\overline{E},\overline{\theta})_0
&\cdots&\cdots&\cdots&(\overline{E},\overline{\theta})_0
&\cdots&\cdots&\cdots&(\overline{E},\overline{\theta})_0
&\cdots&\cdots&\cdots&(\overline{E},\overline{\theta})_0 \\
}
\end{equation*}
\begin{remark}
In our case, the obstruction for lifting Hodge filtration vanish even the lifting is unique due to \autoref{thm_ClassfyR2PdE}. In particular, lifting flow is uniquely determined by the lifting $(E,\theta)_0$.
\end{remark}

Since the Higgs bundle in $\High(S)$ is uniquely determined by its zero, the lifting torsor space of $(\overline{E},\overline{\theta})_0$ is isomorphic to $\bA^1_k$ (non-canonically). In \cite{KYZ20D}, the operator $\Big(\Gr\circ C_2^{-1}\Big)^f$ induces a self map on this torsor space, which is of form $z\mapsto az^p+b$ if we choose an identification of the torsor space with the affine line over $k$. In particular, the solutions of the Artin-Schreier equation
\begin{equation} \label{equ:ArtinSchreier}
az^p+b = z
\end{equation}
correspond to $f$-periodic Higgs bundles in $\High(S_2)$ which lifts $(\overline E,\overline\theta)_{0}$.
Hence, if we extend the field $k$ a little bit, we can always find $f$-periodic Higgs bundles in $\High(S_2)$ which lifts $(\overline E,\overline\theta)_{0}$.

\begin{remark}
If $a=0$, then there are exact one periodic lifting of $(\overline E,\overline\theta)_{0}$ in $\High(S_2)$. In this case, we do not need extend the field, the lifting is already defined over $S_2$.

If $a\neq 0$, then there are exact $p$ periodic lifting of $(\overline E,\overline\theta)_{0}$ in $\High(S_2)$, once we enlarge the field $k$ a little bit properly.
\end{remark}

\textbf{Lifting over $S$.} Working the above lifting procedure inductively, we obtain $f$-periodic Higgs bundles in $\High(\overline{S})$, which lifts $(\overline E,\overline\theta)_{0}$.

\begin{remark}
The lifting of a periodic Higgs bundle over $\Fq$ to a periodic Higgs bundle over $\bZ_p^{ur}$ is in general not unique, as the solutions of the Artin-Schreier equation are not unique in general.
\end{remark}

\textbf{The associated parabolic Fontaine-Faltings modules.} Let $(E,\theta)_0\in \High(\overline{S})$ be $f$-periodic with the corresponding flow
\[\Big((V,\nabla,E^{1,0})_{-1},(E,\theta)_0,(V,\nabla,E^{1,0})_0,\cdots (E,\theta)_{f-1},(V,\nabla,E^{1,0})_{f-1},(E,\theta)_f,\cdots\Big)\]
Similarly, as in \autoref{subsec_HDF2FFM}, we gets a parabolic Fontaine-Faltings module over $(\bP_{\overline{S}}^1,D_{\overline{S}})/\overline{S}$.

\subsubsection{Lifting in supersingular case}
\begin{definition}\label{def_supersingular}
An element $\lambda$ in $\mO_S(S)=W(k)$ such that
\[\overline{\lambda}:=\lambda\pmod{p}\neq 0,1\in k\]
is called \emph{supersingular} if the elliptic curve
$C_{\overline\lambda}$ is supersingular.
\end{definition}

By direct calculation, one can check that if $\lambda$ is supersingular, then the coefficient $a$ appeared in the Artin-Schreier equation \eqref{equ:ArtinSchreier} is zero, for explicit calculation see the \autoref{sec_appendix_B}. In particular, one has the following theorem.
\begin{theorem} \label{mthm_PHIGk2PHIGW}
Assume that $\lambda$ is supersingular. Any $f$-periodic parabolic Higgs bundle $(\overline E,\overline\theta)\in\PHighf(k)$ has a unique $f$-periodic lifting $(E,\theta) \in\PHighf(W(k))$. In other words, the modulo $p$ reduction induces a bijection
\[\PHighf(W(k)) \xrightarrow{1:1} \PHighf(k).\]
\end{theorem}

\section{\bf Overconvergent $F$-isocrystals} \label{sec_main_F_Isoc}

In this section, the aim is to
\begin{itemize}
\item construct a natural injective map (\autoref{mthm_FF2Isoc_classes})
\begin{equation}\label{equ_main_sec_isoc_I}
\xymatrix{
[\MFh(W(k_2'))_{\bZ_{p^f}}^{\rm cy}] \ar@{>->}[r] & [\FIsoch(k'_2)_{\bQ_{p^f}}^{\rm triv}].
}
\end{equation}
where $[\FIsoch(k'_2)_{\bQ_{p^f}}^{\rm triv}]$ is the set of all rank-$2$ overconvergent $F$-isocrystal over projective line with given exponents and trivial determinant modulo an equivalent relation defined in \autoref{def_diff_by_const_Fisoc}, and
\item show that the image of the following composition
\begin{equation}\label{equ_main_sec_isoc_II}
\xymatrix{
\High(k)\ar@{>..>}[r]^-{\eqref{equ_main_sec_FFHDF}} & [\MFh(W(k'_2))_{\bZ_{p^f}}^{\rm cy}] \ar@{>->}[r] & [\FIsoch(k'_2)_{\bQ_{p^f}}^{\rm triv}]
}
\end{equation}
is fixed by the Frobenius action, see \autoref{thm_Frob_preserving}.
\end{itemize}

\subsection{$F$-crystals}

In this subsection, we recall some basic definitions we need for this article, including those of convergent $F$-isocrystals, overconvergent $F$-isocrystals and convergent log-$F$-isocrystals from \cite[Definition 2.1, Definition 2.4 and Definition 7.1]{Ked22}.

Let $X$ be a proper smooth variety over $k$, $D$ be a normal crossing divisor in $X$ and $U = X-D$. We endow $X$ with the natural logarithmic structure induced by $D$, and simply write $(X,D)$ for the corresponding logarithmic scheme.

\subsubsection{(logarithmic) $F$-crystal}
Kato has defined (logarithmic) crystalline site $((X,D)/W)^{\log}_{\rm crys}$, and $\Crys((X,D)/W)$ the category of crystals in \emph{finite coherent $\mO_{(X,D)/W}$-modules}. By functoriality of the crystalline topos, the absolute Frobenius $\Frob_{X}:X\rightarrow X$ gives a functor $\Frob_{X}^{*}\colon\Crys((X,D)/W) \rightarrow \Crys((X,D)/W)$. An \emph{(logarithmic) $F$-crystal} in finite, locally free modules on $U$ is a crystal $\mE$ in finite, locally free $\mO_{(X,D)/W}$-modules together with an isogeny $F\colon \Frob_{X}^{*}\mE\rightarrow \mE$. The $\bZ_p$-linear category of (logarithmic) $F$-crystals in finite, locally free modules is denoted as $\FCrys((X,D)/W)$.

\begin{theorem}[Kato] \label{thm_Kato_equivalent}
There is an equivalence between the following two categories:
\begin{enumerate}[$(1).$]
\item the category of crystals $\mE$ on $((X,D)/W)^{\log}_{\rm crys}$,
\item the category of $\mO_{\mX}$-modules $V$ on $\mX$ with a quasi-nilpotent integrable logarithmic connection
\[\nabla\colon V\rightarrow V\otimes\Omega^1_{\mX/W}(\log\mD).\]
\end{enumerate}
\end{theorem}
\begin{remark} \label{rmk:26} We call the de Rham sheaf $(V,\nabla)$ associated to a crystal $\mE$ to be the \emph{realization of $\mE$ over $(\mX,\mD)$}. Kato's Theorem implies a logarithmic de Rham bundle is a realization of a logarithmic crystal if its connection is quasi-nilpotent. We sometimes simply call such a logarithmic de Rham sheaf $(V,\nabla)$ a logarithmic crystal over $(X,D)$.
\end{remark}

According \autoref{rmk:26}, we also write the logarithmic $F$-crystal as the triple $(V,\nabla,\mF)$.
\begin{corollary}
There is an equivalence between the following two categories
\begin{enumerate}[$(1).$]
\item the category of $F$-crystals $\mE$ on $((X,D)/W)^{\log}_{\rm crys}$,
\item the category of triples $(V,\nabla,\Phi)$, where $V$ is vector bundle on $\mX$, $\nabla$ a quasi-nilpotent integrable logarithmic connection
\[\nabla\colon V\rightarrow V\otimes\Omega^1_{\mX/W}(\log\mD)\]
and $\Phi$ is an injection
\[\Phi\colon \varphi^*(V,\nabla)\hookrightarrow (V,\nabla).\]
\end{enumerate}
\end{corollary}

\begin{corollary}
Let $(V,\nabla,\Fil,\Phi)$ be a Fontaine-Faltings module. By forgetting the filtration, one gets an $F$-crystal over $(X,D)$.
\end{corollary}

\begin{remark}\label{rmk11}
For an $F$-crystal over $(X,D)$, consider its realization $(V,\nabla,\Phi)$ on the $p$-adic formal completion of $(\mX,\mD)$. The presence of the Frobenius structure forces the reductions modulo $\bZ$ of the eigenvalues of the residue map would form a set stable under multiplication by $p$. In particular the eigenvalues are rational numbers. See \cite[7.2]{Ked22}.
\end{remark}

\subsection{$F$-isocrystals}
In this subsection, we also recall some basic definitions needed for this article, including those of convergent $F$-isocrystals, overconvergent $F$-isocrystals and convergent log-$F$-isocrystals from \cite[Definition 2.1, Definition 2.4 and Definition 7.1]{Ked22}.

Let $X$ be a proper smooth variety over $k$, $D$ be a normal crossing divisor in $X$ and $U = X-D$. We endow $X$ with the natural logarithmic structure induced by $D$, and simply write $(X,D)$ for the corresponding logarithmic scheme.

\subsubsection{overconvergent $F$-isocrystal}
Suppose there exists a lifting $\sigma\colon \mU\rightarrow\mU$ of the absolute Frobenius on
$U$. A \emph{convergent $F$-isocrystal} over $U$ is a de Rham bundle $\mE$ over the Raynaud generic fiber $\mU_K$ of the formal completion $\mU$ of $U$ along the special fiber $U$ together with an isomorphism $F\colon \sigma^*\mE\rightarrow\mE$ of de Rham bundles. Denote by \emph{$\FIsoc(U)$} the category of all convergent $F$-isocrystals over $U$. Up to canonical equivalence, this category does not depend on the choice of the lifting $\sigma$. In general, there may not exist a global lifting of the absolute Frobenius on $U$, but one can still define the category $\FIsoc(U)$ (see \cite[definition 2.1]{Ked22}). One way to do this is as follows: we can find local liftings of absolute Frobenius on $U$, define local categories by using these local liftings as above, and use the canonical equivalences between local categories to glue them into a global one.

A convergent $F$-isocrystal is called \emph{overconvergent} if it can be extended to a strict neighborhood of $\mU_K$ in $\mX_K$. Denote by \emph{$\FIsoc^\dagger (U)$} the category of all overconvergent $F$-isocrystal over $U$.

For each finite extension $L$ of $\bQ_p$ within $\overline{\bQ}_p$, let \emph{$\FIsoc^\dagger(U)_L$} denote the category of objects of $\FIsoc^\dagger(U)$ with a $\bQ_p$-linear action of $L$. Let \emph{$\FIsoc^\dagger(U)_{\overline\bQ_p}$} be the $2$-colimit of the category $\FIsoc^\dagger(U)_L$ over all finite extensions $L$ of $\bQ_p$ within $\overline\bQ_p$.

\subsubsection{characteristic polynomials of an overconvergent $F$-isocrystal}
Given an overconvergent $F$-isocrystal $\mE$ on $U$. For any closed point $x$ in $U$, the fiber $\mE_x$ of $\mE$ at $x$ carries an action of (geometric) Frobenius. We define the characteristic polynomial of $\mE$ at $x$ to be
\[P_x(\mE,t)=\det(1-Fr_x\cdot t\mid_{\mE_x}).\]

\subsubsection{convergent log-$F$-isocrystal}
A \emph{convergent log-$F$-isocrystal} is a logarithmic de Rham bundle over $\mX_K$ together with an isomorphism $F$ of logarithmic de Rham bundles similar as that in the definition of convergent $F$-isocrystal (see e.g. \cite[Definition 7.1]{Ked22}). For such objects, the residues of the underlying logarithmic isocrystal are automatically nilpotent. We denote by $\FIsoc^{\rm nilp}_{\log}(X,S)$ the category of all convergent log-$F$-isocrystals on the logarithmic pair $(X,S)$.

\begin{remark}
\begin{enumerate}[$(1).$]
\item Under our assumption $X_K$ is proper, a convergent log-$F$-isocrystals can be algebraicalized to a vector bundle over $X_K$ together with an integral logarithmic connection and a parallel semilinear action.
\item To a logarithmic crystalline representation, we may attach an convergent log-$F$-isocrystal. For a logarithmic crystalline representation $\rho\colon \pi_1(U_K)\rightarrow \textrm{GL}_r(\bZ_{p^f})$, according Faltings' definition of crystalline representation~\cite{Fal89}, there exists an attached logarithmic Fontaine-Faltings module $(V,\nabla,\Fil,\varphi,\iota)$\footnote{Faltings' original definition is for $\bZ_p$-representations. It can be easily extended to $\bZ_{p^f}$-representations by adding an endomorphism structures $\iota$ on the side of Fontaine-Faltings modules. More precisely, see \cite{LSZ19}.} Forgetting the filtration and tensoring $\bQ_p$, one gets the attached convergent log-$F$-isocrystal $(V,\nabla,\varphi,\iota)_{\bQ_p}$.
\end{enumerate}
\end{remark}

\subsubsection{Trace of the Frobenius}
Let $(V,\nabla,\Fil,\varphi)$ be a logarithmic Fontaine-Faltings module over $(\mY,\mD_\mY)$.
For any closed point $x$ in $U_1$ with residue field $k'$, by the smoothness of $Y$, we can find a $\Spf(W)$-point $\widehat{x}$ in $\mU$ which lifts $x$. By restricting on $\widehat{x}$, we gets a Fontaine-Faltings module over this point, which is nothing just a finite generated free filtered $W(k')$-module $V_{\widehat x}$ together with a $\sigma$-semilinear isomorphism $F_{\widehat x}\colon \widetilde{V}_{\widehat x}\simeq V_{\widehat x}$, where $\widetilde{V}_{\widehat x} = \sum\limits_{\ell = a}^{b} \frac1{p^\ell}\Fil^\ell V_{\widehat x} \subset V_{\widehat x} \otimes \bQ_p$. By tensoring $\bQ_p$, one gets an $F$-isocrystal $(V_{\widehat{x}}\otimes \bQ_p,F_{\widehat{x}})$ over the finite field $k'$.

One can easily checks following result.
\begin{lemma} The $(V_{\widehat{x}}\otimes \bQ_p,F_{\widehat{x}})$ is isomorphic to the restriction of $\mE$ on $x$. In particular, the isomorphic class of $F$-isocrystal $(V_{\widehat{x}}\otimes \bQ_p,F_{\widehat{x}})$ does not depend on the choice of $\widehat{x}$.
\end{lemma}

\subsubsection{The dependence of the traces on the choices of the Frobenius structures}

\subsubsection{$F$-isocrystal over $k$ with coefficients}

Let $k$ be a finite field and Let $L$ be an algebraic extension of $\bQ_p$. Recall that the following are equivalent:
\begin{itemize}
\item an $F$-isocrystal over $k$ with coefficient in $L$ of rank $r$;
\item a free $W(k)[\frac1p]\otimes_{\bQ_p} L$-module of rank $r$ together with a $\sigma\otimes \id$-linear morphism
\[F\colon V\rightarrow V.\]
\item a $W(k)[\frac1p]$ vector space of rank $r[L:\bQ_p]$ endowed with a $\sigma$-semilinear isomorphism $F\colon V\rightarrow V$ and with an endomorphism structure
\[L\rightarrow \End(V,F).\]
\end{itemize}
In the following, we will always identify the three kinds of objects, and call them $F$-isocrystals over $k$ with coefficient $L$. Denote by $\FIsoc(k)_L$ the category of all $F$-isocrystals over $k$ with coefficient $L$.

\subsubsection{The $F$-isocrystal $\mE_{1/2}$.}

Since $p\geq3$, we may choose a square root $\sqrt{1-p}$ of $1-p$ in $\bQ_p$. Since $\bQ_{p^2}$ is an extension of $\bQ_p$ of degree $2$, we may find some $\zeta\in \bQ_{p^2}\setminus \bQ_p$ such that $\zeta^2 \in \bQ_{p}$. Thus $\sigma(\zeta)=-\zeta$, where $\sigma$ is the generator of the Galois group $\Gal(\bQ_{p^2}/\bQ_p)$, which is also the lifting of the absolute Frobenius map on $\bF_{p^2}$.

Let $V_{1/2}$ be a $\bQ_{p^2}\otimes_{\bQ_p}\bQ_{p^2}$-module of rank $1$ with basis $e$. Denote by $F_{1/2}$ a $\sigma\otimes\id$-linear endomorphism on $V$ given by
\[F_{1/2}(e) = (1\otimes 1 + \sqrt{1-p} \zeta\otimes \zeta^{-1}) e\]
Then
\[F_{1/2}^2(e) = (1\otimes 1 - \sqrt{1-p} \zeta\otimes \zeta^{-1})\cdot (1\otimes 1 + \sqrt{1-p} \zeta\otimes \zeta^{-1}) e = p.\]

According the equivalent relation, we get an $F$-isocrystal,denote by $\mE_{1/2}$, over $\bF_{p^2}$ with coefficient in $\bQ_{p^2}$.
Let $X$ be an varieties defined over $k$. Assume $k$ contains $\bF_{p^2}$. Then there is a structure morphism
\[f\colon X \rightarrow \Spec(\bF_{p^2}).\]
By pulling back $\mE_{1/2}$ along $f$ we get a constant overconvergent $F$-isocrystal of rank $1$ with coefficient in $\bQ_{p^2}$. By abusing notion, we still denote it by $\mE_{1/2}$.

\subsubsection{The cyclotomic $F$-isocrystal $\mE_{cy}$}
\begin{definition}
Let $V$ be a $\Qp$-module of rank $1$ with basis $e$. Denote by $F$ the $\Qp$-linear endomorphism on $V$ by multiplying $p$. Then we get an $F$-isocrystal, denote by $\mE_{cy}$, over $\Fp$.
\end{definition}

\begin{lemma}
$\mE_{cy} = \mE_{1/2}^{\otimes 2}$.
\end{lemma}

\subsubsection{The change of traces of Frobenius under twisting by $\mE_{1/2}$}

Let $k$ be a finite field with cardinality $p^h$. Let $L$ be an algebraic extension of $\bQ_{p}$.

Let $(V,F)$ be an $F$-isocrystal over $k$ with coefficient in $L$ of rank $r$, or equivalently, a free $W(k)[\frac1p]\otimes_{\bQ_p} L$-module of rank $r$ together with a $\sigma\otimes \id$-linear morphism
\[F\colon V\rightarrow V.\]
Then $F^h$ is a $W(k)[\frac1p]\otimes_{\bQ_p} L$-linear endomorphism on $V$. Denote by $P((V,F),t)$ the characteristic polynomial and by $\tr(V,F)$ the trace of $F^h$ acting on $V$.

\begin{lemma}
$P((V,F),t)\in L[t]$ and $tr(V,F)\in L$.
\end{lemma}
\begin{proof}
Let $e_1,\cdots,e_r$ be a system $W(k)[\frac1p]\otimes_{\bQ_p} L$-basis of $V$. Then $F$ can be represented as
\[F(e_1,\cdots,e_r) = (e_1,\cdots,e_r) A.\]
Thus
\[F^h (e_1,\cdots,e_r) = (e_1,\cdots,e_r) \underbrace{A\cdot A^{\sigma\otimes\id}\cdot A^{(\sigma\otimes\id)^2}\cdot \cdots \cdot A^{(\sigma\otimes\id)^{h-1}}}_{=: B}.\]
Since $B^{\sigma\otimes\id} = A^{-1} B A$, both $P((V,F),t)$ and $tr(V,F)$ are invariant under $\sigma\otimes\id$.
\end{proof}

Denote by $k'$ the field generated by $k$ and $\bF_{p^2}$ and denote by $L'$ the field generated by $L$ and $\bQ_{p^2}$. Then both
$R_1:=\bQ_{p^2}\otimes_{\bQ_p} \bQ_{p^2}$ and $R_2:=W(k)[\frac1p]\otimes_{\bQ_p} L$
can be viewed as a subring of $R:=W(k')[\frac1p]\otimes_{\bQ_p} L'$.
By extending the ring from $R_1$ and $R_2$ to $R$, we gets two objects in $\FIsoc(k')_{L'}$ from $(V,F)$ and $\mE_{1/2}$. We denote by
\[(V,F)\otimes \mE_{1/2}\]
their tensor product in the category $\FIsoc(k')_{L'}$.

\begin{lemma} Suppose $\bF_{p^2}\subseteq k$. Then $2\mid h$,
\[P((V,F)\otimes\mE_{1/2},t) = p^{rh/2}P((V,F),p^{-h/2}t) \quad \text{and} \quad \tr((V,F)\otimes\mE_{1/2}) = p^{h/2} \tr(V,F).\]
\end{lemma}

\begin{proof}
Clearly, the surjective $R$-module of $(V,F)\otimes \mE_{1/2}$ is
\[(V\otimes_{R_2}R)\otimes_R(V_{1/2}\otimes_{R_1}R)\]
which is free over $R$ of rank $r$ with generators
\[e_1':=(e_1\otimes1)\otimes (e\otimes1),\cdots,e'_r:=(e_r\otimes1)\otimes(e\otimes1).\] Denote $\eta:=1\otimes1+\sqrt{1-p}\zeta\otimes\zeta^{-1}$. Then
\[F(e'_1,\cdots e'_r) = (e'_1,\cdots,e'_r)\cdot A\eta.\]
The Lemma follows the following calculate:
\begin{equation*}
\begin{split}
F^h (e_1\otimes e,\cdots e_r\otimes e)
&= (e_1\otimes e,\cdots e_r\otimes e)\cdot A\eta\cdot (A\eta)^{\sigma\otimes\id} \cdot \cdots \cdot (A\eta)^{(\sigma\otimes\id)^{h-1}} \\
&= (e_1\otimes e,\cdots e_r\otimes e)\cdot p^{h/2}B\\
\end{split}
\end{equation*}
\end{proof}

\subsection{The convergence of parabolic Fontaine-Faltings modules.}

In this subsection, we construct the overconvergent $F$-isocrystals from parabolic Fontaine-Faltings modules.

\subsubsection{convergence of a logarithmic de Rham bundle over $(\mY_K,\mD_{\mY_K})$}
Recall that Kedlaya gave an equivalent functor \cite[6.4.1]{Ked07} from the category of convergent
logarithmic isocrystals\cite[6.1.7]{Ked07} to the category of convergent log de Rham bundles\cite[6.3.1]{Ked07}. So by restricting from the associated convergent logarithmic isocrystal, one gets an overconvergent isocrystal from a convergent logarithmic de Rham bundle. Back to our situation, we only need to show the convergence of the underlying logarithmic de Rham bundle of a logarithmic Fontaine-Faltings module. Before this, let us recall Kedlaya's definition of convergence.
\begin{definition}[{\cite[6.3.1]{Ked07}}]
A logarithmic de Rham bundle $(V,\nabla)$ over $(\mY_K,\mD_{\mY_K})$ is called \emph{convergent}, if there exists some strict neighborhood of $\mU_K$ in $\mY_K$, on which the restriction of $(V,\nabla)$ is overconvergent\footnote{See \cite[2.5.3 and 2.5.4]{Ked07}} along $\mZ_K$.
\end{definition}

\begin{remark}
According \cite[Proposition 2.5.6]{Ked07}, a logarithmic de Rham bundle over $(\mY_K,\mD_{\mY_K})$ is convergent if and only if for any $\eta\in[0,1)$, there exists a sufficient small strict neighborhood of $\mU_K$ in $\mY_K$, on which the restriction is $\eta$-convergent\footnote{See the explicit definition for $\eta$-convergent in \cite[Definition 2.4.2]{Ked07}}.
\end{remark}

\subsubsection{Generic fiber of a logarithmic de Rham bundle over $(\mY,\mD_\mY)$}

Let $(V,\nabla)$ be logarithmic de Rham bundle over $(\mY,\mD_\mY)$. By restriction on the Raynaud generic fiber, one gets a logarithmic de Rham bundle $(\mY_K,\mD_{\mY_K})$, which we will simply call \emph{the generic fiber of $(V,\nabla)$}, and denote by $(V_K,\nabla_K)$.

\begin{lemma}
\label{thm_MF_qnilp}
Let $(V,\nabla,\Fil,\varphi)$ be a parabolic Fontaine-Faltings module over $(\mY,\mZ)$. For any $\alpha \in \bQ^n$, the generic fiber of the logarithmic de Rham bundle $(V_\alpha,\nabla_\alpha)$ is $\eta$-convergent for all $\eta\in[0,1)$.
\end{lemma}
\begin{proof}
By the definition, the filtration in any (parabolic) Fontaine-Faltings module has level contained in $[a,b]$with $b-a\leq p-2$. The grading structure in the corresponding graded Higgs bundle $(E,\theta)$ has level contained in $[0,p-2]$. in other words,
there exists graded decomposition $E=\oplus_{i=0}^{p-2} E_i$ such that
the Higgs field is a sum of maps $E_i\rightarrow E_{i+1}$ where $i$ run through $\{0,1,\cdots,p-3\}$. We consider the modulo $p$-reduction of the connection in the Fontaine-Faltings module, which comes from the modulo $p$ reduction of the graded Higgs bundle under the inverse Cartier functor. From the explicit construction of inverse Cartier functor\footnote{Note that the inverse Cartier functor $C_1^{-1}$ (the characteristic $p$ case) is introduced in the seminal work of Ogus-Vologodsky \cite{OgVo07}. See also \cite{LSZ19}.}, one has $\left(\nabla_{\partial}\right)^{p-1}\equiv 0\pmod{p}$. Thus the Lemma follows the definition of $\eta$-convergent in \cite[Definition 2.4.2]{Ked07} immediately.
\end{proof}

\begin{corollary} \label{thm_logFF2ConvLogdR} Let $(V,\nabla,\Fil,\varphi)$ be a parabolic Fontaine-Faltings module over $(\mY,\mZ)$. For any $\alpha\in\bQ^n$, the generic fiber of the logarithmic de Rham bundle $(V_\alpha,\nabla_\alpha)$ is convergent.
\end{corollary}

\begin{proof}
The convergence follows \autoref{thm_MF_qnilp}, by Kedlaya's criterion \cite[2.5.6]{Ked07}.
\end{proof}
Together with Kedlaya's equivalent functor \cite[6.4.1]{Ked07}, we get functors from the category logarithmic Fontaine-Faltings modules $(\mY,\mD_\mY)$ to the category convergent logarithmic isocrystal over $(Y_1,D_1)$ indexed by $\alpha\in \bQ^n$.
\begin{equation}\label{eq_logFF2logIsoc}
\left\{
\begin{array}{lll}
\text{parabolic Fontaine-Faltings}\\
\text{modules over $(\mY,\mD_{\mY})$}\\
\end{array}
\right\}
\xrightarrow{\quad\quad}
\left\{{\text{
convergent logarithmic
} \atop \text{
isocrystal over $(Y_1,D_{Y_1})$
}}\right\}
\end{equation}

Generally, we get a Frobenius structure on these convergent logarithmic isocrystal over $(Y_1,D_{Y_1})$.
\begin{proposition} \label{thm_paraFF2overconvFIsoc}
Let $(V,\nabla,\Fil,\varphi)$ be a parabolic Fontaine-Faltings module over $(\mY,\mD_\mY)$. Let $\{\mE_\alpha\}$ be the associated convergent logarithmic isocrystals over $(Y_1,D_1)$ given in \eqref{eq_logFF2logIsoc}. Then
\begin{enumerate}[$(1).$]
\item the de Rham bundles $F^*(\widetilde{V}_\alpha,\widetilde{\nabla}_\alpha)$ are convergent for all $\alpha\in\bQ^n$, and they have common restriction on the open subset $\mY_K^\circ$.
\item The Frobenius structure in parabolic Fontaine-Faltings module induces an natural injective morphism of logarithmic de Rham bundles over $(\mY,\mD_\mY)$
\[\varphi\colon F^*(\widetilde{V}_0,\widetilde{\nabla}_0)\hookrightarrow (V_0,\nabla_0).\]
\item[$(3)$] After restricting $\mE_\alpha$ onto the open subset $\mY_K^\circ$, one gets an overconvergent $F$-isocrystal over $(U_1,Y_1)$. In summary, we gets a functor
\begin{equation}\label{eq:logFF2logFIsoc}
\left\{
\begin{array}{lll}
\text{parabolic Fontaine-Faltings}\\
\text{modules over $(\mY,\mD_{\mY})$}\\
\end{array}
\right\}
\xrightarrow{\quad\quad}
\left\{{\text{
overconvergent
} \atop \text{
$F$-isocrystal over $(U_1,Y_1)$
}}\right\}
\end{equation}
\end{enumerate}

\end{proposition}
\begin{proof}
Clearly, (1) and (3) follows (2) directly. We show (2) as follows.

Since one always has injection $(V_0,\nabla_0,\Fil_0)\hookrightarrow (V,\nabla,\Fil)$ of filtered parabolic de Rham bundles, the front one endowed with trivial parabolic structure. After taking parabolic version of Faltings tilde functor and inverse Cartier functor, one gets an injective morphism between parabolic de Rham bundles
\[\mF^*(\widetilde{V}_0,\widetilde{\nabla}_0) \hookrightarrow \mF^*(\widetilde V,\widetilde\nabla).\]
Then composing with the Frobenius structure in the Fontaine-Faltings module, one gets the desired injective morphism.
\end{proof}

\begin{remark}
Due to the existence of the parabolic structure, the Frobenius map in (2) is not isomorphism in general. But if the parabolic structure is trivial(in other word, for a logarithmic Fontaine-Faltings module), we will indeed get a convergent logarithmic $F$-isocrystal over $(U_1,Y_1)$.
\end{remark}

We now have endomorphism structures involved.
\begin{corollary} \label{FFM2Fisoc} By forgetting the filtration, and then restricting on the Raynaud generic fiber, one gets the following functor
\begin{equation*}
\left\{
\begin{array}{lll}
\text{parabolic Fontaine-Faltings}\\
\text{modules over $(\mY,\mD_{\mY})/S$ with}\\
\text{$\bZ_{p^f}$-endomorphism structures}\\
\end{array}
\right\}
\xrightarrow{\quad\quad}
\left\{{\text{
overconvergent $F$-isocrystal over
} \atop \text{
$(Y_1,D_{Y_1})$ with coefficient in $\bQ_{p^f}$
}}\right\}
\end{equation*}
\end{corollary}

\subsection{Overconvergent $F$-isocrystals on the projective line}

\subsubsection{Overconvergent $F$-isocrystals with given exponents}

Denote by \emph{$\FIsoch(k)_{\bQ_{p^f}}$} the set of all rank-$2$ overconvergent $F$-isocrystal $\mE$ over $(\bP^1_{k},\{0,1,\lambda,\infty\})$ with coefficients in $\bQ_{p^f}$ such that the exponents along $0,1,\lambda$ are integers and the exponents along $\infty$ are half integers.

Let $k'$ be a field extension of $k$ containing $\bF_{p^f}$. For any $M\in \MFh(W(k'))_{\bZ_{p^f}}$, by \autoref{FFM2Fisoc}, we get an overconvergent $F$-isocrystal $\mE_M$ over $(\bP^1_{k'},\{0,1,\lambda,\infty\})/k'$ endowed with an $\bQ_{p^f}$-endomorphism structure with the same exponents as $M$ (up to modulo $\bZ$). Thus $\mE_M\in \FIsoch(k')_{\bQ_{p^f}}$. This give us a
natural map
\begin{equation} \label{eq_MF2Fisoc}
\MFh(W(k'))_{\bZ_{p^f}} \rightarrow \FIsoch(k')_{\bQ_{p^f}}.
\end{equation}

\subsubsection{An equivalence relation on $\FIsoch(k')_{\bQ_{p^f}}$}

Let $k'$ be a field extension of $k$ containing $\bF_{p^f}$.

\begin{definition} \label{def_diff_by_const_Fisoc}
Let $\mE$ and $\mE'\in\FIsoch(k')_{\bQ_{p^f}}$. We call they are \emph{differed by a constant (over $k'$)}, if there exists an $F$-isocrystal $\mE^\circ$ over $k'$ with coefficient in $\bQ_{p^f}$ of rank $1$ such that
\[\mE' = \mE \otimes \mE^\circ.\]
Differed by a constant is an equivalent relation on the set
\[\FIsoch(k')_{\bQ_{p^f}}.\]
Denote by $[\FIsoch(k')_{\bQ_{p^f}}]$ the set of all equivalent classes.
\end{definition}

Denote by $\FIsoch(k')_{\bQ_{p^f}}^{\rm triv}$ the subset of $\FIsoch(k')_{\bQ_{p^f}}$ with trivial determinant, and denote by $[\FIsoch(k')_{\bQ_{p^f}}^{\rm triv}]$ the image of $\FIsoch(k')_{\bQ_{p^f}}^{\rm triv}$ in $[\FIsoch(k')_{\bQ_{p^f}}]$.

\begin{lemma} \label{mthm_Fisok2Fisokn}
Let $\mE$ and $\mE'\in\FIsoch(k')_{\bQ_{p^f}}$. Then for any $n\geq1$, they are differed by a constant over $k_n'$ if and only if they are differed by a constant over $k'$. Thus one has the natural injection
\[[\FIsoch(k')_{\bQ_{p^f}}] \hookrightarrow [\FIsoch(k'_n)_{\bQ_{p^f}}].\]
\end{lemma}

\begin{lemma} \label{mthm_FF2Isoc}
The map \eqref{eq_MF2Fisoc} induces an injection between the sets of equivalence classes
\[[\MFh(W(k'))_{\bZ_{p^f}}] \hookrightarrow [\FIsoch(k')_{\bQ_{p^f}}].\]
\end{lemma}

\begin{proof}
This follows the facts that the modulo $p$ reductions of the de Rham terms appeared in the Higgs-de Rham flow associating to any object in $\MFhf(W(k'))$ are all stable and the Hodge filtration is unique.
\end{proof}

\begin{lemma} \label{mthm_FF2Isoc_classes}
Assume $2\mid f$. Then the map in \autoref{mthm_FF2Isoc} (replacing $k'$ with $k'_2$) induces following injection
\[[\MFh(W(k_2'))_{\bZ_{p^f}}^{\rm cy}] \hookrightarrow [\FIsoch(k'_2)_{\bQ_{p^f}}^{\rm triv}].\]
\end{lemma}

\begin{proof}
For any $M\in \MFh(W(k'))_{\bZ_{p^f}}^{\rm cy}$, the $F$-isocrystal $\mE_M\otimes \mE_{1/2}^{-1}$ has trivial determinant.
\end{proof}

Together with natural mapping from periodic Higgs bundles to Fontaine-Faltings modules, we get the following result.
\begin{corollary} \label{thm_PHIG_k_to_F_Isoc_k}
Assume $\lambda\in W(k)$ is supersingular. Running Higgs-de Rham flow induces a natural injection
\begin{equation} \label{thm_Higgs2Fisoc}
\PHighf(k') \hookrightarrow [\FIsoch(k')_{\bQ_{p^f}}].
\end{equation}
\end{corollary}

\subsubsection{The Frobenius action on $\FIsochf(k)$}

\begin{proposition}\label{thm_Frob_preserving} \label{mthm_Higgs2FIsoc_Frob_Preserved}
Assume $\lambda\in W(k)$ is supersingular.
The injection in \eqref{thm_Higgs2Fisoc} is preserved by the actions of $\Frob_k$ on both sides.
\end{proposition}

\begin{proof}
Let $(E,\theta)\in \PHighf(k')$ with corresponding $F$-isocrystal $\mE\in \FIsoch(k')_{\bQ_{p^f}}$. Assume it initials the following periodic Higgs-de Rham flow
\[\Flow = \{(E,\theta)_0=(E,\theta),(V,\nabla,\Fil)_0,(E,\theta)_1,(V,\nabla,\Fil)_1,\cdots\}.\]
Since $\lambda$ is supersingular, it lifts uniquely (up to a constant) to an $f$-periodic flow in $\HDFh(W(k'))$.
\[\widehat\Flow = (\widehat E,\widehat\theta)_0,(\widehat V,\widehat\nabla,\widehat\Fil)_0,(\widehat E,\widehat\theta)_1,(\widehat V,\widehat\nabla,\widehat\Fil)_1,\cdots.\]
Then by forgetting the Frobenius structure in Fontaine-Faltings module, one gets the underlying filtered de Rham bundle
\[(\widehat V,\widehat\nabla,\widehat\Fil) = (\widehat V,\widehat\nabla,\widehat\Fil)_0\oplus (\widehat V,\widehat\nabla,\widehat\Fil)_1 \oplus \cdots \oplus (\widehat V,\widehat\nabla,\widehat\Fil)_{f-1}\]
with a $\bZ_{p^f}$-endomorphism structure $\iota$. We note that due to the existence of the Frobenius structure we have
\[\Gr((\widehat V,\widehat\nabla,\widehat\Fil)_{f-1})\simeq (\widehat E,\widehat\theta)_0.\]
Thus we can reconstruct $(E,\theta)$ from the $0$-th eigen component $(\widehat V,\widehat\nabla)_{f-1}$.

By the construction of $\mE$, its underlying de Rham bundle is just the generic fiber $(\widehat V_K,\widehat\nabla_K)$. We can also taking the $0$-th eigen component, which is just $(\widehat V_K,\widehat\nabla_K)_{f-1}$. Since the modulo $p$-reduction of $(\widehat V,\widehat\nabla)_{f-1}$ is stable, up to isomorphism $(\widehat V_K,\widehat\nabla_K)$ has a unique integral extension, which is just $(\widehat V,\widehat\nabla)_{f-1}$.

In summary, from $\mE$ we can reconstruct the Higgs bundle as follows:
\begin{itemize}
\item find the $0$-th eigen component of underlying de Rham bundle of $\mE$;
\item find an integral extension of the de Rham bundle in the first step.
\item take grading (the Hodge filtration is unique, due to \autoref{thm_ClassfyR2PdE}) and modulo $p$, one gets the original Higgs bundle $(E,\theta)$.
\end{itemize}
Now starting from $\Frob_k(\mE)$ and following steps as above, one then get the Higgs bundle $\Frob_k(E,\theta)$.
\end{proof}

\begin{remark}
Due to the existence of the Hodge filtration, there is no Frobenius action on the intermediate sets $[\PHDFhf(W(k_f))]$ and $[\MFh(W(k))_{\bZ_{p^f}}]$. If one forgets the Hodge filtrations in Fontaine-Faltings modules, then he will get some ``parabolic $F$-crystals'', on which there should exist an action of $\Frob_k$. And the natural maps between them should preserves the Frobenius structure.
\end{remark}

\section{\bf $p$-to-$\ell$ companion} \label{sec_main_p_to_l_bijections}

In this section, the aim is to construct the injections
\begin{equation}\label{equ_main_sec_p_ell}
\xymatrix@C=2.2cm{
[\FIsoch(k'_2)_{\bQ_{p^f}}^{\rm triv}] \ar@{>->}[r]^-{\autoref{mthm_FIsoc2Loc}} & [\Loch(k'_2)] \ar@{>->}[r]^-{\autoref{mthm_LocSys_equivRelation}} & \Loch(\overline{k})^{\Frob_{k_2'}}\\
}
\end{equation}
And show that the image of the following composition
\begin{equation}\label{equ_main_sec_p_ell_II}
\xymatrix{
\High(k) \ar@{>->}[r]^-{\eqref{equ_main_sec_isoc_II}}& [\FIsoch(k'_2)_{\bQ_{p^f}}^{\rm triv}] \ar@{>->}[r]^-{\eqref{equ_main_sec_p_ell}} & \Loch(\overline{k})^{\Frob_{k_2'}}
}
\end{equation}
is contained in $\Loch(\overline{k})^{\Frob_{k}}$, see \autoref{mthm_HIGk2Lock}. Finially, Yu's formular(\autoref{thm_Yu}) for a numeric Simpson correspondence implies that the composition in \eqref{equ_main_sec_p_ell_II} is a genuine Simpson correspondence
\[\High(k) \xrightarrow[\autoref{thm_genuineSimCorr}]{1:1} \Loch(\overline{k})^{\Frob_{k}}.\]

Here $[\Loch(k'_2)]$ and $\Loch(\overline{k})^{\Frob_{k_2'}}$ are defined in \autoref{subsec_locsys_projline} and \autoref{subsec_locsys_projline_diff_const}. Roughly speaking,
\begin{itemize}
\item $[\Loch(k'_2)]$ is the set of isomorphic classes of rank $2$ tame $\ell$-adic local systems $\bL$ on the punctured projective line $\bP^1_{k_2'} \setminus \{0,1,\lambda,\infty\}$ with prescribed eigenvalues of the local monodromies modulo an equivalence defined in  \autoref{def_diff_const_locsys}.
\item $\Loch(\overline{k})^{\Frob_{k_2'}}$ is the set of all isomorphic classes of rank $2$ tame $\ell$-adic local systems $\bL$ fixed by the Frobenius $\Frob_{k_2'}$ on the punctured projective line $\bP^1_{\overline{k}} \setminus \{0,1,\lambda,\infty\}$ with prescribed eigenvalues of the local monodromies.
\end{itemize}

As a consequence, we have
\begin{theorem}\label{thm_unramified}
Suppose $\lambda$ is supersingular.
\begin{enumerate}[$(1).$]
\item  The trace field of any $F$-isocrystal in $\FIsoch(k'_2)_{\bQ_{p^f}}^{\rm triv}$ is unramified above $p$.
\item The trace field of the isocrystal attached to an Fontaine-Faltings module in $\MFh(W(k'_2))_{\bZ_{p^f}}^{\rm cy}$ is unramified.
\end{enumerate}
\end{theorem}
\begin{proof} (2) follows from (1) and the bijection. We only need to show (1). For any $F$-isocrystal $\mE\in \FIsoch(k'_2)_{\bQ_{p^f}}^{\rm triv}$, denote by $\EK$ the trace field and denote by $\bL$ the associated $\ell$-adic local system. For any $\sigma\colon \overline{\bQ}_\ell \rightarrow \overline{\bQ}_\ell\cong \overline{\bQ}_p$, we have the $p$-adic companion $\mE^\sigma$ and $\ell$-adic companion $\bL^\sigma$. Since taking companion preserves the eigenvalues of local monodormies, $\bL^\sigma$ is also contained in $\Loch(\overline{k})^{\Frob_{k}}$. Thus $\mE^\sigma$ is contained in $\FIsoch(k'_2)_{\bQ_{p^f}}^{\rm triv}$ and $\sigma(\EK)\subset \bQ_{p^f}$ for any $\sigma$.  Hence $\EK$ is unramified above $p$.
\end{proof}

\subsection{$\ell$-adic local systems}
\subsubsection{The character $\chi_{1/2}$}
Recall that the Galois group $G_{\bF_{p}}=\Gal(\overline{\bF}_{p}/\bF_{p})$ is isomorphic to $\widehat{\bZ}$ with a topological generator $\sigma$. Denote by
\[\chi_{1/2}\]
the $\bQ_\ell$-character of the subgroup $G_{\bF_{p^2}}=\Gal(\overline{\bF}_{p}/\bF_{p^2})$
given by
\[\chi_{1/2}(\sigma^{2}) = p.\]
Clearly $\chi_{1/2}^2=\bQ_\ell(1)$ is just the cyclotomic character.

\subsubsection{The change of characteristic polynomials under twisting by the character $\chi_{1/2}$}

Let $k$ be a finite field with cardinality $p^h$ containing $\bF_{p^2}$. The character $\chi_{1/2}$ can be restricted on the absolute Galois group $G_k$ of $k$. Let $\bL$ be a $\Qbar_\ell$-representation of the absolute Galois group $G_k$ of $k$.

Denote by $P(\bL,t)$ the characteristic polynomial and by $\tr(\bL)$ the trace of $\sigma^h$ acting on $\bL$. Then
\begin{lemma} let $\bL$ be a local system of rank $r$. Then
\[P(\bL\otimes \chi_{1/2},t) = p^{rh/2}P(\bL,p^{-h/2}t) \quad \text{and} \quad \tr(\bL\otimes\chi) = p^{h/2}\tr(\bL).\]
\end{lemma}

\subsection{$\ell$-adic local systems over punctured projective line and Yu's formula}

\subsubsection{$\ell$-adic local systems over punctured projective line} \label{subsec_locsys_projline}

Denote by $\Loch(k)$ the set of isomorphic classes of rank $2$ tame $\ell$-adic local systems $\bL$ on the punctured projective line $\bP^1_{k} \setminus \{0,1,\lambda,\infty\}$ with following prescribed eigenvalues of the local monodromies:
\begin{itemize}
\item the local monodromies around $\{0,1,\lambda\}$ is unipotent;
\item the local monodromy around $\infty$ is quasi-unipotent and has double eigenvalue $-1$.
\end{itemize}

\subsubsection{An equivalence relation on $\Loch(k)$} \label{subsec_locsys_projline_diff_const}

\begin{definition}\label{def_diff_const_locsys}
 Let $\bL$ and $\bL' \in \Loch(k)$. We call they are differed by a character, if there exists a character $\chi$ of the absolute Galois group $\Gal(\overline{k}/k)$ such that
\[\bL' = \bL \otimes \chi.\]
Denote by $[\Loch(k)]$ the set of all equivalent classes.
\end{definition}

\begin{lemma} All local systems in $\Loch(k)$ are geometrically irreducible.
\end{lemma}
\begin{proof} Suppose not. Then there exists some rank-1 sub local system $\bW$ of the geometric part of some $\bL$ in $\Loch(k)$. Then the local monodromy matrix of $\bW$ around $\{0,1,\lambda\}$ are all equal to 1, and around $\{\infty \}$ is $-1$. As the four generators $\{\gamma_0,\gamma_1,\gamma_\lambda,\gamma_\infty \}$ of the geometric fundamental group of $\bP^1 \setminus \{0,1,\infty,\lambda \}$ around the 4 punctures have one relation
\[\gamma_0 \cdot \gamma_1,\cdot \gamma_\lambda\cdot \gamma_\infty=e\]
we obtain\[1 \cdot 1 \cdot 1 \cdot (-1)=1,\]
which leads a contradiction.
\end{proof}

\begin{lemma} Let $\bL$ and $\bL' \in \Loch(k)$. Then $\bL$ and $\bL'$ are differed by a character if and only if they have isomorphic geometric parts (i.e., restrictions on $U_{\overline{k}}$).
\end{lemma}

\begin{proof} The ``only if'' part is trivial. Suppose there have the same geometric parts. Let $\rho$ and $\rho'$ be the representations of $\pi_1(U_k)$ associated to two equivalent local systems $\bL$ and $\bL'$. Then by assumption
\[\rho\mid_{\pi_1(U_{\overline k})} = \rho'\mid_{\pi_1(U_{\overline k})}.\]
Assume $\gamma \in \Gal(\overline{k}/k)$. Then for any two lifting $\widehat\gamma$ and $\widehat{\gamma}'$ of $\gamma$ in $\pi_1(\bP^1_k\setminus\{0,1,\lambda,\infty\})$, $\widehat\gamma^{-1}\cdot \widehat\gamma' \in \pi_1(U_{\overline{k}})$. So $\rho(\widehat\gamma^{-1}\cdot \widehat\gamma') = \rho'(\widehat\gamma^{-1}\cdot \widehat\gamma')$. This implies
\[\rho'(\widehat\gamma)\cdot \rho(\widehat{\gamma})^{-1} = \rho'(\widehat\gamma')\cdot \rho(\widehat{\gamma}')^{-1}\]
In other words, the value $\rho'(\widehat\gamma')\cdot \rho(\widehat{\gamma}')^{-1}$ does not depend on the choice of the lifting, denote it by $\chi(\gamma)$. We only need to show that $\chi$ is a character of $\Gal(\overline{k}/k)$.

Since $\pi_1(U_{\overline{k}})$ is a normal subgroup of $\pi_1(U_k)$, $\widehat\gamma^{-1}g\widehat\gamma \in \pi_1(U_{\overline{k}})$ for any $g\in \pi_1(U_{\overline{k}})$. So $\rho'(\widehat\gamma^{-1}g\widehat\gamma) = \rho(\widehat\gamma^{-1}g\widehat\gamma)$ and $\rho'(g)=\rho(g)$. Thus
\[\chi(\gamma) \rho(g) = \rho(g) \chi(\gamma).\]
By Schur's lemma,
\begin{equation}\label{equ:schur}
\chi(\gamma)\in \End(\rho\mid_{\pi_1(U_{\overline{k}})}) \cong \Qbar_\ell.
\end{equation}
Next, we only need to show $\chi$ is multiplicative. For any two elements $\gamma_1$ and $\gamma_2$ in $\Gal(\overline{k}/k)$, choose liftings $\widehat{\gamma}_1$ and $\widehat{\gamma}_2$ respectively. Then
\begin{equation*}
\begin{split}
\chi(\gamma_1\gamma_2) := & \rho'(\gamma_1\gamma_2)\cdot \rho(\gamma_1\gamma_2)^{-1} \\
= & \rho'(\gamma_1)\cdot\chi(\gamma_2)\cdot\rho(\gamma_1)^{-1} \\
\overset{\eqref{equ:schur}}{=}& \rho'(\gamma_1) \cdot\rho(\gamma_1)^{-1} \cdot \chi(\gamma_2)\\
=&\chi(\gamma_1)\cdot\chi(\gamma_2). \qedhere
\end{split}
\end{equation*}
\end{proof}

\begin{corollary} \label{mthm_LocSys_equivRelation}
One has an injection
\begin{equation*}
[\Loch(k)] \hookrightarrow \Loch(\overline{k})^{\Frob_k}.
\end{equation*}
\end{corollary}

\subsubsection{Yu's formula}
By Drinfeld and Deligne the set of $\Loch(\overline{k})^{\Frob_k}$ is finite. However, it is not clear that how the number depends on $q:=\#k$ precisely. Very recently Hongjie Yu \cite{Yu23} has solved Deligne's conjecture on counting $\ell$-adic local systems in terms of parabolic Higgs bundles. His general theorem applying to our special case turns out:
\begin{theorem}[Hongjie Yu{\cite{Yu23}}] \label{thm_Yu}
\[\#\High (k) = \# \Loch(\overline{k})^{\Frob_k}.\]
\end{theorem}
\begin{remark}
When $\lambda$ is supersingular, we will show this numeric Simpson correspondence in fact underlies a genuine Simpson correspondence, see \autoref{thm_genuineSimCorr}.
\end{remark}

\subsection{Abe's theorem on Deligne's $p$-to-$\ell$ companion}

In this subsection, we choose a prime $\ell\not=p$ and fix an isomorphism $\phi\colon \overline {\bQ}_p\simeq \overline {\bQ}_\ell$. Our aim is to construct a natural injection
\[\High(k) \hookrightarrow \Loch(\overline{k})^{\Frob_k}.\]
The construction is diagrammatically sketched below:
\begin{equation*} \footnotesize
\xymatrix@R=1cm@C=1cm@M=4mm{
[\PHDFhf(W(k))]
\ar@{>->}[d]^{\autoref{mthm_PHDFf2MFf}}_{k\subseteq k' \atop \bF_{p^f}\subseteq k'} \ar[r]^{\autoref{mthm_PHDFf2PHIGf}}_{1:1}
& \PHighf(W(k))
\ar[r]^-{\lambda=\text{supersingular}}_-{1:1\atop \autoref{mthm_PHIGk2PHIGW}}
& \PHighf(k)
\ar[r]_{1:1 \atop \autoref{mthm_PHIGf2PHIG}}^{(\#k+1)!\mid f}
& \PHigh(k)
\ar[d]_{\lambda=\text{supersingular} \atop \autoref{mthm_HIG2PHIG}}^{1:1} \\
[\MFh(W(k'))_{\bZ_{p^f}}]
\ar@{>->}[rr]^{\autoref{mthm_FF2Isoc}}
\ar@{>->}[d]_{\autoref{thm_cyclDeterminant_FFM}}
\ar@{>->}[dr]
&& [\FIsoch(k')_{\bQ_{p^f}}]
\ar@{>->}[dd]^{\autoref{mthm_Fisok2Fisokn}}
& \High(k)
\ar@{>-->}[dd]^-{\exists}_-{\autoref{mthm_HIGk2Lock}} \\
[\MFh(W(k_2'))_{\bZ_{p^f}}^{\rm cy}]
\ar@{>->}[d]^{\autoref{mthm_FF2Isoc_classes}}
\ar@{>->}[r]
& [\MFh(W(k_2'))_{\bZ_{p^f}}]
\ar@{>->}[dr]^{\autoref{mthm_FF2Isoc}}
&\\
[\FIsoch(k'_2)_{\bQ_{p^f}}^{\mathrm{triv}}]
\ar@{>->}[rr] \ar@{>->}[d]_{p\text{-to-}\ell}^{\autoref{mthm_FIsoc2Loc}}
&
&[\FIsoch(k'_2)_{\bQ_{p^f}}]
& \Loch(\overline{k})^{\Frob_k}
\ar@{>->}[d]^{k\subset k'_2}\\
[\Loch(k'_2)]
\ar@{>->}[rrr]^{\autoref{mthm_LocSys_equivRelation}}
&
&
& \Loch(\overline{k})^{\Frob_{k'_2}} \\
}
\end{equation*}

\subsubsection{$p$-to-$\ell$ companion over projective line}

By applying $p$-to-$\ell$ companion, which was conjectured by Deligne and was proven by Abe in \cite{Abe18}, to an overconvergent $F$-isocrystal $\mE$ in $\FIsoch(k'_2)_{\bQ_{p^f}}^{\rm triv}$, then one gets a rank-2 $\ell$-adic irreducible local system $\bL_{\mE}$ on $(\bP_{k_f}^1\setminus \{0,1,\lambda,\infty\})$ with trivial determinant.

\begin{proposition} \label{mthm_FIsoc2Loc}
The local system $\bL_\mE$ is contained in $\Loch(k'_2)$. The $p$-to-$\ell$ companion induces us an injection
\[[\FIsoch(k'_2)_{\bQ_{p^f}}^{\rm triv}] \rightarrow [\Loch(k'_2)].\]
\end{proposition}

\begin{proof}
By \cite[ Corollary 2.5.4.]{Ked07}, the local system $\bL_\mE$ is docile along $0,1,\infty$, since $\mE$ has unipotent monodromy along $0,1,\infty$.

By tensoring the $F$-isocrystal and $\ell$-adic character associated to the rank-$1$ parabolic Fontaine-Faltings module $\mO(1/2(\infty)-1/2(0))$ respectively, one can shift the parabolic structure from $\infty$ to $0$. Thus by direct calculation, we get the eigenvalue of the monodromy is $-1$ and the exponents of the residue is $1/2$.

Next, we need to show the injectivity. Suppose two $F$-isocrystals $\mE,\mE'\in \FIsoch(k'_2)_{\bQ_{p^f}}^{\rm triv}$ have equivalent $\ell$-adic companions. In other word, $\bL_\mE$ and $\bL_{\mE'}$ are differed by a character of $\Gal(\bark/k'_2)$. Hence there exists a finite extension $k_2''$ of $k_2'$ such that the base change of $\bL_\mE$ and $\bL_{\mE'}$ from $k_2'$ to $k_2''$ are coincide with each other. Now by the bijection of $p$-to-$\ell$ companion, one gets
\[\mE \simeq \mE' \in \FIsoch(k''_2)_{\bQ_{p^f}}^{\rm triv}.\]
In general, one cannot descent this isomorphism to an isomorphism in $\FIsoch(k''_2)_{\bQ_{p^f}}^{\rm triv}$, this is because the Frobenius structure is non-linear. But they underlying overconvergent isocrystals are isomorphic to each other, as these overconvergent isocrystals are irreducible. Thus they Frobenius structure is differed by a constant.
\end{proof}

\begin{proposition} \label{mthm_HIGk2Lock}
Suppose $\lambda$ is supersingular, by composing the morphisms in diagram ahead this section, we construct an injective map
\[\High(k) \hookrightarrow \Loch(\overline{k})^{\Frob_{k'_2}}.\]
More finely, the image of this map is contained in $\Loch(\overline{k})^{\Frob_{k}}$, and one get an injection
\[\High(k) \hookrightarrow \Loch(\overline{k})^{\Frob_{k}}.\]
\end{proposition}
\begin{proof} Let $(E,\theta)\in \High(k)$ and let $\mE$ be the associated $F$-isocrystal in $\FIsoch(k'_2)^{\rm triv}_{\bQ_{p^f}}$. By \autoref{mthm_Higgs2FIsoc_Frob_Preserved}, $\mE$ is invariant under the action of $\Frob_{k}$. Since the $p$-to-$\ell$ companion is preserved by $\Frob_k$, the corresponding $\ell$-adic local system is also invariant under the action of $\Frob_k$.
\end{proof}

By Yu's formula \autoref{thm_Yu} for numeric Simpson correspondence, we gets a genuine Simpson correspondence.
\begin{corollary} \label{thm_genuineSimCorr}
Assume that $\lambda$ is supersingular. Then the injection
\[\High(k) \rightarrow \Loch(\overline{k})^{\Frob_{k}}\]
is actual a bijection.
\end{corollary}

\section{\bf Constructing family of abelian varieties in positive characteristic and lifting the Hodge filtration to characteristic zero} \label{sec_main_family_over_k}

In this section, we use Drinfeld theorem on Langlands correspondence over function field of characteristic $p$ to show that any local system in $\Loch(k'_2)^{\Frob_k}$ comes from family of abelian varieties and the Hodge filtration attached to this family can be lifted to characteristic zero in supersingular case. More precisely, the following is the statement of the main result.
\begin{theorem} \label{thm_construction_family_mod_p}
Suppose $\lambda\in W=W(k)$ is supersingular. Consider $(\bP^1_W,\{0,1,\lambda,\infty\})$ the projective line with four marked points. Then for a given local system $\bL\in \Loch(k_2')^{\Frob_k}$ with cyclotomic determinant, there exists an abelian scheme
\[f: A \to U_{k_2'}\]
of $\GL_2$-type over $U_{k'_2}:=\bP^1_{k'_2}\setminus\{0,1,\lambda,\infty\}$ such that
\begin{enumerate}[$(1).$]
\item all eigen sheaves $\bL_i$'s of $R^1f_*\overline{\bQ}_\ell$ are contained in $\Loch(k_2')^{\Frob_k}$ and $\bL_1=\bL$;
\item the Dieudonn\'e crystal attached to $f$ underlies a parabolic Fontaine-Faltings module $(V,\nabla,\Fil,\Phi)^\FF$;
\footnote{The Dieudonn\'e crystal attached to $f$ has a realization over $U_{W(k'_2)}=\bP^1_{W(k'_2)}\setminus\{0,1,\lambda,\infty\}$, which is a de Rham bundle $(V,\nabla)$ together with a Frobenius semilinear endomapping $\Phi$. The crystal underlying a parabolic Fontaine-Faltings module $(V,\nabla,\Fil,\Phi)^\FF$ means that there is an isomorphism
\[(V,\nabla,\Phi) \cong (V,\nabla,\Phi)^\FF\mid_{U_{W(k'_2)}}.\]}
\item The Hodge filtration
\footnote{Consider the realization of the Dieudonn\'e crystal over $U_{k'_2}$, which is isomorphic to the relative de Rham cohomology $R^1_{dR}f_*O_A$ and is also the modulo $p$ reduction of $(V,\nabla,\Phi)\otimes_Wk$, the realization of the crystal over $U_{W(k'_2)}$. The relative differential forms define a natural Hodge filtration on $R^1_{dR}f_*O_A\cong (V,\nabla)\otimes_Wk$. This filtration is simply called the \emph{Hodge filtration attached to $f$.}}
attached to $f$ coincides with the modulo $p$ reduction of $\Fil^\FF$. Consequently, the Hodge filtration can be lifted to characteristic zero.
\end{enumerate}
\end{theorem}

A key ingredient in the proof of our main results is the following \autoref{thm_Drinfeld_GL2}, which is a byproduct of Drinfeld's first work on the Langlands correspondence for $\GL_2$ \cite{Dri77}. We first record a setup.
\begin{setup}\label{setup:curve_finite_field}
Let $p$ be a prime number and let $q=p^a$. Let $C$ be a smooth, affine, geometrically irreducible curve over $\bF_q$ with smooth compactification $\overline{C}$. Let $Z:=\overline{C}\setminus C$ be the reduced complementary divisor.
\end{setup}

\begin{theorem}\label{thm_Drinfeld_GL2}(Drinfeld) Notation as in \autoref{setup:curve_finite_field} and let $\bL$ be a rank 2 irreducible $\Qlbar$ sheaf on $C$ with determinant $\Qlbar(1)$. Suppose $\bL$ has infinite local monodromy around some point at $\infty\in Z$. Then $\bL$ comes from a family of abelian varieties in the following sense: let $\EK$ be the field generated by the Frobenius traces of $\bL$ and suppose $[\EK:\bQ]=h$. Then there exists an abelian scheme
\[
\pi_C\colon A_{C}\rightarrow C
\]
of dimension $h$ and an isomorphism $\EK\cong \End_{C}(A_C)\otimes\bQ$, realizing $A_C$ as a $\GL_2(\EK)$-type abelian scheme, such that $\bL$ occurs as a summand of $R^1(\pi_C)_*\Qlbar$. Moreover, $A_{C}\rightarrow C$ is totally degenerate around $\infty$.
\end{theorem}

\subsection{Proof of \autoref{thm_construction_family_mod_p}}

In the following, we use \autoref{thm_Drinfeld_GL2} to give the proof of \autoref{thm_construction_family_mod_p}. We first give the idea and list the steps of the proof:
\begin{enumerate}[Step $1):$]
\item Construct a family by using Drinfeld theorem. This family is not that we want. One needs to modify this family by an isogeny to ensure (2) and (3).
\item According bijections in \autoref{sec_main_p_to_l_bijections}, the overconvergent $F$-isocrystal attached to Drinfeld family comes from a Fontaine-Faltings module.
\item From this Fontaine-Faltings module, one gets an isogeny of the Dieudonn\'e crystal attached to the family. This isogeny will induces an isogeny of the original family. This family satisfies (2) by construction.
\item According the strong $p$-divisibility condition in Fontaine-Faltings module, the new family also satisfies (3).
\end{enumerate}

\subsubsection*{{\bf Step 1).} Drinfeld family attached to the local system.}
Given a local system $\bL\in \Loch(k_2')^{\Frob_k}$ with cyclotomic determinant, the restriction of $\bL$ to the geometric fundamental group is irreducible with infinite local monodromy at least on one puncture. Denote by $\EK$ the trace field of $\bL$. By applying Drinfeld's \autoref{thm_Drinfeld_GL2} to $\bL$, there exists an abelian scheme of $\GL_2(\EK)$-type
\[\pi \colon A \to U_{k_2'}\]
over the punctured projective line at $\{0,1,\infty,\lambda\}$ and with $\bL$ being an eigen summand of the associated local system. In the following, we show there is an isogeny of this family satisfying all requirements in \autoref{thm_construction_family_mod_p}

Let $\bV\coloneqq R^1_\et \pi_* \Qbar_{\ell}= \bigoplus_{i=1}^g \bL_i$ be the eigen decomposition with $\bL_1=\bL$. Then all $\bL_i$'s are contained in $\Loch(k_2')^{\Frob_k}$ with cyclotomic character. This is because all of them are conjugate to $\bL$.

\subsubsection*{{\bf Step 2).} Eigen decomposition of the attached Dieudonn\'e crystal and its realization.}
Denote by $\bD(\pi)$ the Dieudonn\'e crystal over $\bP^1_{k_2'}\setminus\{0,1,\infty,\lambda\}$ attached to $\pi$, on which the ring $\mO_\EK$ naturally acts. Forgetting the Verschiebung structure in the Dieudonn\'e crystal, then one gets an $F$-crystal over $U_{k_2'}$, which is overconvergent by \cite[3.17]{Tri08}. Since the trace field $\bE$ is unramified above $p$, the $F$-crystal has decomposition of rank $2$ eigen sub ones via the action of $\EK$ after extending the coefficient from $\bZ_p$ to $\bZ_{p^f}$ for some sufficiently large $f$
\[\bD(\pi) = \bigoplus_{i=1}^{g}\mE_i.\]
The eigen decomposition induces that for they realizations over $U_{W(k_2')}=\bP^1_{W(k_2')}\setminus\{0,1,\infty,\lambda\}$, an decomposition of de Rham bundles endowed with an Frobenius structures
\[(V,\nabla,\Phi)_U = \bigoplus_{i=1}^g (V,\nabla,\Phi)_{i,U}.\]

Since $\{\mE_i\}$ and $\{\bL_i\}$ are all coming from the same family, they are all companion to each other under the $p$-to-$\ell$ companion. Thus by the discuss of the bijections in \autoref{sec_main_p_to_l_bijections}, all eigen components $\mE_i$ come from Fontaine-Faltings modules in $\MFh(W(k_2'))^{\rm cy}_{\bZ_{p^f}}$.
In an explicit way, there exists an Fontaine-Faltings module $(V,\nabla,\Fil,\Phi)^{\rm FF}$ such that
\[(V,\nabla,\Phi)_U\otimes \bQ \cong (V,\nabla,\Phi)^{\rm FF} \mid_{U} \otimes \bQ.\]
After extending the coefficient from $\bQ_p$ to $\bQ_{p^f}$ on both sides, one gets $\EK$-eigen components $(V,\nabla,\Phi)_{i,U}$ and $(V,\nabla,\Phi)^{\rm FF}_{i}$, which are endowed with the natural $\bZ_{p^f}$-endomorphism structures $\iota_{i,U}$ and $\iota_i^{\rm FF}$. By choosing suitable order, we may assume (for each $i$)
\[(V,\nabla,\Phi,\iota)_{i,U}\otimes \bQ \cong (V,\nabla,\Phi,\iota)^{\rm FF}_{i}\mid_{U} \otimes \bQ\]
the left and right sides corresponding $\mE_i$ and $M_i$ respectively. In other words, $(V,\nabla,\Phi)_{i,U}$ can be parabolically extended to boundary after tensoring $\bQ$.

By multiplying suitable power of $p$, we may assume under the above isomorphism, one has
\[(V,\nabla,\Phi,\iota)_{i,U} \subseteq (V,\nabla,\Phi,\iota)^{\rm FF}_{i}\mid_{U}.\]

\subsubsection*{{\bf Step 3).} Verschiebung and $p$-isogeny}

By extending the coefficient, one gets a Verschiebung on $(V,\nabla,\Phi)\otimes \bQ_p$. By restricting onto the new lattice, one gets a Verschiebung structure $\mV$ on $(V,\nabla,\Fil,\Phi)^\FF$. By adding back the Verschiebung structures on both sides, we gets an isogeny between two Dieudonn\'e crystals. By \cite[Lemma 2.13]{KrPa22}, there exists an isogenous abelian scheme over $U$, which we just call again $f: A\to U_{k_2'}$ such its $F$-crystal is equal to that of the Fontaine-Faltings module.

\subsubsection*{{\bf Step 4).} Hodge filtration and its lifting}
By taking relative differential $1$-forms attached to $f$ one gets the Hodge filtration on $(V,\nabla)^\FF\otimes \bF_{q^{2f}}$ given by
\[E^{'1,0}\coloneqq R^0f'_*\Omega^1 _{A'/\bP^1} (\log \Delta) \subset (V,\nabla)^\FF\otimes \bF_{q^{2f}}=R^1_{dR}f_*(\Omega^\bullet_{A/\bP^1}(\log \Delta),d),\]
which is a rank-$g$ sub bundle. The Hodge filtration coincide with that coming from the family. In other words, $\Fil$ is a filtration lifts $E^{'1,0}$. This is because the relative Frobenius $\Phi$ on the Fontaine-Faltings module satisfies the strong $p$-divisible condition with respect to the filtration $\Fil$, the Hodge filtration $E^{1,0}$ coincides with the modulo $p$ reduction of the filtration on the Fontaine-Faltings module.

\section{\bf Lifting abelian scheme from characteristic $p$ to characteristic zero by Grothendieck-Messing-Kato logarithmic deformation theorem} \label{sec_main_lifting}

Let $\lambda$ be an algebraic number not equal to $0$ and $1$ and let $L\ni \lambda$ be a number field containing it. Assume $\frakp$ is a finite unramified place of $L$ above $p>3$ such that $\lambda$ is $\frakp$-adic integral, $\lambda\not\equiv 0,1 \pmod{\frakp}$ and $\lambda$ is supersingular in the sense of \autoref{def_supersingular}. Denote by $k=k_\frakp$ the residue field at $\frakp$, then we have the natural embedding map $L\hookrightarrow L_\frakp = W(k)[\frac1p]$.

Let $(\overline{E},\overline{\theta})$ be a Higgs bundle in $\High(k)$. According the bijections given in \autoref{sec_main_p_to_l_bijections}, we get the uniquely periodic lifting $(E,\theta)$ of $(\overline{E},\overline{\theta})$ contained in $\PHigh(W(k))$, and an $\ell$-adic local system
$\bL\in \Loch(k'_2)$ with cyclotomic determinant. Let $f\colon A\rightarrow U_{k_2'}$ be an abelian scheme constructed in \autoref{thm_construction_family_mod_p}.
 In this section, we show that the Higgs bundles are motivic and the family $f$ lifts to an arithmetic family of $\GL_2$-type. The following is the precise statement of the main result.

\begin{theorem}  \label{thm_family_from_W_to_number_field}
Let $\lambda$ be supersingular given as above.
For any Higgs bundle $(\overline{E},\overline{\theta}) \in \High(k)$, denote by $(E,\theta)\in \PHigh(W(k))$ the unique periodic lifting. Then after enlarging the number field $L$, there is a family of abelian variety over $\bP^1_L$ of $\GL_2$-type such that $(E,\theta)$ is a direct summand of the Higgs bundle attached to this family.
\end{theorem}

We first sketch the main steps of the proof of \autoref{thm_family_from_W_to_number_field}.
In the upcoming steps, we will provide a detailed plan to lift the family $f$ to the one required by \autoref{thm_family_from_W_to_number_field}

\begin{enumerate}[Step $1):$]

\item First,we modify the family $f$, by applying Zarhin's trick and base changing the family along a covering $\pi: C\to \bP^1$. Then one gets a semistable family $f^{4,4}_{\pi_k}$ with full level-$3$ structure and a principal polarization structure. This induces a log classifying mapping
\[\overline{\varphi}_k\colon C_k \rightarrow \overline{\mA}_{8g,3}\]
from the log base curve $(C_k,D_k)$ to the log moduli scheme $(\bar{\mA}_{8g,3},\infty)$ of principle polarized abelien varieties of dimension $8g$ with level-3 structure constructed by Faltings-Chai. This was done in \autoref{sec_sub_classifying_mapping}.

\item From the family $f^{4,4}_{\pi_k}$ of abelian varieties, one gets a Dieudonn\'e crystal $(V,\nabla,\Phi,\mV)_{f^{4,4}_{\pi_k}}$ and the Hodge filtration of the relative differential $1$-forms attached to the family. The Hodge filtration can be lifted to characteristic zero, denoted by $F_{f^{4,4}_{\pi_k}}$,by applying \autoref{thm_construction_family_mod_p}.

Next, we show the lifting filtration is compatible with the polarization. In other words, the isomorphism $\iota$ of the principle polarization of $f^{4,4}_{\pi_k}$ sends $F_{f^{4,4}_{\pi_k}}$ to the sub bundle $(V_{f^{4,4}_{\pi_k}}/F_{f^{4,4}_{\pi_k}})^\vee$ of the the Dieudonn\'e crytal of the dual abelian scheme
$(f^{4,4}_{\pi_k})^t$, which is the lifting of the Hodge filtration of the relative differential 1-forms attached to
$(f^{4,4}_{\pi_k})^t.$

\item
Next, by applying Grothendieck-Messing-Kato \autoref{thm_lifting_family_to_W}, we show that $f^{4,4}_{\pi_k}$ lifts to an abelian scheme over $W(k)$, this was done in \autoref{thm_lift_classfying_mapping}.

\item Next, we show the abelian scheme over $W(k)$ can be descending to a number field. This was done in \autoref{thm_familyWittring_to_familyNumberring}.

\item Next, by applying Weil restriction, we descend the family from the curve $C$ to the projective line and obtain an abelian scheme $h: B\to \bP^1_L$ with bad reduction on $\{0,1,\lambda,\infty\}$ of type-$(1/2)_\infty$ and such that $(E,\theta)$ has descending over an algebraic number field $L'$,which is a direct summand of the Higgs bundle attached to $h$.

\item Finally, by applying \autoref{thm_Simpson}, we prove \autoref{thm_family_from_W_to_number_field} the main result in this section. More explicitly, there exists a factor $h'$ of $h$ such that the abelian scheme $h'$ is of $\text{GL}_2$-type and $(E,\theta)_{\bL}$ is an eigen Higgs bundle attached to $h'$.
\end{enumerate}

In the following we give construction step by step.

\subsection{{\bf Step 1).} Classifying mapping} \label{sec_sub_classifying_mapping}
In order to get a classifying map from the base into a fine moduli space of principal polarized abelian varieties, we need to add a level structure and a principal polarization to the family.

\subsubsection{level structure} \label{sec_subsub_level}

Firstly, by base change, we add a level structure.
\begin{lemma} \label{thm_add_level}
By enlarging $k$, there exists a finite covering between two projective smooth curve over $k$
\[\pi_k\colon C_k \rightarrow \bP_k^1\]
which is \'etale over $U_k$ such that the pullback family of $f$
\[f_{\pi_k}\colon A_{\pi_k} \rightarrow \pi_k^{-1}(U_k)\]
has full $3$-level structure.
\end{lemma}
\begin{proof}
Let $k(t)$ be the function field of the projective line and $A_\eta$ is the generic fiber of $f$. Then by adding the coordinates of all torsion points of $A_\eta$ of order $3$ to $k(t)$, one gets a separable finite field extension of $k(t)$. In particular, one gets a curve $C_{\kappa}$ over some finite extension $\kappa$ of $k$ and a finite morphism $\pi_k\colon C_\kappa\rightarrow \bP^1_k$ such that $\pi$ is \'etale over $U_k$. By enlarging the field $k$, we may assume $\kappa=k$ and $\pi_k$ is a $k$-morphism between two proper smooth curves over $k$. The curve $C_k$ satisfies our requirement clearly.
\end{proof}

By the smoothness, we find and fix a lifting of $\pi_k$ over $W(k)$
\[\pi\colon C\rightarrow \bP^1_{W(k)}.\]
Denote by $D\subset C$ the pullback divisor of $\{0,1,\lambda,\infty\}$ under $\pi$. Then
\[\pi_k^{-1}(U_k) = C_k\setminus D_k,\]
where $D_k = \pi_k^{-1}(\{0,1,\lambda,\infty\})$.

\subsubsection{Zarhin's trick} \label{sec_subsub_Zarhin}

Let $f_{\pi_k}\colon A_{\pi_k} \rightarrow C_k$ be the pullback family given as in \autoref{thm_add_level}. By Zarhin trick, the fiber product
\[ f^{(4,4)}_{\pi_k} : A^{4,4}_{\pi_k}:=(A_{\pi_k} \times A_{\pi_k}^t)^4 \to C_k\setminus D_k.\]
carries a principle polarization
\[\iota: A^{(4,4)}_{\pi_k} \xrightarrow{\simeq} \left(A_{\pi_k}^{(4,4)}\right)^t.\]

\subsubsection{Faltings-Chai's compactification}
By Faltings-Chai Theorem \cite{FaCh90}, there exists a fine arithmetic moduli space $\mA_{8g,3}$ of principle polarized abelian varieties with level-$3$ structure, which is smooth over $\bZ[e^{{2i \pi \over 3}},1/3]$. The moduli space carries
an universal abelian scheme
\[\pi_{univ} \colon \mE_{univ} \rightarrow \mA_{8g,3}.\]

Further more, there exists a smooth Toroidal compactification $\overline{\mA}_{8g,3} \supset \mA_{8g,3}$ over $\bZ[e^{{2i \pi \over 3}},1/3]$ and a smooth compactification of the universal abelian scheme
\[\overline{\pi}^{uni}: \overline{\mE}_{univ} \to \overline{\mA}_{8g,3}\]
such that $\mA \setminus \mA^0=:\Delta$ is a relative normal crossing divisor over $\overline{\mA}_{8g,3}\setminus \mA_{8g,3}=:\infty$.

\subsubsection{Classifying mapping}
Recall the notation $\overline{\mA}_{8g,3}$, which is the compactification of the moduli space $\mA_{8g,3}$ of principal polarized abelian varieties of dimension $8g$ with full $3$-level. There is a universal family $\mE_{univ}$ of abelian varieties over $\mA_{8g,3}$ which can be extended to a family $\overline{\mE}_{univ}$ of generalized abelian varieties with full $N$-level. By the universal property of the moduli space, one gets a classifying mapping.
\begin{proposition}
\label{thm_classifying_mapping_k}
There exists a unique morphism $\overline{\varphi}_k\colon C_k \rightarrow \overline{\mA}_{8g,3}$ such that
\[\overline{\varphi}_k(\mE_{univ})\mid_{C_k\setminus D_k} = A^{(4,4)}_{\pi_k}.\]
\end{proposition}
\begin{proof}
By the universal property of $\mA_{8g,3}$, there exists $\varphi_k\colon C_k\setminus D_k \rightarrow A_{8g,3}$ such that
\[\overline{\varphi}_k(\mE_{univ})\mid_{C_k\setminus D_k} = A^{(4,4)}_{\pi_k}.\]
. Since $\overline{\mA}_{8g,3}$ is projective and regular, the mapping $\varphi_k$ can be extended uniquely.
\end{proof}
\begin{remark}
To lift the family $A^{(4,4)}_{\pi_k}$ is equivalent to lift the classifying map $\overline{\varphi}_k$.
\end{remark}

\subsection{\bf{Step 2).} Polarization on the log Dieudonn\'e module}
\label{sec_sub_polarization_and_Fil}

\subsubsection{The lifting Hodge filtration on the log Dieudonn\'e module of $f_{\pi_k}$}
Let
$(V,\nabla,\Phi,\mV)_{f_{\pi_k}}$ denote the realization of the logarithmic Dieudonn\'e module of $f_{\pi_k}$ over $(\mC,\mD)$,the $p$-adic formal completion of $(C,D)$. By the $\GL_2$-action, its decomposing as form
\[(V,\nabla,\Phi,\mV)_{f_{\pi_k}}=\bigoplus_{i=1}^g(V,\nabla,\Phi,\mV)_{f_{\pi_k},i}.\]
By \autoref{thm_construction_family_mod_p}, the triple $(V,\nabla,\Phi,\mV)_{f_{\pi_k} i}$ underlies a log Fontaine-Faltings module
$\pi_{par}^*(V,\nabla,F,\Phi)^\FF_i$, which is the parabolic pullback of a parabolic Fontaine-Faltings module. Hence, it carries the Hodge filtration
\[ F_{f_{\pi_k}} = \bigoplus_{i=1}^g \pi_{par}^* F_i \subset V_{f_{\pi_k}}.\]
where $ \pi^*_{par} F_i=: \mL^{1,0}_i$ is a positive line bundle on $\mC$ over $W(k)$, and $V_{f_{\pi_k} i}/\mL^{1,0}_i=:\mL^{0,1}_i$ is a negative line bundle
with $\mL^{0,1}_i=\left(\mL^{1,0}_i\right)^{-1}$.

\subsubsection{The lifting Hodge filtration on the log Dieudonn\'e module of $f^t_{\pi_k}$}
Similarly, we find the realization of the logarithmic Dieudonn\'e module attached to $f_{\pi_k}^t$
\[ (V,\nabla,\Phi,\mV)_{f^t_{\pi_k}}=(V,\nabla,\Phi,\mV)^\vee_{f_{\pi_k}}
=\bigoplus_{i=1}^g(V,\nabla,\Phi,\mV)^\vee_{f_{\pi_k} i}\]
which carries the Hodge filtration
\[ F_{f^t_{\pi_k}} = (V_{f_{\pi_k}}^{\oplus 4}/ F_{f_{\pi_k}})^\vee= \bigoplus_{i=1}^g \mL_i^{0,1\vee} \subset V_{f^t_{\pi_k}} = V^\vee_{f_{\pi_k}} .\]

\subsubsection{The lifting Hodge filtration on the log Dieudonn\'e module of $f^{(4,4)}_{\pi_k}$ and $(f^{(4,4)}_{\pi_k})^t$}
Putting everything together, we find the realizations of the logarithmic Dieudonn\'e module attached to $f^{(4,4)}_{\pi_k}$ and $(f^{4,4}_{\pi_k})^t$
\[(V,\nabla,\Phi,\mV)_{f^{(4,4)}_{\pi_k}} = (V,\nabla,\Phi,\mV)_{f_{\pi_k}i}^{\oplus 4} \oplus (V,\nabla,\Phi,\mV)_{f_{\pi_k}i}^{\vee \oplus 4}\]
\[(V,\nabla,\Phi,\mV)_{(f^{(4,4)}_{\pi_k})^t} = (V,\nabla,\Phi,\mV)_{f_{\pi_k}i}^{\vee \oplus 4} \oplus (V,\nabla,\Phi,\mV)_{f_{\pi_k}i}^{\oplus 4}.\]
and Hodge filtrations on the realizations
\[F_{f^{(4,4)}_{\pi_k}} = F_{f_{\pi_k}}^{\oplus 4} \oplus F_{f^t_{\pi_k}}^{\oplus 4} \subset (V,\nabla,\Phi,\mV)_{f^{(4,4)}_{\pi_k}}\]
\[F_{(f^{(4,4)}_{\pi_k})^t} =F_{f^t_{\pi_k}}^{\oplus 4} \oplus F_{f_{\pi_k}}^{\oplus 4} \subset (V,\nabla,\Phi,\mV)_{(f^{(4,4)}_{\pi_k})^t}.\]

Consequently, $F_{f^{(4,4)}_{\pi_k}}$ is a positive vector bundle and $V_{f^{(4,4)}_{\pi_k}}/F_{f^{(4,4)}_{\pi_k}}\cong \left(F_{f^{(4,4)}_{\pi_k}}\right)^\vee$ is a negative vector bundle.

\subsubsection{Compatibility of the lifting Hodge filtration with the principal polarization}
The principal polarization $\iota$ induces an isomorphism, by abusing notion we still denote it by $\iota$
\[\iota: (V,\nabla,\Phi,\mV)_{f^{(4,4)}_{\pi_k}} \rightarrow (V,\nabla,\Phi,\mV)_{(f^{4,4}_{\pi_k})^t}\cong \left((V,\nabla,\Phi,\mV)_{f^{(4,4)}_{\pi_k}}\right)^\vee.\]
\begin{proposition} \label{thm_polarization_and_filtration}
$\iota(F_{f^{(4,4)}_{\pi_k}}) = (V_{f^{(4,4)}_{\pi_k}}/F_{f^{4,4}_{\pi_k}})^\vee$.
\end{proposition}

\begin{proof}
First we consider the modulo $p$ reduction. Since the abelian scheme $f^{(4,4)}_{\pi_k}$ is a principal polarized abelian scheme
and $ F_{f^{(4,4)}_{\pi_k}}$ modulo $p$ is the Hodge filtration of the relative differential forms on $f^{(4,4)}_{\pi_k}$ and
$V_{f^{(4,4)}_{\pi_k}}/F_{f^{(4,4)}_{\pi_k}})^\vee$ is the Hodge filtration of the relative differential forms on $(f^{4,4}_{\pi_k})^t$.
Hence the isomorphism $\iota$ between the filtered Dieudonn\'e modules modulo $p$ is nothing but the isomorphism $\iota$ between the filtered de Rham bundles. Hence
\[\iota\left(F_{f^{(4,4)}_{\pi_k}}\right) \pmod{p} = (V_{f^{(4,4)}_{\pi_k}}/F_{f^{4,4}_{\pi_k}})^\vee\pmod{p}.\]
We will prove this lemma by contradiction. Suppose $\iota(F_{f^{(4,4)}_{\pi_k}})\neq (V_{f^{(4,4)}_{\pi_k}}/F_{f^{4,4}_{\pi_k}})^\vee$. Then there exists some $n\geq1$ such that
\begin{equation} \label{eq_mod_pn}
\iota\left(F_{f^{(4,4)}_{\pi_k}}\right) \pmod{p^n} = (V_{f^{(4,4)}_{\pi_k}}/F_{f^{4,4}_{\pi_k}})^\vee\pmod{p^n}
\end{equation}
and
\begin{equation} \label{eq_mod_pn+1}
\iota\left(F_{f^{(4,4)}_{\pi_k}}\right) \pmod{p^{n+1}} = (V_{f^{(4,4)}_{\pi_k}}/F_{f^{4,4}_{\pi_k}})^\vee\pmod{p^{n+1}}.
\end{equation}
Consider the composition
\[F_{f^{(4,4)}_{\pi_k}} \xrightarrow{\quad \iota \quad} (V_{f^{(4,4)}_{\pi_k}})^\vee \twoheadrightarrow (F_{f^{4,4}_{\pi_k}})^\vee\pmod{p^{n+1}},\]
which is zero modulo $p^n$ by \eqref{eq_mod_pn} and nonzero modulo $p^{n+1}$ by \eqref{eq_mod_pn+1}. Thus by dividing $p^n$ and reduction modulo $p$, the composition induces a non-zero morphism
\[F_{f^{(4,4)}_{\pi_k}} \pmod{p} \rightarrow (F_{f^{(4,4)}_{\pi_k}})^\vee \pmod{p}.\]
But, this contradicts to the fact that $F_{f^{(4,4)}_{\pi_k}} \pmod{p}$ is positive.
\end{proof}

\subsection{{\bf Step 3).} Grothendieck-Messing-Kato logarithmic deformation theorem}

\begin{theorem}[Grothendieck-Messing-Kato logarithmic deformation theorem] \label{thm_lifting_family_to_W}
Let $(Y,D)$ be a smooth curve over $W(k)$ together with a relative normal crossing divisor $D$. Let
\[\psi_1: Y_k \rightarrow \overline{\mA}_{8g,3}\]
be a morphism such that $\psi_1(Y_k\setminus D_k)\subset \mA_{8g,3}$. Assume that
\begin{enumerate}[$(1).$]
\item the the pulled back Hodge bundle $E^{1,0}_{\psi_1}:=\psi^*E^{1,0}_{\overline{\mA}_{8g,6}}$ is positive, i.e. any quotient bundle of $E^{1,0}_{\psi_1}$ has positive degree with respect to an ample divisor $H$ on $Y$ and $E^{0,1}_{\psi_1}$ is negative.
\item $E^{1,0}_{\psi_1}$ has a lifting as a sub vector bundle $F\subset (V,\nabla,\Phi,\mV)_{\psi_1}$ over $W(k)$ and compatible with the polarization, i.e.
\[\iota(F)=(V_{\psi_1}/F)^\vee.\]
\end{enumerate}
Then $\psi_1$ lifts to a log map $\psi$ over $W(k)$ and such that
the sub bundles $E^{1,0}_{\psi}$ and $F$ in $(V,\nabla,\Phi,\mV)_{\psi_1}$ coincide with each other.
\end{theorem}
\begin{proof}
Take a collection of local liftings
$\{\psi_{2 \beta}\}_\beta$ over $W_2(k) $ of $\psi_1$, which induces a collection of local liftings
$\{E^{1,0}_{\psi_{2 \beta}}\}_\beta$ over $W_2(k)$ of $E^{1,0}_{\psi_1}$.
Since by assumption $E^{1,0}_{\psi_1}$ has a global lifting $F\otimes W_2(k)$,
the obstruction cocycle defined by $\{E^{1,0}_{\psi_{2 \beta}}\}_\beta$ vanishes in $H^1(C\otimes k,\mathrm{Sym}^2 E^{0,1}_{\psi_1}). $ Hence, by \autoref{thm_main_comparing_obstruction} the obstruction cocycle defined by $\{\psi_{2 \beta}\}_\beta$ vanishes in $H^1(C_k,\psi^*_1 \Theta^{\log}_{{\mA_{g,N}}/W} )$ and one obtains a global lifting $\psi_2$ over $W_2(k)$ of $\psi_1$.

We show now two sub bundles $E^{1,0}_{\psi_2}$ and $F\otimes W_2(k)$ in $(V,\nabla,\Phi,\mV)_{\psi_1}\otimes W_2(k)$ coincide with each other.
Take the quotient bundle
\[0\to F\otimes W_2(k)\to V_{\psi_1}\otimes W_2(k)\to Q\otimes W_2(k).\to 0\]
and the projection
\[\alpha\colon E^{1,0}_{\psi_2}\hookrightarrow V_{\psi_1}\otimes W_2(k)\to Q\otimes W_2(k).\]
Since $\alpha=0$ (mod $p$). We obtain the map
\[{\alpha\over p}\colon  E^{1,0}_{\psi_1}=E^{1,0}_{\psi_2}\otimes k\to Q\otimes k =E^{1,0\vee}_{\psi_1}.\]
By the assumption $E^{1,0}_{\psi_1}$ is positive and $E^{1,0\vee}_{\psi}$ is negative, which implies that
${\alpha\over p}=0$. Hence $\alpha=0$ and $E^{1,0}_{\psi_2}=F\otimes W_2(k)$.

By repeating the above procedure inductively we finish the proof.
\end{proof}

\begin{corollary} \label{thm_lift_classfying_mapping}
The classifying mapping $\overline{\varphi}_k\colon C_k \rightarrow \mA_{8g,3}$ can be lifted to a mapping
\[\overline{\varphi}\colon C \rightarrow \mA_{8g,3}.\]
\end{corollary}
\begin{proof}
Back to our situation, we have a principal polarized abelian scheme
\[f^{(4,4)}_{\pi_k}: A^{(4,4)}_{\pi_k} \to C_k\] semistable bad reduction on $D$ and carries a level-3 structure, whose Dieudonn\'e module $(V,\nabla,\Phi,\mV)_{f^{(4,4)}_{\pi_k}}$
carries a Hodge filtration
\[V_{f^{(4,4)}_{\pi_k}} \supset F_{f^{(4,4)}_{\pi_k}}\]
lifting the Hodge filtration $E^{1,0}_{f^{(4,4)}_{\pi_k}}$ of the relative differential forms attached to $f_{\pi_k}^{(4,4)}$. The sub bundle
$F_{f^{(4,4)}_{\pi_k}}$ is positive and compatible with the principal polarization
\[\iota(F_{f^{(4,4)}_{\pi_k}}) = (V_{f^{(4,4)}_{\pi_k}})/ F_{f^{(4,4)}_{\pi_k}})^\vee.\]
Since the pullback of the completion of the universal family \[\psi_1^*f^{uni}\colon \psi_1^*\mA\to C_k\]
is also a semistable model of the smooth part of $f^{(4,4)}_{\pi_k},$ we have an isomorphism between the log Dieudonn\'e modules
\[(V,\nabla,\Phi,\mV)_{f^{(4,4)}_{\pi_k}}\simeq (V,\nabla,\Phi,\mV)_{\psi_1}\]
as the canonical extension of Dieudonn\'e module on the smooth part. This isomorphism
induces an isomorphism between Hodge bundles over the close fiber
\[( (V,\nabla,\Phi,\mV)_{f^{(4,4)}_{\pi_k}}\otimes k \supset E^{1,0}_{f^{(4,4)}_{\pi_k}})\simeq ( (V,\nabla,\Phi,\mV)_{\psi_1}\otimes k\supset E^{1,0}_{\psi_1}).\]
Thus the log map
\[\overline{\varphi}_k: C_k\to \overline{\mA}_{8g,6}\]
satisfies the conditions required in \autoref{thm_lifting_family_to_W}, hence it lifts to a
\[\overline{\varphi}: \mC \to \overline{\mA}_{8g,6}\]
such that the log Fontaine-Faltings module
attached to the pulled back abelian scheme $\psi^*f^{uni}: \psi^* \mA\to C$ has the form
\[ (V,\nabla,E^{1,0},\Phi)_\psi\simeq \bigoplus_{i=1}^{g}(\pi^*(V,\nabla,F,\Phi)^{FF \oplus 4}_{i}\oplus
\pi^*(V,\nabla,F,\Phi)^{FF \vee \oplus 4}_{i})\]
where $(V,\nabla,F,\Phi)^\FF_{i}\in \MFh(W(k))_{\bZ_{p^f}}$. Since $C$ is projective over $W$,the $\overline{\varphi}$ is algebraic.
\end{proof}

Summerizing what we have done above
\begin{theorem}
Let $\lambda$, $(\overline{E},\overline{\theta})$ and $(E,\theta)$ be given as in \autoref{thm_family_from_W_to_number_field}. Let $(E,\theta)\in \PHighf(W(k))$. Then after enlarging the field $k$, there exists a finite cover $\pi: (C,D)\to(\bP^1_{W(k)},\{0,1,\lambda,\infty\})$ \'etale on $\bP^1_{W(k)}-\{0,1,\lambda,\infty\}$ and a family of abelian varieties
\[f^{(4,4)}_{\pi}: A^{(4,4)}_{\pi}\to C\]
with semistable bad reduction on $D$ such that $\pi^*(E,\theta)$ is realized by $f^{(4,4)}_{\pi}$. That is, the Higgs bundle attached to $f^{(4,4)}_{\pi}$ is of form
\[(E,\theta)_{f^{(4,4)}_{\pi}} =\bigoplus_{i=1}^{8g}\pi^*(E,\theta)_i\]
with all $(E,\theta)_i\in\High(W(k))$
and $(E,\theta)\simeq (E,\theta)_1$.
\end{theorem}

\subsection{{\bf Step 4).} Descending the family to one over a number field}

\begin{theorem} \label{thm_familyWittring_to_familyNumberring}
Let $\lambda$, $(\overline{E},\overline{\theta})$ and $(E,\theta)$ be given as in \autoref{thm_family_from_W_to_number_field}. After enlarging the field $L$,there exists a finite covering $\pi\colon C\rightarrow \bP^1_L$ defined over $L$ and a family of abelian variety over $C$ such that $\pi^*(E,\theta)$ is a direct summand of the Higgs bundle attached to this family.
\end{theorem}

\begin{proof}
Applying \autoref{thm_lifting_family_to_W} to any Higgs bundle $(\overline{E},\overline{\theta})\in \PHigh(k_\frakp)$, there exist
\begin{itemize}
\item a finite extension $k$ of $k_\frakp$,
\item a curve $C$ defined over $W(k)$,
\item a finite covering mapping
\[\pi: (C,D)\to (\bP^1_{W(k)},\{0,1,\lambda,\infty\}))\]
which is \'etale outsider $\{0,1,\lambda,\infty\}$,and
\item an abelian scheme
\[f^{(4,4)}_{\pi}: A^{(4,4)}_{\pi}\to C\]
with semistable bad reduction over $D$
\end{itemize}
such that the Higgs bundle attached to $f^{(4,4)}_{\pi}$ has the form
\[(E,\theta)_{f^{(4,4)}_{\pi}}=\bigoplus_{i=1}^{8g}\pi^*(E,\theta)_i\]
where all $(E,\theta)_i\in \PHigh(W(k))$ with $(E,\theta)_1=(E,\theta)$ and $\pi^*$ is the parabolic pullback.

As the singular fibers of the abelian scheme $f^{(4,4)}_{\pi}$ over $D$ are maximal degenerated, as in \cite[Section 4]{KYZ22}, it is rigid. Hence, the abelian scheme is defined over some number field $L'$. In other words, there exists a finite field extension $L'$ such that the curve $C$ is defined over $L'$ and the abelian scheme
\[f^{(4,4)}_{\pi}: A^{(4,4)}_{\pi}\to C\]
is also defined over $L'$. In particular, all $p$-adic sub Higgs bundles $\pi^*(E,\theta )_i$ in the above decomposition are in fact algebraic sub Higgs bundles of the Higgs bundle $(E,\theta)_{f^{(4,4)}_{\pi}}/L$ attached to $f^{(4,4)}_{\pi}/L$.
\end{proof}

\subsection{{\bf Step 5).} Descending the abelian scheme $f^{(4,4)}_{\pi}$ over $C$ to $\bP^1_L$.}

Applying Weil restriction along $\pi\colon C\rightarrow \bP^1_L$ to the family in \autoref{thm_family_from_W_to_number_field}, one obtains an abelian scheme
\[h\colon B\to \bP^1_L\]
with bad reduction on $\{0,1,\lambda,\infty\}$ and such that $(E,\theta)$ is a direct summand of the Higgs bundle $(E,\theta)_{h}$ attached to the abelian scheme $h$.

We take then the simple factor, say $h'$ of $h$ such that $(E,\theta )$ is contained in the Higgs bundle $(E,\theta)_{h'}$.
In the following, we show that the family $h'$ is of $\text{GL}_2$-type.

\begin{lemma} \label{thm_sub_arith_loc_sys} Let $\bV_{h'_0}$ denote the Betti local system attached to the smooth fiber space of $h'$. Then there exists a number field $\EK$ such that $\bV_{h'_0}\otimes \mO_\EK$ contains a rank-2 sub local system $\bW$.
\end{lemma}
\begin{proof}
Consider the moduli space $\text{Grass}( 2,\bV_{h_0'})$ rank-2 sub local systems in $\bV_{h'_0}$. Then it is defined over $\bZ$. As $(E,\theta)$ is a sub Higgs bundle of parabolic degree zero in
$(E,\theta)_{h'}$, by Simpson correspondence we obtain a rank-2 complex sub local system
\[\bW_{(E,\theta)}\subset \bV_{h'_0}\otimes \bC.\]
Hence, there exists a number field $\EK$ such that $\bW_{(E,\theta)}$ is defined by $\mO_{\EK}$. In particular, we find a rank-2 sub local system $\bW$ in $\bV_{h'_0}\otimes \mO_{\EK}$.
\end{proof}

By applying \autoref{thm_Simpson}, we finally show the family $h'$ satisfies the requirement in \autoref{thm_family_from_W_to_number_field} as follows:
\begin{proof}[Proof of \autoref{thm_family_from_W_to_number_field}]
We only need to show the abelian scheme $h'$ is of $\text{GL}_2$-type and $(E,\theta)$ is isomorphic to an eigen-sheaf of the Higgs bundle attached to $h'$.

Consider the rank-2 sub local system $\bW\subset \bV_{h'_0}\otimes \mO$ constructed in \autoref{thm_sub_arith_loc_sys}. Then the Higgs bundle corresponding to $\bW$ is tautologically a graded sub Higgs bundle $(E,\theta)_{\bW}\subset(E,\theta)_{h'}$. Further more, the Higgs bundles corresponding all Galois conjugates $\bW^\sigma$ are graded sub Higgs bundles of $(E,\theta)_{h'}$. By Simpson's theorem, we find an abelian scheme over $\bP^1$ of $\text{GL}_2$-type, such that $\bW$ is an eigen-sheaf attached to this abelian scheme. By the construction this abelian scheme is a sub abelian scheme of $h'$. As $h'$ is already simple. We show $h'$ is of $\text{GL}_2$-type. Since $(E,\theta)$ is stable, it is isomorphic to an eigen sheaf attached to $h'$.
\end{proof}

\section{Isomonodromy Deformations of eigen local systems attached to abelian schemes of $\text{GL}_2$-type over $\bP_{\bC}^1-\{0,1,\lambda,\infty\}$}

The main result of Lin-Sheng-Wang says that if $(E,\theta)$ is an eigen Higgs bundle of the Higgs bundle attached an abelian scheme over $P^1_{\bC}$ with some given bad reduction types on $\{0,1,\lambda,\infty\}$ then the zero of the Higgs field $(\theta)_0$ is a torsion point w.r.t. the elliptic curve $y^2=x(x-1)(x-\lambda)$. In this section, the first aim in this section is to prove the converse direction.
\begin{theorem}[\autoref{thm_main}] \label{thm_main_motivic_torsion_description}
Given a $4$-marked complex projective line $(\bP^1,\{0,\,1,\,\lambda,\,\infty\})$ and a Higgs bundle
$(E,\theta)\in \High(\bC)$. Assume the zero of the Higgs field is a torsion point. Then $(E,\theta)$ is motivic. More precisely, there exists a family of abelian varieties $f: A\to \bP_{\bC}^1$ of $\GL_2$-type such that $(E,\theta)$ is an eigen Higgs bundle attached to $f$.
\end{theorem}

The second aim of this section is to given the proof the following main result \autoref{thm_main_Higg_mod_p_to_largest_family} and \autoref{thm_main_motivic_torsion_description}.

\begin{theorem}[\autoref{thm_mainII}] \label{thm_main_Higg_mod_p_to_largest_family}
Let $L$ be a number field and let $\lambda_0 \in M_{0,4}(L)$. Assume $\frakp$ is a finite place such that $\lambda_0$ is a $\frakp$-adic integer and $\lambda_0$ is supersingular at $\frakp$ in the sense \autoref{def_supersingular}. For any
$(\overline E,\overline\theta)\in \HIG_{\lambda_0}^{{\rm gr}{1\over2}}(\overline{k}_{\frakp})$, denote by $(E,\theta)$ the unique motivic lifting in $\HIG_{\lambda_0}^{{\rm gr}{1\over2}}(\overline{\bQ})$ and denote by $f_{\lambda_0}$ the family constructed in\autoref{thm_family_from_W_to_number_field}. Then there exists a finite \'etale covering $\widetilde M_{0,4}\to M_{0,4}$ (depending on $(\overline{E},\overline{\theta})$) such that $f_{\lambda_0}$ can be extended to an abelian scheme
\[f\colon A\to \widetilde S_{0,4}= S_{0,4}\times_{M_{0,4}} \widetilde M_{0,4}\]
of $\text{GL}_2$-type,
with bad reduction on the four punctures. In other words, there exists a point $\widehat{\lambda}_0$ in the preimage of $\lambda_0$ under $\widetilde M_{0,4}\to M_{0,4}$ with
\[f\mid_{S_{0,4}\times_{M_{0,4}}\{\widehat{\lambda}_0\}} \cong  f_{\lambda_0}.\]
\end{theorem}

Since there are infinitely many Higgs bundles $(\overline E,\overline\theta)\in \HIG_{\lambda_0}^{{\rm gr}{1\over2}}(\overline{k}_{\frakp})$ in \autoref{thm_main_Higg_mod_p_to_largest_family}, one gets following result.
\begin{corollary}
There exist infinitely many abelian schemes of the form given in \autoref{thm_main_Higg_mod_p_to_largest_family}.
\end{corollary}

We note that \autoref{thm_main_Higg_mod_p_to_largest_family} is needed in the proof of \autoref{thm_main_motivic_torsion_description}. We will first prove  \autoref{thm_main_Higg_mod_p_to_largest_family}.

\subsection{The proof of \autoref{thm_main_Higg_mod_p_to_largest_family}}
Let's first give outline.

Start with a Higgs bundle $(\overline E,\overline\theta)\in \HIG_{\lambda_0}^{{\rm gr}{1\over2}}(\overline{k}_{\frakp})$, in \autoref{thm_family_from_W_to_number_field}, we found an abelian scheme $f_{\lambda_0}: A_{\lambda_0}\to \bP^1_L$ of $\text{GL}_2(\EK)$-type and with bad reduction on $\{0,\,1,\,\lambda_0,\,\infty \} $ of type-$(1/2)_\infty$, with the eigen sheave decomposition of the filtered parabolic de Rham bundle
\[(V,\nabla,E^{1,0},\Phi)_{f_{\lambda_0}}=\bigoplus_{i=1}^{g}(V,\nabla,E^{1,0},\Phi)_{f_{\lambda_0} i},\]
where $E^{1,0}_{f_{\lambda_0}\,i}\simeq \mO $ and $V_{f_{\lambda_0}\,i}/ E^{1,0}_{f_{\lambda_0}\,i}\simeq \mO(-1)$ for all $i=1,\cdots,g$ and
\[\text{Gr}_{E^{1,0}_{f_{\lambda_0}\,1}}(V,\nabla,E^{1,0} )_{f_{\lambda_0}\,1} =(E,\theta).\]
Our aim is to extend $f_{\lambda_0}$ to an abelian scheme over $\widetilde{S}_{0,4}$ for a finite \'etale covering $\widetilde M_{0,4}\to M_{0,4}$.

\begin{enumerate}[Step $1).$]
\item We first extends the Hodge filtration of $f_{\lambda_0}$ to  the formal neighborhood $S_{4,0}|_{\hat U_{\lambda_0}}$ of $S_{4,0}\mid_{\lambda_0}=\bP^1_L\setminus\{0,1,\lambda_0,\infty\}$ in $S_{0,4}$, where $\widehat{U}_{\lambda_0}$ is the formal neighborhood of $\lambda_0$ in $M_{0,4}$ and $S_{4,0}|_{\hat U_{\lambda_0}} = S_{4,0} \times_{M_{0,4}} \hat U_{\lambda_0}$.

\item Next, by applying Grothendieck-Messing-Kato theorem, the family $f_{\lambda_0}$ extends to an abelian scheme $f_{\widehat{U}_{\lambda_0}}$ over a finite covering of $S_{4,0}|_{\hat U_{\lambda_0}}$.

\item Next, by using the Hom functor and Weil restriction, one finds a non-empty Zariski open set $U\subset M_{0,4}$ of $\lambda_0$ and a finite \'etale cover $\tau\colon\widetilde U\to U$ such that the abelian scheme $f_{\lambda_0}$ extends to an abelian scheme $f_{\widetilde U}$ over $S_{0,4}\mid_{\widetilde{U}} = S_{0,4}\times_{M_{0,4}} \widetilde{U}$.

\item Finally, By analyzing the monodromy, we extend the family $f_{\widetilde U}$ to the entire base space and get an abelian scheme $f_{\widetilde{M}_{0,4}}$ over
$\widetilde S_{0,4}.$
\end{enumerate}

\subsection*{{\bf Step 1).} Extending the Hodge filtration to one over the formal neighborhood.} Denote by $\widehat{U}_{\lambda_0}$ the formal neighborhood of $\lambda_0$ in $M_{0,4}$ and denote by $S_{4,0}|_{\hat U_{\lambda_0}} = S_{4,0} \times_{M_{0,4}} \hat U_{\lambda_0}$ the formal neighborhood of $S_{4,0}\mid_{\lambda_0}=\bP^1_L\setminus\{0,1,\lambda_0,\infty\}$ in $S_{0,4}$.

By the isomonodromy deformation of the local system attached to the family $f_{\lambda_0}$ (or the realization of the crystal attached to $f_{\lambda_0}$), the de Rham bundle $(V,\nabla)_{f_{\lambda_0}}$ attached to the family $f_{\lambda_0}$ lifts naturally to a de Rham bundle $(V_{ S_{0,4}|_{\hat U_{\lambda_0}}},\nabla_{ S_{0,4}|_{\hat U_{\lambda_0}}})$ on $S_{0,4}|_{\hat U_{\lambda_0}}$ which endowed with an $\EK$-action. The $\EK$-action induces a decomposition of $(V_{ S_{0,4}|_{\hat U_{\lambda_0}}},\nabla_{ S_{0,4}|_{\hat U_{\lambda_0}}})$ into rank $2$ sub de Rham bundles
\[(V_{ S_{0,4}|_{\hat U_{\lambda_0}}},\nabla_{ S_{0,4}|_{\hat U_{\lambda_0}}}) = \bigoplus_{i=1}^{g} (V_{ S_{0,4}|_{\hat U_{\lambda_0}}i},\nabla_{ S_{0,4}|_{\hat U_{\lambda_0}}i}).\]

\begin{lemma} \label{thm_formal_deform_over_M04} The Hodge filtration
\[E^{1,0}_{f_{\lambda_0}} =\bigoplus_{i=1}^{g} E^{1,0}_{{f_{\lambda_0}i}}\subset (V_{f_{\lambda_0}},\nabla_{f_{\lambda_0}}) = \bigoplus_{i=1}^g (V_{f_{\lambda_0}i},\nabla_{f_{\lambda_0}i})\]
can be lifted to a Hodge filtration of $(V_{ S_{0,4}|_{\hat U_{\lambda_0}}},\nabla_{ S_{0,4}|_{\hat U_{\lambda_0}}})$.
\end{lemma}

\begin{proof}
Since the obstruction for lifting the Hodge filtration
\[E^{1,0}_{f_{\lambda_0}} =\bigoplus_{i=1}^{g} E^{1,0}_{{f_{\lambda_0} i}}\subset (V,\nabla)_{f_{\lambda_0}}\]
lies in
\[ \bigoplus_{i=1}^{g}H^1(\bP^1,E^{1,0\vee} _{f_{\lambda_0} i} \otimes (V_{f_{\lambda_0},i}/ E^{1,0}_{{f_0}i})) =
\bigoplus_{i=1}^{g}H^1(\bP^1,\mO(-1) )=0.\]
There exists a lifting $E^{1,0}_{S_{0,4}|_{\hat U_{\lambda_0}}}$ of $E^{1,0}_{f_{\lambda_0}}$ in $V_{ S_{0,4}|_{\hat U_{\lambda_0}}}$ with an $\EK$-multiplication.
\end{proof}

\subsection*{{\bf Step 2).} Extending the family to one over the formal neighborhood.}
Similar to that we have done in \autoref{sec_main_lifting}, the lifting of the Hodge filtration shall lead a lifting of some classifying mapping. As in \autoref{thm_add_level}, we add the full $3$-level structure to $f_{\lambda_0}$, then one gets a finite covering mapping
\[\pi_{\lambda_0}: (C,D)_{\lambda_0}\to (\bP^1,\{0,\,1,\,\lambda_0,\,\infty\})\]
ramified only at the punctures. Now, we vary $\lambda$ in $M_{0,4}$, the covering can be extended locally around $\lambda=\lambda_0$. Thus after passing through a finite \'etale covering $\widetilde{M_{0,4}}$ of $M_{0,4}$, we extend the covering globally and get a finite cover $\pi$ as in the following Cartier diagram
\begin{equation*}
\xymatrix{
(C,D)_{\lambda_0} \ar[r] \ar[d]^{\pi_{\lambda_0}} & (C,D)_{\widetilde{M}_{0,4}} \ar[d]^{\pi} \\
(\bP^1,\{0,\,1,\,\lambda_0,\,\infty\}) \ar[r] \ar[d] & \widetilde{S}_{0,4} \ar[d] \ar[r] & S_{0,4} \ar[d]\\
\{\lambda_0\} \ar@{^(->}[r] & \widetilde{M}_{0,4} \ar[r]^{\text{finite}} & M_{0,4}\\
}
\end{equation*}

Since the base change family $f_{\pi_{\lambda_0}}: A_{\pi_{\lambda_0}}\to C_{\lambda_0}$ has the full $3$-level structure, which induces a log period mapping into a smooth compactification of the fine Hilbert modular variety defined by the multiplication field $\EK$
\[\psi_{f_{\pi_{\lambda_0}}}: (C,D)_{\lambda_0}\to (\overline{\mH}_{\EK},\infty).\]
Together with Grothendieck-Messing-Kato log deformation theorem, the \autoref{thm_formal_deform_over_M04} implies that there is a lifting of the period mapping over $\hat U_{\lambda_0}$
\[\psi_{\pi_{\hat U_{\lambda_0}}}: C_{\widetilde M_{0,4}}|_{\hat U_{\lambda_0}}\to (\overline{\mH}_{\EK},\infty).\]
By pulling back the universal family along $\psi_{\pi_{\hat U_{\lambda_0}}}$, one gets an abelian scheme $f_{\hat U_{\lambda_0}}$ over $C_{\widetilde M_{0,4}}|_{\hat U_{\lambda_0}}$.

\subsection*{{\bf Step 3).} Extending the family over the formal neighborhood to one over an \'etale neighborhood.}

\begin{lemma}
The lifting $\psi_{\pi_{\hat U_{\lambda_0}}}$ over the formal neighborhood $\hat U_{\lambda_0}$ extends to one over an finite \'etale neighborhood $\widetilde{U}$. In other words, there exists
\[\psi \colon  S_{0,4}\mid_{\widetilde{U}}=S_{0,4}\times_{M_{0,4}} \widetilde{U} \to (\overline{\mH}_{\EK},\infty)\]
such that $\psi_{\pi_{\hat U_{\lambda_0}}} = \psi\mid_{ C_{\widetilde M_{0,4}}|_{\hat U_{\lambda_0}}}$.
\end{lemma}
\begin{proof}
For a positive integer $d$, we consider the moduli functor
\begin{equation*}
\widetilde{M}_{0,4}\text{-Sch} \longrightarrow \mathrm{Set}
\end{equation*}
defined by
\[T \mapsto \{\psi\colon C_{\widetilde{M}_{0,4}}\times_{\widetilde{M}_{0,4}} T \rightarrow (\overline{\mH_{\EK}},\infty) \mid \deg\psi\leq d\}\]
Then the functor is represented by finite type $M_{0,4}$-scheme $\mM_{\pi,d}$. Let $\beta$ denote the structure morphism
\[\beta: \mM_{\pi,d}\to \widetilde{M}_{0,4}.\]

Let $d_0:=\deg \psi_{f_{\pi_{\lambda_0}}},$ then the existence of the lifting $\psi_{\pi_{\hat U_{\lambda_0}}}$ implies
\[\hat U_{\lambda_0}\subset \beta(\mM_{\pi,d_0})\subseteq \widetilde{M}_{0,4}.\]
Hence the constructible
subset $\beta(\mM_{\pi,d_0})$ contains a non-empty Zariski open set $U\subseteq \widetilde{M}_{0,4}$. By means of the isomonodromy deformation, there exists a finite covering
\[\widetilde{\widetilde{M}}_{0,4}\rightarrow \widetilde{M}_{0,4}\]
which is \'etale over $U$ (denote by $\widetilde{U}=\psi^{-1}(U)$)
such that there exists a mapping
\[\psi_C \colon C_{\widetilde{M}_{0,4}} \times_{\widetilde{M}_{0,4}} \widetilde{U} \rightarrow (\overline{\mH_{\EK}},\infty)\]
which extends $\psi_{\pi_{\hat U_{\lambda_0}}}$. By pullback the universal family of abelian varieties along $\psi_C$, one gets a family of abelian varieties \[f_{\widetilde U}: B'_{\widetilde U}\to C_{\widetilde{M}_{0,4}} \times_{\widetilde{M}_{0,4}} \widetilde{U}.\qedhere\]
\end{proof}

\subsection*{{\bf Step 4).} Extending the family over the \'etale neighborhood to one over $\widetilde{S}_{0,4}$.}
To finish the proof of \autoref{thm_main_Higg_mod_p_to_largest_family}, we only need to show the following lemma.
\begin{lemma}
The lifting $\psi_{\widetilde{U}}$ over the formal neighborhood $\hat U_{\lambda_0}$ extends to $\widetilde{M}_{0,4}$. In other words, there exists
\[\psi \colon \widetilde S_{0,4} \to (\overline{\mH}_{\EK},\infty)\]
such that $\psi_{\pi_{\hat U_{\lambda_0}}} = \psi\mid_{ C_{\widetilde M_{0,4}}|_{\hat U_{\lambda_0}}}$.
\end{lemma}

\begin{proof}
Since the local system associated to this family has trivial local monodromy around the fibers $C_{\widetilde{M}_{0,4}} \times_{\widetilde{M}_{0,4}} \{\lambda\}$ for each $\lambda \in \widetilde{\widetilde{M}}_{0,4}\setminus \widetilde{U}$, we can extend the abelian scheme across those fibers by a well known theorem due to Deligne and get a family
\[f_{\widetilde{\widetilde{M}}_{0,4}}: B'_{\widetilde{\widetilde{M}}_{0,4}} \to C_{\widetilde{M}_{0,4}} \times_{\widetilde{M}_{0,4}} \widetilde{\widetilde{M}}_{0,4},\]
where $\widetilde{\widetilde M}_{0,4}\to \widetilde{M}_{0,4}$ is a finite \'etale covering. By taking Weil restriction, we descend
this abelian scheme back to $\widetilde{S}_{0,4}$ and get a mapping $\psi$.
\end{proof}

\subsection{The proof of \autoref{thm_main_motivic_torsion_description}}

\begin{proof}[Proof of \autoref{thm_main_motivic_torsion_description}]
Let $g\colon C\to \bP^1$ be the Legendre family. Then one may identify the smooth locus of $g$ with $M_{0,4}$, the moduli space of projective line with $4$-punctures, which sends $\lambda$ to the projective line with punctures at $\{0,1,\lambda,\infty\}$. For any $\lambda\neq 0,1,\infty$, the fiber of $g$ at $\lambda$ is just the elliptic curve given by the double cover $\pi_\lambda\colon C_\lambda\to \bP^1$ ramified on $\{0,1,\lambda,\infty\}$.

Assume $(E,\theta)$ is motivic. Then the modulo $\frakp$ reduction of $(E,\theta)$ is periodic for almost all places $\frakp$. According \autoref{thm_LSW}, the modulo $\frakp$ reduction of $(\theta)_0$ is torsion. By a theorem of Pink \cite{Pin04}, it itself is torsion. Conversely, assume $(\theta)_0$ is a torsion point with order $m$, in the following we show $(E,\theta)$ is motivic.

We choose a number field $K$ and an integer $\lambda_0\in \mO_K$ such that $C_{\lambda_0}$ is an elliptic curve with complex multiplication. Choose a sufficient large place $\frakp$ such that the reduction of $C_{\lambda_0}$ at $\frakp$ is supersingular and $\frakp\nmid m$. Let $\Sigma_m\subset \widetilde{C}$ be the $m$-torsion (multiple) section, $T_m=\pi(\Sigma_m)\subset \widetilde{S}_{0,4}$.
Then $T_m$ is \'etale over $\widetilde{M}_{0,4}$. Let $T'_m$ be the irreducible component of $T_m$ containing $(\theta)_0$.
\begin{equation*}
\xymatrix@C=2mm{
C_{\lambda_0} \ar@{^(->}[r] \ar[d] & \widetilde{C}\ar[r] \ar[d] & C \ar[d] \\
\bP^1\setminus\{0,1,\lambda_0,\infty\} \ar[d] \ar@{^(->}[r] & \widetilde{S}_{0,4}\ar[r] \ar[d]& S_{0,4}\ar[d]\\
\{\lambda_0\} \ar@{^(->}[r] & \widetilde{M}_{0,4} \ar[r] & M_{0,4}
}
\qquad
\xymatrix{
& C_\lambda \ar@{^(->}[r] \ar[d] & \widetilde{C} \ar@{}[r]|\supset \ar[d] & \Sigma_m \\
(\theta)_0 \ar@{}[r]|-\in & \bP^1\setminus\{0,1,\lambda,\infty\} \ar[d] \ar@{^(->}[r] & \widetilde{S}_{0,4} \ar[d] \ar@{}[r]|\supset & T'_m \\
& \{\lambda\} \ar@{^(->}[r] & \widetilde{M}_{0,4}\\
}
\end{equation*}
Choose a Higgs bundle $(E_0,\theta_0)$ in $\HIG_{\lambda_0}^{{\rm gr}{1\over2}}(\bC)$ with zero located in the intersection set $T'_m\cap \bP^1\setminus\{0,1,\lambda_0,\infty\} \subset \widetilde{S}_{0,4}$. Then the zero of this Higgs field $\theta_0$ is torsion of order $m$. Since the modulo $\frakp$ reduction of the Higgs bundle $(E_0,\theta_0)$ is also torsion and of order $m$ with $\frakp\nmid m$. By \autoref{thm_LSW}, the reduction $(E_0,\theta_0)\pmod{\frakp}$ is periodic. According \autoref{thm_main_Higg_mod_p_to_largest_family}, there exists a family of abelian varieties $f\colon A\to \widetilde{S}_{0,4}$ of $\GL_2$-type such that $(E_0,\theta_0)\pmod{\frakp}$ is an eigen Higgs bundle attached to $f_{\lambda}$. In other words, there is an eigen component $(E_0,\theta_0)^{per}$ of the Higgs bundle attached to the family $f$ such that \[(E_0,\theta_0)^{per}\mid_{(\bP^1,\{0,1,\lambda_0,\infty\})} \pmod{\frakp} = (E_0,\theta_0) \pmod{\frakp}.\]
Since $(E_0,\theta_0)^{per}$ comes from families of abelian varieties, it is motivic. Thus the zero of the Higgs field is an algebraic section consisting of torsion points.

We claim that the torsion points in the section are all of order $m$. By the constancy of the order, we only need to show the order $(E_0,\theta_0)^{per}\mid_{(\bP^1,\{0,1,\lambda_0,\infty\})}=(E_0,\theta_0)$. Since modulo $\frakp$ mapping is injective for torsion points of order coprime to $p$, we only need to show the order of $\theta_0^{per}$ is coprime to $p$. This follows the fact that $C_{\lambda_0}$ is endowed with complex multiplication. Because the number field generated by an $p$-torsion point is ramified over $\Qp$, but the field generated by the zero of $\theta_0^{per}\mid_{(\bP^1,\{0,1,\lambda_0,\infty\})}$ is unramified.

Thus torsion section $(\theta^{per}_0)_0$ passes through $(\theta_0)_0$. By the choice of $(E_0,\theta_0)$, it also passes through $(\theta)_0$. Hence $(E,\theta)$ is motivic.
\end{proof}

\appendix

\section{\bf Obstruction to lifting a family of abelian varieties} \label{sec_main_compare_obstructions}

In this appendix, we prove that the Kodaira-Spencer map sends the obstruction of lift the classifying map to the obstruction of lift the Hodge filtration. We firstly fix some notations.
\begin{itemize}
\item[$\bullet$]$k$: a perfect field of characteristic $p > 0$.
\begin{itemize}
\item $W := W(k)$,
\item $K := {\rm Frac} W$.
\end{itemize}
\item[$\bullet$]$X$: a (proper) smooth curve over $W$.
\begin{itemize}
\item $X_k$: the special fiber of $X$;
\item $X_n$: the modulo $p^n$ reduction of $X$;
\item $X_K$: the generic fiber of $X$;
\item $\mX$: the $p$-adic formal completion of $X$ along the special fiber $X_k$;
\item $\mX_K$: the rigid analytic space associated to $\mX$.
\end{itemize}
\item[$\bullet$] $D\subset X$: a relative divisor, flat over $W$.
\item[$\bullet$] $U=X\setminus D$: the complement of $D$ in $X$.
\begin{itemize}
\item $U_k$, $U_n$, $U_K$, $\mU$, $\mU_K$ defined similar as those for $X$.
\item $\mD_K$: the complement of $\mU_K$ in $\mX_K$ (finite union of disks);
\end{itemize}
\item [$\bullet$] $\mA_{g,N}$: the module space of principal polarized abelian varieties of dimensional $g$ with full $N$-level for $N\geq 3$.
\begin{itemize}
\item $\pi_{univ}\colon\mE_{univ}\rightarrow \mA_{g,N}$: the universal family of abelian varieties.
\item $\overline{\mA}_{g,N}\supset \mA_{g,N}$: the compactification of $\mA_{g,N}$.
\item $\overline \pi_{univ}\colon \overline{\mE}_{univ}\rightarrow \overline{\mA}_{g,N}$: the family of generalized abelian varieties with full $N$-level over $X(N)$.
\end{itemize}
\item[$\bullet$] $\bD(\mE)$: the (logarithmic) Dieudonn\'e $F$-crystal associated to a (semistable) family $\mE$ over $X_k$.
\begin{itemize}
\item $\bD(\mE)(X_n)$ the realization of $\bD(\mE)$ on $X_n$, which is a vector bundle over $X_n$ together with a connection and a Frobenius endomorphism.
\item $\bD(\mE)(\mX) = \varprojlim\limits_n\bD(\mE)(X_n)$ the realization of $\bD(\mE)$ on $\mX$.
\end{itemize}
\end{itemize}

\subsection{Obstruction of lifting a morphism.}
Let $f\colon (X,M_X)\rightarrow (S,M_S)$ be a morphism between schemes with fine logarithmic structures. Denote by $\Omega_{X/S}^1(\log(M_X/M_S))$ the sheaf of logarithmic differentials; write this simply by $\omega^1_{X/S}$ if there is no risk of confusion about the logarithmic structures. The dual, a.k.a. the sheaf of logarithmic vector fields, is denoted by $\Theta_{X/S}(\log(M_X/M_S))$ or $\Theta_{X/S}^{\log}$.

We will periodically refer to the following type of commutative diagram of fine logarithmic schemes:
\begin{equation}
\label{diag:smooth}
\xymatrix{(T_0,M_{T_0}) \ar[r]^{g_0} \ar[d]^{\iota} & (X,M_X) \ar[d]^{f} \\
(T,M_T) \ar[r]^{t}
& (S,M_S) \\}
\end{equation}
where $\iota$ is an exactly closed immersion and $T_0$ is defined in $T$ by a quasi-coherent sheaf of ideals $I$ with $I^2=0$. As usual, because $I^2=0$, $I$ is naturally a quasi-coherent sheaf of modules on $T_0$.

\begin{definition}\cite[3.3 on p. 201]{Kat89} A morphism $f\colon (X,M_X)\rightarrow (S,M_S)$ between fine logarithmic schemes is called \emph{smooth} if
\begin{enumerate}[$(1).$]
\item $f$ is locally of finite presentation and
\item For any commutative diagram as in \eqref{diag:smooth}, locally on $T$ there exists a lift $g\colon(T,M_T)\rightarrow (X,M_X)$ of $g_0$ such that $g\circ \iota=g_0$ and $f\circ g =t$.
\end{enumerate}
\end{definition}
Condition (2) is the logarithmic version of formal smoothness for schemes.

We now follow \cite[Proposition 3.9 on p. 203]{Kat89} to understand the space of lifts. Suppose $g$ and $g'$ are two liftings of $g_0$. We define an element $\alpha_{g,g'}$ in $\mathrm{Hom}(g_0^*\omega^1_{X/S},I)$ which satisfies
\begin{itemize}
\item $\alpha_{g,g'}(\mathrm{d}a) = g^*(a)-g'^*(a)$ for $a\in \mO_X$ and
\item $\alpha_{g,g'}(\mathrm{d}\log a) = u(a) -1$ for $a \in M_X$,
\end{itemize}
where $u(a)$ is the unique local section of $\ker(\mO^*_T \rightarrow \mO^*_{T_0}) \subset M_T$ such that $g^*(a)=g'^*(a)*u(a)$. By the arguments of \cite[p. 203]{Kat89}, $g=g'$ if and only if $\alpha_{g,g'} =0$.

In general, there is an obstruction to lift the map $g_0$ globally; this obstruction can be described using the $\alpha$ just defined. Choose local liftings $g_i$. On the overlap open set the lifting $g_i$ differs with $g_j$ by $\alpha_{ij} = \alpha_{g_i,g_j}$.
\begin{lemma}(Kato) \label{obs:map} Let $f\colon (X,M_X)\rightarrow (S,M_S)$ be a smooth morphism between schemes with fine logarithmic structures. For any commutative diagram as in (\eqref{diag:smooth})
\begin{enumerate}[$(1).$]
\item The $\alpha(g_0):=(\alpha_{ij})$ is a well-defined element in
\[H^1(T_0,\mHom(g_0^* \omega^1_{X/S},I))=H^1(T_0,g_0^* \Theta^{\log}_{X/S}\otimes I)\]
which does not depend on the choice of local liftings. It is the obstruction to lift $g_0$ globally, i.e. $\alpha=0$ if and only if there exists an $(S,M_S)$-morphism $g:(T,M_T)\rightarrow (X,M_X)$ lifting $g_0$;
\item if $\alpha(g_0)=0$, the set of lifts $g$ of $g_0$ is an affine space under $H^0(T_0,\mHom(g_0^* \omega^1_{X/S},I))=\mathrm{Hom}(g^*_0\omega_{X/S}^1,I)$.
\end{enumerate}

\end{lemma}

\subsection{Obstruction of lifting a sub-bundle}
We next consider the obstruction to lifting a sub-bundle. In this paragraph, the logarithmic structures play no role. Let $\iota:T_0\rightarrow T$ be a square zero thickening with ideal sheaf $I$. Let $V$ be a vector bundle over $T$ together with an symmetric isomorphism
\[\tau\colon V\rightarrow V^t\]
in the sence that $\tau^t=\tau$ where $V^t$ is the dual vector bundle of $V$. Then $\tau$ can be view as an element in $\mathrm{Sym}^2(V^t)$
\[\tau\in\mathrm{Sym}^2(V^t).\]
Let $\overline{L}$ be a vector sub-bundle of $\overline{V}=V\otimes_{\mO_T} \mO_{T_0}$ such that
\[\tau(\overline{L}) = (\overline{V}/\overline{L})^t.\]
This is equivalent to say that
\[\tau \pmod{I} \in \ker\left( \mathrm{Sym}^2(\overline{V}^t) \rightarrow \mathrm{Sym}^2(\overline{L}^t) \right)\]
\begin{lemma}
Zariski locally on $T$ there exist liftings $L$ of $\overline{L}$ such that $\tau(L)=(V/L)^t$. In other words,
\[\tau \in \ker\left( \mathrm{Sym}^2(V^t) \rightarrow \mathrm{Sym}^2(L^t) \right)\]
\end{lemma}
\begin{proof}
Locally choose a basis $e_1,\cdots,e_{2g}$ of $L$ such that $\overline{L}$ is generated by $e_1,\cdots,e_g \pmod{I}$. Denote by $f_1,\cdots,f_{2g}$ the dual basis of $e_1,\cdots,e_{2g}$. Then locally $\tau$ can be represented as
\[\tau = \sum_{i=1}^{2g} a_{ij} \cdot f_i\otimes f_j \in \mathrm{Sym}^2(V^t)\]
with $a_{ij}=a_{ji}$ for each pair $(i)$. Then $\tau(\overline{L}) = (\overline{V}/\overline{L})^t$ means that the coefficients matrix of $\tau$ under the basis $e_1,\cdots,e_{2g}$ has following form
\[A=\left(\begin{array}{cccccc}
a_{1,1} & \cdots & a_{1,g} & a_{1,g+1} & \cdots & a_{1,2g} \\
\vdots & & \vdots & \vdots & & \vdots \\
a_{g,1} & \cdots & a_{g,g} & a_{g,g+1} & \cdots & a_{g,2g} \\
a_{g+1,1} & \cdots & a_{g+1,g} & a_{g+1,g+1} & \cdots & a_{g+1,2g} \\
\vdots & & \vdots & \vdots & & \vdots \\
a_{2g,1} & \cdots & a_{2g,g} & a_{2g,g+1} & \cdots & a_{2g,2g} \\
\end{array}\right) = \left(\begin{array}{cc}
A_{11} & A_{12} \\ A_{21} & A_{22}\\
\end{array}\right)\]
with $A_{11}=0\pmod{I}$ and $A_{12}=A_{21}^T$ invertible.
Now it is easy to find a matrix $Q\in I^{g\times g}$ such that
\[\left(\begin{array}{cc}1&Q^T\\&1\\\end{array}\right) \left(\begin{array}{cc}
A_{11} & A_{12} \\ A_{21} & A_{22}\\
\end{array}\right) \left(\begin{array}{cc}1&\\Q&1\\\end{array}\right)\]
has form $\left(\begin{array}{cc} 0&B_{12}\\B_{21}&A_{22}\\ \end{array}\right)$. This means that the coefficients matrix of $\tau$ under the basis $(e_1,\cdots,e_{2g})\left(\begin{array}{cc}1&\\Q&1\\\end{array}\right)$ has form $\left(\begin{array}{cc} 0&B_{12}\\B_{21}&A_{22}\\ \end{array}\right)$. Denote $L$ the subsheaf generated $e_1,\cdots,e_g$. Then $\tau(L)=(V/L)^t$, since $B_{12}=B_{21}^T$ is invertible.
\end{proof}
In general, there is an obstruction to get a global lifting $L$ of $\overline{L}$ such that
\[\tau(L)=(V/L)^t.\]
We choose local liftings $L_i$ of $\overline{L}$ with $\tau(L_i)=(V/L_i)^t$. Then one has following commutative diagram
\begin{equation*}
\xymatrix{
L_i \ar[r] \ar[d]^{\tau}_{\simeq} & V \ar[r] \ar[d]^{\tau}_{\simeq} & V/L_j \ar[d]^{\tau}_{\simeq}\\
(V/L_i)^t \ar[r] & V^t \ar[r] & L_j^t \\
}
\end{equation*}
Consider the composition $\alpha_{ij}\colon L_i \hookrightarrow V \xrightarrow{\tau} V^t \twoheadrightarrow L_j^t$ over the overlap open subset; this morphism is zero modulo $I$. Thus, it factors thought a map $\beta_{ij}: \overline{L_i} \rightarrow \overline{L_j}^t \otimes I$.
\begin{equation}
\label{eq_map_beta}
\xymatrix{L_i \ar@{->>}[d] \ar@{^(->}[r] & V \ar[r]^{\tau}_{\simeq} & V^t \ar@{->>}[r] & L_j^t\\
\overline{L_i} \ar[rrr]^{\beta_{ij}} &&& \overline{L_j}^t \otimes_{\mO_{T_0}} I \ar@{^(->}[u]\\}
\end{equation}
Since $\tau$ is self dual, one has $\alpha_{ij}^t=\alpha_{ji}$. Thus $\beta_{ij}\in \mathrm{Sym}^2(\overline{L}^t)\otimes_{\mO_{T_0}}I$.
From the definition, it is easy to see that $\beta_{ij}=0$ if and only if $L_i=L_j$. One concludes the following result.
\begin{lemma}
\label{obs:fil}
Let $i:T_0\rightarrow T$ be a square-zero thickening with ideal sheaf $I$. Let $V$ be a vector bundle over $T$ with a symmetric isomorphism
\[\tau\colon V\rightarrow V^t.\]
Let $\overline{L}$ be a sub bundle of $\overline{V}=V\otimes_{\mO_T} \mO_{T_0}$ such that $\tau(\overline{L})=(\overline{V}/\overline{L})^t$.
\begin{enumerate}[$(1).$]
\item
Then $\beta(\overline{L}):=(\beta_{ij})$ (see. Diagram~\eqref{eq_map_beta}) is a well-defined element in $H^1(T_0,\mathrm{Sym}^2(\overline{L}^t)\otimes I))$,which does not depend on the choice of local liftings. It is the obstruction to lift $\overline{L}$ globally, that is, $\beta=0$ if and only if there exists a lift $L\subset V$ of $\overline{L}$ such that $\tau(L)=(V/L)^t$.
\item
If $\beta(\overline{L})=0$, the set of isomorphism classes of liftings of $\overline{L}$ is an affine space under $H^0(T_0,\mathrm{Sym}^2(\overline{L}^t)\otimes I))$.
\end{enumerate}
\end{lemma}

\subsection{Identifying obstruction groups via Higgs field}
Let $k$ be a perfect field with characteristic $p>0$ and let $W:=W(k)$ be the ring of $p$-typical Witt vectors. Let $X/\Spec(W)$ be a smooth $W$-scheme. (We will soon add the assumption that $X/\Spec(W)$ is proper, but for now we only require smoothness.) Given a relative normal crossing divisor $D$ on $X$, we set
\[M_D := \{g \in \mO_X \mid g \text{is invertible outside} D\} \subset \mO_X.\]
Then $M_D$ is a fine logarithmic structure on $X$. Let $(X_r,M_{D_r})$ denote the reduction modulo $p^r$. Then for any $s\geq r$, $(X_s,M_{D_s})$ is an object of the (logarithmic crystalline) site $((X_r,M_{D_r})/W)_{crys}^{log}$.

\subsubsection{The classifying mapping and families of abelian varieties}

Let $\overline{\varphi}_r: X_r \rightarrow {\mA_{g,N}}$ be a morphism between schemes. The morphism $\varphi_r$ induces a semistable family of abelian varieties with full $N$-level over $X_r$
\[\mE_{X_r}:=\varphi_r^*(\overline{\mE}_{univ}).\]

\subsubsection{Dieudonn\'e crystal associated to $\mE_{X_r}$ and its realizations}
Denote by $\pi_{{X_k}}\colon \mE_{{X_k}} \rightarrow {X_k}$ the reduction modulo $p$. We let $\bD(\mE_{{X_k}})$ denote the attached logarithmic Dieudonn\'e crystal on $(({X_k},M_{D_{{X_k}}})/W)^{\log}_{crys}$. We denote by $(V_{\mE_{{X_k}}/X_n},\nabla_{\mE_{{X_k}}/X_n})$ the realization of $\bD(\mE_{{X_k}})$ on $(X_n,M_{D_{X_n}})$, furnished by \cite[Theorem 6.2(b) on p. 218]{Kat89}. Take the inverse limit
\[\varprojlim\limits_n (V_{\mE_{{X_k}}/X_n},\nabla_{\mE_{{X_k}}/X_n}) =: (V_{\mE_{{X_k}}/\mX},\nabla_{\mE_{{X_k}}/\mX}),\]
which is a vector bundle over formal scheme $\mX$ with a connection.
We emphasize that $(V_{\mE_{{X_k}}/\mX},\nabla_{\mE_{{X_k}}/\mX})$ and $(V_{\mE_{{X_k}}/X_n},\nabla_{\mE_{{X_k}}/X_n})$ only depend on $\pi_{{X_k}}$.

\begin{remark} The polarization on $\mE_{X_k}$ induces isomorphism
\[\tau\colon (V_{\mE_{{X_k}}/\mX},\nabla_{\mE_{{X_k}}/\mX}) \rightarrow (V_{\mE_{{X_k}}/\mX},\nabla_{\mE_{{X_k}}/\mX})^t.\]
\end{remark}

\begin{remark}
There are several ways of constructing the crystal $\bD(\mE_{{X_k}})$. For instance, one may take relative logarithmic crystalline cohomology of $\pi_{{X_k}}$. Alternatively, if ${U_k}\subset {X_k}$ is the subset over which $\pi_{{X_k}}$ is a smooth morphism of schemes, set $G_{{U_k}}:=\mE_{{U_k}}[p^{\infty}]$. Applying the contravariant Dieudonn\'e functor to $G_{{U_k}}$, we obtain a Dieudonn\'e crystal on ${U_k}$. As $\pi_{\mE_{{X_k}}}$ is semistable, We note that this Dieudonn\'e crystal has logarithmic poles along $D_{{X_k}}={X_k}\backslash {U_k}$.
\end{remark}

\begin{lemma}
\label{independence}
Assume that $\pi'_{X_k}\colon \mE'_{{X_k}}\rightarrow {X_k}$
is another semistable family of abelian variety, which coincides with $\pi_{{X_k}}$ on the open set ${U_k}={X_k}\setminus {D_k}$. Then one has an isomorphism
\[\bD(\mE_{{X_k}}) \simeq \bD(\mE'_{{X_k}}).\]
\end{lemma}

\begin{proof}
Since $\pi'_{X_k}$ and $\pi_{X_k}$ are coincide over $U_k$, one has an isomorphism $\bD(\mE_{{X_k}})\mid_{U_k} \simeq \bD(\mE'_{{X_k}})\mid_{U_k}$ of convergent $F$-isocrystals over $U_k$. Recall \cite[Theorem 5.2.1]{Ked07} and \cite[Theorem 6.4.5]{Ked07}, the composition functor defined by restriction
\[F\textrm{-Isoc}_{\log}(U_k,X_k) \rightarrow F\textrm{-Isoc}^\dagger(U_k) \rightarrow F\textrm{-Isoc}(U_k)\]
is fully faithful from the category of convergent log-$F$-isocrystals over $(U_k,X_k)$ to the category of the convergent $F$-isocrystals over $U_k$. Thus there exists an isomorphism $\bD(\mE_{{X_k}}) \simeq \bD(\mE'_{{X_k}})$ extending $\bD(\mE_{{X_k}})\mid_{U_k} \simeq \bD(\mE'_{{X_k}})\mid_{U_k}$.
\end{proof}

\subsubsection{The Hodge filtration}
By instead taking relative \emph{logarithmic de Rham cohomology} of $\pi_{X_r}$, we obtain a Griffiths-tranverse filtration on $(V_{\mE_{{X_k}}/X_r},\nabla_{\mE_{{X_k}}/X_r})$,which we denote by
\begin{equation}
\label{equ:Fil}
\Fil_{\mE_{X_r}/X_r} \subset (V_{\mE_{{X_k}}/X_r},\nabla_{\mE_{{X_k}}/X_r}).
\end{equation}
\begin{remark}
\label{rmk_HodgeFil_polarization}
This filtration is known as the Hodge bundle and satisfies
\[\tau(\Fil_{\mE_{X_r}/X_r}) \cong \left(V_{\mE_{X_r}/X_r}/\Fil_{\mE_{X_r}/X_r}\right)^t.\]
\end{remark}

\subsubsection{The associated graded Higgs bundle and Kodaira-Spencer map.}
Taking the associated graded Higgs bundle, one gets
\[(E_{\mE_{X_r}},\theta_{\mE_{X_r}})
= \mathrm{Gr}_{\Fil_{\mE_{X_r}/X_r}}
(V_{\mE_{{X_k}}/X_r},\nabla_{\mE_{{X_k}}/X_r})
= (E^{1,0}_{\mE_{X_r}/X_r} \oplus E^{0,1}_{\mE_{X_r}/X_r},\theta_{\mE_{X_r}}),\]
where $E^{1,0}_{\mE_{X_r}/X_r} = R^0\pi_{X_r,*} \omega^1_{\mE_{X_r}/X_r}$, $E^{0,1}_{\mE_{X_r}/X_r} = R^1\pi_{X_r,*} \mO_{\mE_{X_r}}$ and $\theta_{\mE_{X_r}}$ is the graded Higgs field
\[\theta_{\mE_{X_r}}\colon
E^{1,0}_{\mE_{X_r}/X_r}
\rightarrow
E^{0,1}_{\mE_{X_r}/X_r}
\otimes \omega^1_{X_r/W_r}.\]
We rewrite the Higgs field in the form
\[\theta_{\mE_{X_r}}\colon \Theta_{X_r/W}^{\log}\rightarrow \mHom(E^{1,0}_{\mE_{X_r}/X_r},E^{0,1}_{\mE_{X_r}/X_r}),\]
which is also known as Kodaira-Spencer map.
Due to the existence of principal polarization, by \autoref{rmk_HodgeFil_polarization}, the Higgs field $\theta_{\mE_{X_r}}$ facts through, still denoted by $\theta_{\mE_{X_r}}$,
\[\theta_{\mE_{X_r}} \colon \Theta_{X_r/W}^{\log}\rightarrow \mathrm{Sym}^2 \left(E^{1,0}_{\mE_{X_r}/X_r}\right)^t\]

\subsubsection{Dieudonn\'e crystal, Filtered de Rham bundle and Higgs bundle associated to the universal family}
Similarly starting from the universal family abelian varieties over the moduli space $\mA_{g,N}$, one gets Dieudonn\'e crystal $\bD(\overline{\mE}_{univ})$, and filtered logarithmic de Rham bundle over $\overline{\mA_{g,N}}$
\[\Fil_{\mE_{\mA_{g,N}}} \subset (V_{\mE_{\mA_{g,N}}},\nabla_{\mE_{\mA_{g,N}}})\]
And the Kodaira-Spence map
\begin{equation}
\label{theta:universal}
\theta_{\mE_{{\mA_{g,N}}}}\colon
\Theta^{\log}_{{\mA_{g,N}}/W} \longrightarrow \mathrm{Sym}^2 \left( E^{1,0}_{\overline{\mE}_{univ}/{\overline{\mA}_{g,N}}}\right)^t.
\end{equation}

\begin{lemma} [Faltings-Chai]
$\theta_{\mE_{{\mA_{g,N}}}}$ is an isomorphism.
\end{lemma}
\begin{remark}
Recall the associated Higgs bundle sends a local (logarithmic) vector field $\partial$ to
\begin{equation*}
\theta_{\mE_{{\mA_{g,N}}}}(\partial) = \partial \circ \theta_{\mE_{{\mA_{g,N}}}}\colon
\xymatrix{E^{1,0}_{\mE_{{\mA_{g,N}}}/{\mA_{g,N}}} \ar[r]^(0.4){\theta_{\mE_{{\mA_{g,N}}}}} & E^{0,1}_{\mE_{{\mA_{g,N}}}/{\mA_{g,N}}} \otimes \omega^1_{{\mA_{g,N}}} \ar[r]^(0.6)\partial & E^{0,1}_{\mE_{{\mA_{g,N}}}/{\mA_{g,N}}}.}
\end{equation*}
\end{remark}

\begin{remark}
Since the family $\mE_{X_r}$ is the pull back of the universal family via the classifying mapping $\overline{\varphi}_r$, all Dieudonn\'e crystals and Filtered de Rham bundles and Higgs bundles are the pullbacks of those associated to the universal family via the classifying mapping.
\end{remark}

\subsubsection{Identifying the obstruction groups}
By \autoref{obs:map}, the obstruction to lift $\overline{\varphi}_r$ is located in $H^1(X_k,\varphi_k^* \Theta^{\log}_{{\mA_{g,N}}/W})$.
If the obstruction vanishes, then all liftings form an $H^0(X_k,\varphi_k^* \Theta^{\log}_{{\mA_{g,N}}/W})$-torsor.
By \autoref{obs:fil}, the obstruction to lift $\Fil_k$ is located in
\[H^1\left(X_k,\mathrm{Sym}^2 \left(E^{1,0}_{\mE_{X_1}/X_1}\right)^t\right) = H^1\left(X_k,\varphi^*_k\mathrm{Sym}^2 \left(E^{1,0}_{\overline{\mE}_{univ}/{\overline{\mA}_{g,N}}}\right)^t\right).\]
If the obstruction vanishes, then all liftings form a homogeneous space over the group
\[H^0\left(X_k,\mathrm{Sym}^2 \left(E^{1,0}_{\mE_{X_1}/X_1}\right)^t\right) = H^0\left(X_k,\varphi^*_k\mathrm{Sym}^2 \left(E^{1,0}_{\overline{\mE}_{univ}/{\overline{\mA}_{g,N}}}\right)^t\right).\]
By \eqref{theta:universal}, one has an isomorphism between the obstruction groups
\begin{equation}
\label{equ_ObsGroup1}
\varphi_k^*\theta_{\mE_{{\mA_{g,N}}}}\colon H^1(X_k,\varphi_k^* \Theta^{\log}_{{\mA_{g,N}}/W}) \overset\sim\longrightarrow H^1\left(X_k,\varphi_k^* \mathrm{Sym}^2 \left(E^{1,0}_{\overline{\mE}_{univ}/{\overline{\mA}_{g,N}}}\right)^t\right)
\end{equation}
and an isomorphism between the torsor groups
\begin{equation}
\label{equ_ObsGroup2}
\varphi_k^*\theta_{\mE_{{\mA_{g,N}}}}\colon H^0(X_k,\varphi_k^* \Theta^{\log}_{{\mA_{g,N}}/W}) \overset\sim\longrightarrow H^0\left(X_k,\varphi_k^* \mathrm{Sym}^2 \left(E^{1,0}_{\overline{\mE}_{univ}/{\overline{\mA}_{g,N}}}\right)^t\right).
\end{equation}

\subsection{Comparing the obstructions}

In this subsection, we show that to give a lift $\overline{\varphi}_{r+1}\colon X_{r+1}\rightarrow {\overline{\mA}_{g,N}}$ of $\overline{\varphi}_r$ is equivalent to give a lift of the Hodge filtration onto the realization of $\bD(\mE_{{X_k}})$ on $X_{r+1}$. The main result is
\begin{theorem}
\label{thm_main_comparing_obstruction}
\begin{enumerate}[$(1).$]
\item The obstruction of lifting $\overline{\varphi}_r$ maps to the obstruction of lifting the filtration in \eqref{equ:Fil} under the map $\varphi_k^*\theta_{\mE_{{\mA_{g,N}}}}$ in \eqref{equ_ObsGroup1}.
\item Suppose the obstructions vanish. For any lifting $\overline{\varphi}_{r+1}$ of $\overline{\varphi}_r$, one gets $\overline{\varphi}_{r+1}^*(\Fil_{\mE_{univ}/\mA_{g,N}})$ a lifting of the filtration $\Fil_{\mE_{X_r}/X_r}$.
\end{enumerate}
\end{theorem}

Since the obstructions are defined as the differences of local liftings, to show \autoref{thm_main_comparing_obstruction}, one only need to show the following result.
\begin{lemma} Let $U_{g,N}$ be an open subvariety of $\overline{\mA}_{g,N}$. Denote by $U_r:=\varphi_r(U_{g,N})$, which is an open subscheme of $X_r$. Denote by $U_{r+1}$ the open subscheme of $X_{r+1}$, which has the same underlying topological space as $U_{r}$. By shrinking the open subset $U_{g,N}$, we assume there exists a local lifting $\varphi_{r+1}$ of $\overline{\varphi}_r$ over $U_{r+1}$. Denote by $\Fil_{r+1}$ the pullback of $\Fil_{\mE_{univ}/\mA_{g,N}}$ along $\varphi_{r+1}$ which is a lifting of the filtration $\Fil_{\mE_{X_r}/X_r}$ over $U_{r+1}$. Then the following diagram communicates
\begin{equation*}
\xymatrix@R=2cm@C=4cm{
\left\{{\text{all local lifting of} \atop \text{$\overline{\varphi}_{r}$ over $U_{r+1}$}}\right\}
\ar[r]^{\varphi'_{r+1}\mapsto \varphi'^*_{r+1} \left(\Fil_{\mE_{univ}/\mA_{g,N}}\right)}
\ar[d]_{\varphi'_{r+1}\mapsto \varphi'_{r+1}-\varphi_{r+1}}
&
\left\{{\text{all local lifting of} \atop \text{$\Fil_{\mE_{X_r}/X_r}$ over $U_{r+1}$}}\right\}
\ar[d]^{\Fil'_{r+1} \mapsto \Fil'_{r+1} - \Fil_{r+1}}
\\
\left(\varphi_k^* \Theta^{\log}_{{\mA_{g,N}}/W}\right)(U_1)
\ar[r]_{\varphi_k^*\theta_{\mE_{{\mA_{g,N}}}}}
&
\varphi_k^* \mathrm{Sym}^2 \left(E^{1,0}_{\overline{\mE}_{univ}/{\overline{\mA}_{g,N}}}\right)^t(U_1)
\\
}
\end{equation*}
\end{lemma}

\begin{proof}
Let $\varphi_{r+1}$ and $\varphi'_{r+1}$ be two lifting of $\overline{\varphi}_r$ over $U_{r+1}$. Denote by
\[\varphi'_{r+1}-\varphi_{r+1} := \alpha \in \mathrm{Hom}(\varphi_{1}^*\omega_{{\mA_{g,N}}}^1(U_{g,N}),\mO_{X_1}(U_{1})) = \left(\varphi_1^*\Theta^{\log}_{{\mA_{g,N}}/W}\right) (U_{1})\]
defined by the following formula (take $t_1,\cdots,t_d$ such that $\omega_{\mA_{g,N}}(U_{g,N}))$ has basis $\{\mathrm{d}\log(t_i)\}_{1\leq i\leq d}$
\[\alpha(\mathrm{d}\log t_i)= \frac{\varphi'^*_{r+1}(t_i)/\varphi_{r+1}^*(t_i)-1}{p^r} \pmod{p}.\]
Denote by $\varphi_1$ the restriction of $\varphi_r$ on $X_1$. Then
\begin{equation}
\label{alpha:explicite}
\alpha = \sum_{i=1}^{d} \frac{\varphi'^*_{r+1}(t_i)/\varphi_{r+1}^*(t_i)-1}{p^r}\cdot \varphi_1^*\left(t_i\frac{\partial}{\partial t_i}\right)
\end{equation}
Now consider $\bD(\mE)$ the logarithmic crystal associated to the semi-stable family $\overline{\mE}_{univ}/{\overline{\mA}_{g,N}}$.

Denote $\mE_{U_{r+1}}:= \varphi_{r+1}^* \mE_{{\mA_{g,N}}}$ and $\mE'_{U_{r+1}}:= \varphi'^*_{r+1} \mE_{{\mA_{g,N}}}$, then $\mE_{U_{r+1}}\mid_{X_1} = \mE'_{U_{r+1}}\mid_{X_1} =: \mE_{U_1}$ and one has a natural isomorphisms
\[\pi: \varphi_{r+1}^*(V_{\mE_{{\mA_{g,N}}}},\nabla_{\mE_{{\mA_{g,N}}}}) \overset{\sim}\longrightarrow \bD(\mE_{U_1})(U_{r+1},D_{U_{r+1}}),\]
\[\pi': \varphi'^*_{r+1}(V_{\mE_{{\mA_{g,N}}}},\nabla_{\mE_{{\mA_{g,N}}}}) \overset{\sim}\longrightarrow \bD(\mE_{U_1})(U_{r+1},D_{U_{r+1}}).\]
Thus one gets an isomorphism of de Rham bundles
\[\pi'\circ\pi^{-1}:\varphi'^*_{r+1}(V_{\mE_{{\mA_{g,N}}}},\nabla_{\mE_{{\mA_{g,N}}}}) \overset{\sim}\longrightarrow \varphi_{r+1}^*(V_{\mE_{{\mA_{g,N}}}},\nabla_{\mE_{{\mA_{g,N}}}}).\]
Let $(Z,M_Z)$ be the PD-envelope of the diagonal morphism
\[({\mA_{g,N}},M_{D_{{\mA_{g,N}}}}) \rightarrow ({\mA_{g,N}},M_{D_{{\mA_{g,N}}}})\times_W ({\mA_{g,N}},M_{D_{{\mA_{g,N}}}}).\]
and $(Z_1,M_{Z_1})$ the first infinitesimal neighborhood of $({\mA_{g,N}},M_{D_{{\mA_{g,N}}}})$ in $(Z,M_Z)$.

Since $\varphi_{r+1}$ and $\varphi'_{r+1}$ are equal after reduction modulo $p^r$, by the universal property of the first infinitesimal neighborhood the morphism $(\varphi_{r+1},\varphi'_{r+1})$ from $U_{r+1}$ to ${\mA_{g,N}} \times_W {\mA_{g,N}}$ factors through $(Z_1,M_{Z_1})$.
\begin{equation*}
\xymatrix{U_{r+1} \ar[rrr]^(0.4){(\varphi_{r+1},\varphi'_{r+1})} \ar[dr]^{\delta} & & & ({\mA_{g,N}},M_{D_{{\mA_{g,N}}}})\times_W ({\mA_{g,N}},M_{D_{{\mA_{g,N}}}})\\
& (Z_1,M_{Z_1}) \ar[r] & (Z,M_Z) \ar[ur] & \\}
\end{equation*}
Let $p_1,p_2\colon (Z_1,M_{Z_1}) \rightarrow ({\mA_{g,N}},M_{D_{{\mA_{g,N}}}})$ be the first and the second projections, respectively. According to \cite[6.7]{Kat89}, one has an isomorphism $\eta: p_2^* V\simeq p_1^* V$ given by
\begin{equation} \label{eq_Taylor_Conn}
\eta(p_2^*(v)) = p_1^*(v) + \sum_{i=1}^{d} p_1^*\left(\nabla\left(t_i\frac{\partial}{\partial t_i}\right)(v)\right)\cdot \left(\frac{p_2^*(t_i)}{p_1^*(t_i)}-1\right),
\end{equation}
where $v$ is any local section of $V$. Pulling back the isomorphism $\eta$ back onto $U_{r+1}$ via $\delta$ one gets an isomorphism $\delta^*(\eta)$ which is just equal $\pi'\circ\pi^{-1}$ by the definition of pull back of a crystal.

Recall the pull back filtration $\Fil^1_{U_{r+1}}V_{\mE_{U_{r+1}}} := \varphi_{r+1}^*\left(\Fil^1_{\mE_{{\mA_{g,N}}}}V_{\mE_{{\mA_{g,N}}}}\right)$ on $\varphi_{r+1}^*(V_{\mE_{{\mA_{g,N}}}},\nabla_{\mE_{{\mA_{g,N}}}})$. Consider following commutative diagram
\bigskip
\begin{equation*}
\xymatrix@C=1cm{
\varphi'^*_{r+1}\left(\Fil^1_{\mE_{{\mA_{g,N}}}}V_{\mE_{{\mA_{g,N}}}}\right)
\ar@{^(->}[r]
\ar@{..>}@/^25pt/[rrr]
&
\varphi'^*_{r+1}V_{\mE_{{\mA_{g,N}}}} \ar[r]^\simeq_{\pi'\circ\pi^{-1}}
\ar[d]_{\tau}^{\simeq}
&
\varphi_{r+1}^*V_{\mE_{{\mA_{g,N}}}} \ar[d]^{\tau}_{\simeq} \ar@{->>}[r]
&
\varphi_{r+1}^*\left(V_{\mE_{{\mA_{g,N}}}}/\Fil^1_{\mE_{{\mA_{g,N}}}}V_{\mE_{{\mA_{g,N}}}}\right) \ar[d]^{\tau}_{\simeq}
\\
&
\varphi'^*_{r+1}V^t_{\mE_{{\mA_{g,N}}}} \ar[r]^\simeq_{\pi'\circ\pi^{-1}}
&
\varphi_{r+1}^*V^t_{\mE_{{\mA_{g,N}}}},\ar@{->>}[r]
&\varphi_{r+1}^*\left(\Fil^1_{\mE_{{\mA_{g,N}}}}V_{\mE_{{\mA_{g,N}}}}\right)^t.
\\
}
\end{equation*}
Since $\theta_{\mE_{{\mA_{g,N}}}}$ is the associated graded of $\nabla_{\\mE_{{\mA_{g,N}}}}$, by \eqref{eq_Taylor_Conn}, the above dotted arrow is given by
\begin{equation} \label{eq_Taylor_Higgs}
\varphi'^*_{r+1}(v) \mapsto \sum_{i=1}^{d}\varphi_{r+1}^*\left(\theta\left(t_i\frac{\partial}{\partial t_i}\right)(v)\right) \cdot \left(\frac{\varphi^*_{r+1}(t_i)}{\varphi^*_{r+1}(t_i)}-1\right).
\end{equation}
Dividing by $p^r$ and considering the reduction modulo $p$ of the dotted arrow, one gets a morphism of sheaves over $U_{1}$
\[\Fil^1_{U_{1}}V_{\mE_{U_{1}}}
\rightarrow
\frac{V_{\mE_{U_{1}}}}
{\Fil^1_{U_{1}}V_{\mE_{U_{1}}}},\]
By \eqref{eq_map_beta}, the difference between the two filtrations
\[\beta := \Fil'_{r+1} - \Fil_{r+1}\]
is defined as the composition
\[\beta\colon
\Fil^1_{U_{1}}V_{\mE_{U_{1}}}
\rightarrow
\frac{V_{\mE_{U_{1}}}}
{\Fil^1_{U_{1}}V_{\mE_{U_{1}}}} \rightarrow \left(\Fil^1_{U_{1}}V_{\mE_{U_{1}}}\right)^t.\]
By \eqref{eq_Taylor_Higgs}, the composition $\beta$ can be computed explicitly
\begin{equation} \label{beta_ij:explicite}
\beta = \sum_{i=1}^d \frac{\varphi_{r+1}^*(t)/\varphi^*_{r+1}(t)-1}{p^r} \cdot \varphi_1^*\left( \theta_{\mE_{{\mA_{g,N}}}}(\partial/\partial t)\right) = \varphi_k^*\theta_{\mE_{{\mA_{g,N}}}}(\alpha).
\end{equation}
Thus, the lemma follows.
\end{proof}

\section{\bf Arithmetic Simpson Correspondence and $\GL_2$-Motivic Local Systems over $\bP^1\setminus\{0,1,\lambda,\infty\}$}
\label{sec_appendix_A}

This appendix is a record of conference report given in the conference\\
\begin{center}
{\tiny \blue https://irma.math.unistra.fr/\~{}lfu/Activities/Sino-French\%20AG\%20Conference.html}
\end{center}
by the second author.

\autoref{conj:SYZ} predicts that

\begin{itemize}
\item there exists $26$ (classes of) families of elliptic curves $f:\mY \rightarrow \bP^1$ with bad reduction over $D=\{0,1,\lambda,\infty\}$ and the set of zeros $\{(\tau)_0\}$ of the Kodaira-Spence maps equals to $C_\lambda^{\mathrm{tor}}/\{\pm1\}$ of orders $1,2,3,4,6$.
\item in general, there exists families of $g$-dimensional abelian varieties $f:\mY \rightarrow \bP^1$ endowed with real multiplication $L$, with bad reduction over $D$, the set of zeros $\{(\tau)_0\}$ of the Kodaira-Spence maps equals to $C_\lambda^{\mathrm{tor}}/\{\pm1\}$ of orders $d$ and such that $[\bQ(\zeta_d)^+:\bQ]=g$.
\end{itemize}

\newlength\savedwidth
\newcommand\whline{\noalign{\global\savedwidth\arrayrulewidth
\global\arrayrulewidth 3pt}
\hline
\noalign{\global\arrayrulewidth\savedwidth}}
\newlength\savewidth
\newcommand\shline{\noalign{\global\savewidth\arrayrulewidth
\global\arrayrulewidth 1.5pt}
\hline
\noalign{\global\arrayrulewidth\savewidth}}

{\scriptsize
\begin{center}
\begin{tabular}{|c|c|c|c|c|c|}
\hline
$g$ & $d=$ order of $(\tau)_0$ & $L=\bQ(\zeta_d)^+$ & number (of classes) of families \\ \shline
$1$ & 1,\,2,\,3,\,4,\,6 & $\bQ$ & $26$ \\ \shline
\multirow{4}{*}{2} & 5 & $\bQ(\sqrt{5})$ & 6 \\ \cline{2-4}
& 8 &$\bQ(\sqrt{2})$& 12\\ \cline{2-4}
& 10 &$\bQ(\sqrt{10})$& 18 \\ \cline{2-4}
& 12 &$\bQ(\sqrt{3})$& 24 \\ \shline
\multirow{4}{*}{3} & 7 & $\bQ(\zeta_{7}+\zeta_{7}^{-1})$ & 8 \\ \cline{2-4}
& 9 & $\bQ(\zeta_{9}+\zeta_{9}^{-1})$ & 12 \\ \cline{2-4}
& 14 & $\bQ(\zeta_{14}+\zeta_{14}^{-1})$ & 24\\ \cline{2-4}
& 18 & $\bQ(\zeta_{18}+\zeta_{18}^{-1})$ & 36 \\ \shline
\multirow{5}{*}{4} & 15 & $\bQ(\zeta_{15}+\zeta_{15}^{-1})$ & 24
\\ \cline{2-4}
& 16 & $\bQ(\zeta_{16}+\zeta_{16}^{-1})$ & 24 \\ \cline{2-4}
& 20 & $\bQ(\zeta_{20}+\zeta_{20}^{-1})$ & 36 \\ \cline{2-4}
& 24 & $\bQ(\zeta_{24}+\zeta_{24}^{-1})$ & 48 \\ \cline{2-4}
& 30 & $\bQ(\zeta_{30}+\zeta_{30}^{-1})$ & 72 \\ \hline
$\vdots$& $\vdots$ & $\vdots$ & $\vdots$ \\ \hline
\end{tabular}
\end{center}
}

J. Lu, X. Lv and J. Yang found that there indeed exist $26$ (classes of) families of elliptic curves, which are list as in the following table

\renewcommand\arraystretch{2}
\begin{adjustwidth}{-6mm}{0mm}
\tiny
\begin{tabular}{|c|l|c|l|}
\cline{1-3}
{order of $(\tau)_0$} & elliptic curves $\mY/\bP^1$ with bad reductions over $\{0,1,\infty,\lambda\}$ & $a$ \\ \shline
{1} & $y^2=x(x-t+\lambda)(x-t+\lambda t)$ & $-$ \\ \shline
\multirow{3}{*}{2} & $y^2=(x-1)(x-\lambda)(x-t)$ & $-$ \\ \cline{2-3}
& $y^2=x(x-\lambda)(x-t)$ & $-$ \\ \cline{2-3}
& $y^2=x(x-1)(x-t)$ & $-$ \\ \shline
{3} & $y^2=x^3+\frac{(a-3)^2t-4a}{4(a-1)}x^2-\frac{a-3}{2}tx+\frac{a-1}{4}t$ & $\lambda(a+1)(a-3)^3+16a^3=0$ \\ \shline
\multirow{3}{*}{4} & $y^2=x^3+4(t-a)x^2+(t-1)(t-\lambda)x $ & $a^2-\lambda=0$ \\ \cline{2-3}
& $y^2=x^3+4(t-a)x^2+t(t-\lambda)x$ & $a^2-2a+\lambda=0$ \\ \cline{2-3}
& $y^2=x^3+4(t-a)x^2+(t^2-t)x$ & $a^2-2\lambda a+\lambda=0$ \\ \shline
\multirow{3}{*}{6} & $y^2=(1-t)x^3+\frac{(a-3)^2-4a(1-t)}{4(a-1)}x^2-\frac{a-3}{2}x+\frac{a-1}{4}$ & $(a+1)(a-3)^3+16(1-\lambda) a^3=0$ \\ \cline{2-3}
& $y^2=(\lambda-t)x^3+\frac{(a-3)^2\lambda-4a(\lambda-t)}{4(a-1)}x^2-\frac{a-3}{2}\lambda x+\frac{a-1}{4}\lambda$ & $\lambda(a+1)(a-3)^3+16(\lambda-1)a^3=0$ \\ \cline{2-3}
& $y^2=tx^3+\frac{(a-3)^2-4at}{4(a-1)}x^2-\frac{a-3}{2}x+\frac{a-1}{4}$ & $(a+1)(a-3)^3+16\lambda a^3=0$ \\ \cline{1-3}
\end{tabular}
\end{adjustwidth}

\def\fai{\varphi_{\lambda,{p}}}

The self map $\fai=\mathrm{Gr}\circ \mC_{1,2}^{-1}$ on $\mM_{\mathbb{F}_q}\simeq \bP^1_{\mathbb{F}_q}$ induced by Higgs-de Rham flow
has the following explicit form
\begin{equation*}
\fai(z)=\frac{z^p}{\lambda^{p-1}}\cdot\left( \frac{f_\lambda(z^p)}{g_\lambda(z^p)}\right)^2,
\end{equation*}
where {\tiny
\[ \makebox[-5mm]{} f_\lambda(z^p)=\det\left(\begin{array}{ccccc} \frac{\lambda^p(1-z^p)-(\lambda^p-z^p)\lambda^{2}}{2}&
\frac{\lambda^p(1-z^p)-(\lambda^p-z^p)\lambda^{3}}{3} &\cdots&
\frac{\lambda^p(1-z^p)-(\lambda^p-z^p)\lambda^{(p+1)/2}}{(p+1)/2} \\ \frac{\lambda^p(1-z^p)-(\lambda^p-z^p)\lambda^{3}}{3} &
\frac{\lambda^p(1-z^p)-(\lambda^p-z^p)\lambda^{4}}{4} &\cdots&
\frac{\lambda^p(1-z^p)-(\lambda^p-z^p)\lambda^{(p+3)/2}}{(p+3)/2} \\ \vdots&\vdots&\ddots&\vdots\\ \frac{\lambda^p(1-z^p)-(\lambda^p-z^p)\lambda^{(p+1)/2}}{(p+1)/2} &
\frac{\lambda^p(1-z^p)-(\lambda^p-z^p)\lambda^{(p+3)/2}}{(p+3)/2} &\cdots&
\frac{\lambda^p(1-z^p)-(\lambda^p-z^p)\lambda^{p-1}}{p-1} \\
\end{array} \right)\] }
and
{\tiny \[\makebox[-5mm]{} g_\lambda(z^p)=\det\left(\begin{array}{ccccc} \frac{\lambda^p(1-z^p)-(\lambda^p-z^p)\lambda^1}{1}&
\frac{\lambda^p(1-z^p)-(\lambda^p-z^p)\lambda^{2}}{2} &\cdots&
\frac{\lambda^p(1-z^p)-(\lambda^p-z^p)\lambda^{(p-1)/2}}{(p-1)/2} \\ \frac{\lambda^p(1-z^p)-(\lambda^p-z^p)\lambda^{2}}{2} &
\frac{\lambda^p(1-z^p)-(\lambda^p-z^p)\lambda^{3}}{3} &\cdots&
\frac{\lambda^p(1-z^p)-(\lambda^p-z^p)\lambda^{(p+1)/2}}{(p+1)/2} \\ \vdots&\vdots&\ddots&\vdots\\ \frac{\lambda^p(1-z^p)-(\lambda^p-z^p)\lambda^{(p-1)/2}}{(p-1)/2} &
\frac{\lambda^p(1-z^p)-(\lambda^p-z^p)\lambda^{(p+1)/2}}{(p+1)/2} &\cdots&
\frac{\lambda^p(1-z^p)-(\lambda^p-z^p)\lambda^{p-2}}{p-2} \\
\end{array} \right).\]}

\begin{equation*}
\begin{split}
&\varphi_{\lambda,3}(z)= z^3 \left(\frac{z^3+\lambda(\lambda+1)}{(\lambda+1)z^3+\lambda^2}\right)^2;\\
& \varphi_{-1,3}(z)=z^{3^2};\\
&\makebox[16cm]{}\\
& \varphi_{\lambda,5}(z)= z^5\left(\frac{z^{10}-\lambda(\lambda+1)(\lambda^2-\lambda+1)z^5+\lambda^4(\lambda^2-\lambda+1)}{(\lambda^2-\lambda+1)z^{10}-\lambda^2(\lambda+1)(\lambda^2-\lambda+1)z^5+\lambda^6 }\right)^2;\\
&\varphi_{\lambda,5}(z)=z^{5^2} \text{ if and only if $\lambda$ is a $6$-th primitive root of unit}; \\
\end{split}
\end{equation*}

For $k=\bF_{3^4}$ and $\lambda\in k\setminus\{0,1\}$, the map $\varphi_{\lambda,3}$ is a self $k$-morphism on $\bP^1_k$
\[\varphi_{\lambda,3}:k\cup \{\infty\}\rightarrow k\cup \{\infty\}.\]
For $\alpha=\sqrt{1+\sqrt{-1}}$ as a generator of $k=\bF_{3^4}$ over $\bF_3$, every elements in $k$ can be uniquely expressed in form $a_3\alpha^3+a_2\alpha^2+a_1\alpha+a_0$, where $a_3,a_2,a_1,a_0\in\{0,1,2\}$.

We use the integer $3^3a_3+3^2a_2+3a_1+a_0\in [0,80]$ stand for the element $a_3\alpha^3+a_2\alpha^2+a_1\alpha+a_0\in k$
§\[\varphi_{\lambda,3}:\{0,1,2,\cdots,80,\infty\}\rightarrow \{0,1,2,\cdots,80,\infty\}.\]§

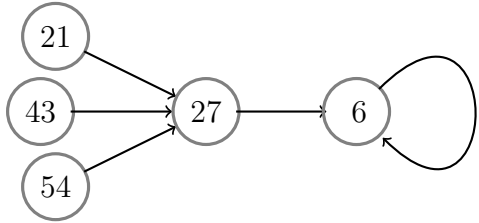
\begin{center}
\begin{tikzpicture}
[L1Node/.style={circle,draw=black!50, very thick, minimum size=7mm}]
\node[L1Node] (n1) at (-2, 1){$21$};
\draw [thick,->](-1.62,0.8) -- (-0.4,0.2);
\node[L1Node] (n1) at (-2.2, 0){$43$};
\draw [thick,->](-1.8,0) -- (-0.45,0);
\node[L1Node] (n1) at (-2, -1){$54$};
\draw [thick,->](-1.62,-0.8) -- (-0.4,-0.2);
\node[L1Node] (n1) at (0, 0){$27$};
\draw [thick,->](0.4,0) -- (1.6,0);
\node[L1Node] (n1) at (2, 0){$~6\,$};
\draw [thick,->](2.3,0.3) .. controls (4,2) and (4,-2) .. (2.35,-0.35);
\end{tikzpicture}
\end{center}
\[\rho: \pi_1\left(\bP_{W(\bF_{3^4})[1/3]}^1\setminus\left\{0,1,\infty,2\sqrt{1+\sqrt{-1}}\right\}\right)\longrightarrow \mathrm{GL}_2(\bF_3).\]

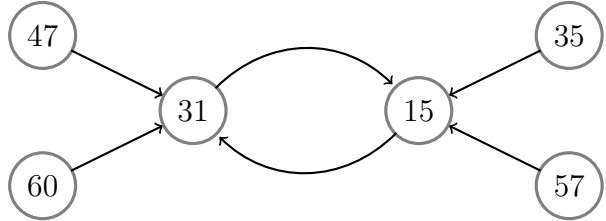
\begin{center}
\begin{tikzpicture}
[L1Node/.style={circle,draw=black!50, very thick, minimum size=7mm}]
\node[L1Node] (n1) at (0, 1){$47$};
\draw [thick,->](0.38,0.8) -- (1.6,0.2);
\node[L1Node] (n1) at (0, -1){$60$};
\draw [thick,->](0.38,-0.8) -- (1.6,-0.2);
\node[L1Node] (n1) at (2, 0){$31$};
\draw [thick,->](2.3,0.3) .. controls (3,1) and (4,1) .. (4.65,0.35);
\node[L1Node] (n1) at (7, 1){$35$};
\draw [thick,->](6.62,0.8) -- (5.4,0.2);
\node[L1Node] (n1) at (7, -1){$57$};
\draw [thick,->](6.62,-0.8) -- (5.4,-0.2);
\node[L1Node] (n1) at (5, 0){$15$};
\draw [thick,->](4.7,-0.3) .. controls (4,-1) and (3,-1) .. (2.35,-0.35);
\end{tikzpicture}
\end{center}

\[\rho: \pi_1\left(\bP_{W(\bF_{3^4})[1/3]}^1\setminus\left\{0,1,\infty,\sqrt{-1}\right\}\right)\longrightarrow \mathrm{GL}_2(\bF_{3^2});\]

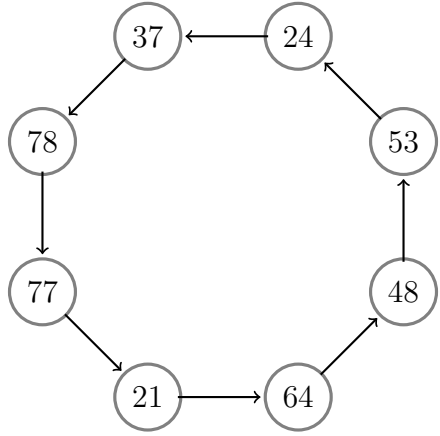
\begin{center}
\begin{tikzpicture}
[L1Node/.style={circle,draw=black!50, very thick, minimum size=7mm}]
\node[L1Node] (n1) at (-1,-2.4){$21$};
\draw [thick,->](-0.6,-2.4) -- (0.5,-2.4);
\node[L1Node] (n1) at (1,-2.4){$64$};
\draw [thick,->](1.3,-2.1) -- (2.05,-1.35);
\node[L1Node] (n1) at (2.4,-1){$48$};
\draw [thick,->](2.4,-0.6) -- (2.4,0.5);
\node[L1Node] (n1) at (2.4,1){$53$};
\draw [thick,->](2.1,1.3) -- (1.35,2.05);
\node[L1Node] (n1) at (1,2.4){$24$};
\draw [thick,->](0.6,2.4) -- (-0.5,2.4);
\node[L1Node] (n1) at (-1,2.4){$37$};
\draw [thick,->](-1.3,2.1) -- (-2.05,1.35);
\node[L1Node] (n1) at (-2.4,1){$78$};
\draw [thick,->](-2.4,0.6) -- (-2.4,-0.5);
\node[L1Node] (n1) at (-2.4,-1){$77$};
\draw [thick,->](-2.1,-1.3) -- (-1.35,-2.05);
\end{tikzpicture}
\end{center}
\[ \rho: \pi_1\left(\bP_{W(\bF_{3^8})[1/3]}^1\setminus\left\{0,1,\infty,\sqrt{-1}\right\}\right)\longrightarrow \mathrm{GL}_2(\bF_{3^8}).\]

\def\hzero{0}\def\hone{1.2}\def\htwo{2.4}\def\hthree{3.6}\def\hfour{4.8}\def\hfive{5.4}
\def\hsix{5.76}\def\hseven{6}\def\height{6.18}\def\hend{7.3}\def\hinf{7.5}
\def\wzero{0}\def\wone{2}\def\wtwo{4}\def\wthree{6}\def\wfour{8}\def\wfive{9}\def\wsix{9.6}
\def\wseven{10}\def\weight{10.3}\def\wend{12.3}\def\winf{12.5}
\def\rang{-60}

\begin{adjustwidth}{-8mm}{0mm}
\begin{center}
\begin{tikzpicture}
\filldraw[black] (\wzero,\hzero) circle (2pt) node[below]{};
\draw[black, thick,->] (\wzero,\hzero) -- (\wzero,\hend) node[black,below right]{$\mathrm{GL(\hat{\bZ})}$};

\filldraw[blue] (\wzero,\hone) circle (2pt) node[black,left] {$\mathrm{GL}(\bQ_2)$};
\draw[blue, thick] (\wzero,\hone) -- (\wend,\hone);

\filldraw[blue] (\wzero,\htwo) circle (2pt) node[black,left] {$\mathrm{GL}(\bQ_3)$};
\draw[blue, thick] (\wzero,\htwo) -- (\wend,\htwo);

\filldraw[blue] (\wzero,\hthree) circle (2pt) node[black,left] {$\mathrm{GL}(\bQ_5)$};
\draw[blue, thick] (\wzero,\hthree) -- (\wend,\hthree);

\filldraw[blue] (\wzero,\hfour) circle (2pt) node[black,left] {$\mathrm{GL}(\bQ_p)$};
\draw[blue, thick] (\wzero,\hfour) -- (\wend,\hfour);

\filldraw[blue] (\wzero,\hfive) circle (1.5pt) node[black,left] {$\mathrm{GL}(\bQ_\ell)$};
\draw[blue, thick] (\wzero,\hfive) -- (\wend,\hfive);
\node[blue] at (3.5,5.5){$\ell$-adic representation(Fontaine-Mazur)}; 

\filldraw[blue] (\wzero,\hsix) circle (1.4pt);
\filldraw[blue] (\wzero,\hseven) circle (1.2pt);
\filldraw[blue] (\wzero,\height) circle (1pt);

\filldraw[blue] (\wzero,\hinf) circle (2pt) node[black,left] {$\mathrm{GL}(\bC)$};
\draw[blue, thick] (\wzero,\hinf) -- (\winf,\hinf);
\draw[black, thick,->] (\wzero,\hzero) -- (\wend,\hzero) node[above left]{$\pi^{\mathrm{et}}_1(X/\bQ,*)$};

\begin{rotate}{-90}{ }\end{rotate}

\filldraw[green] (\wone,\hzero) circle (2pt) node[black,below] {\tiny $\pi^{\mathrm{et}}_1(X/\bQ_2,*)$ };
\draw[green, thick] (\wone,0) -- (\wone,\hend);

\filldraw[green] (\wtwo,\hzero) circle (2pt) node[black,below] {\tiny $\pi^{\mathrm{et}}_1(X/\bQ_3,*)$ };
\draw[green, thick] (\wtwo,\hzero) -- (\wtwo,\hend);

\filldraw[green] (\wthree,\hzero) circle (2pt) node[black,below] {\tiny $\pi^{\mathrm{et}}_1(X/\bQ_5,*)$ };

\draw[green, thick] (\wthree,\hzero) -- (\wthree,\hend);

\filldraw[green] (\wfour,\hzero) circle (2pt) node[black,below] {\tiny $\pi^{\mathrm{et}}_1(X/\bQ_p,*)$ };
\draw[green, thick] (\wfour,\hzero) -- (\wfour,\hend);

\filldraw[green] (\wfive,\hzero) circle (1.5pt);
\draw[green, thick] (\wfive,\hzero) -- (\wfive,\hend);
\filldraw[green] (\wsix,\hzero) circle (1.4pt);
\filldraw[green] (\wseven,\hzero) circle (1.2pt);
\filldraw[green] (\weight,\hzero) circle (1pt);

\filldraw[green] (\winf,\hzero) circle (2pt) node[black,below] {\tiny $\pi^{\mathrm{top}}_1(X/\mathbb{C},*)$};
\draw[green, thick] (\winf,\hzero) -- (\winf,\hinf);
\draw[red, thick] (\wzero,\hzero) -- (13,7.8);
\filldraw[black] (\wone,\hone) circle (2pt) node[below right] {Crystalline};
\filldraw[black] (\wone,\htwo) circle (1.5pt);
\filldraw[black] (\wone,\hthree) circle (1.5pt);
\filldraw[black] (\wone,\hfour) circle (1.5pt);
\filldraw[black] (\wone,\hfive) circle (1.5pt);
\filldraw[black] (\wone,\hsix) circle (1.4pt);
\filldraw[black] (\wone,\hseven) circle (1.2pt);
\filldraw[black] (\wone,\height) circle (1pt);

\filldraw[black] (\wtwo,\hone) circle (1.5pt);
\filldraw[black] (\wtwo,\htwo) circle (2pt) node[below right] {Crystalline};
\filldraw[black] (\wtwo,\hthree) circle (1.5pt);
\filldraw[black] (\wtwo,\hfour) circle (1.5pt);
\filldraw[black] (\wtwo,\hfive) circle (1.5pt);
\filldraw[black] (\wtwo,\hsix) circle (1.4pt);
\filldraw[black] (\wtwo,\hseven) circle (1.2pt);
\filldraw[black] (\wtwo,\height) circle (1pt);

\filldraw[black] (\wthree,\hone) circle (1.5pt);
\filldraw[black] (\wthree,\htwo) circle (1.5pt);
\filldraw[black] (\wthree,\hthree) circle (2pt) node[below right] {Crystalline};
\filldraw[black] (\wthree,\hfour) circle (1.5pt);
\filldraw[black] (\wthree,\hfive) circle (1.5pt);
\filldraw[black] (\wthree,\hsix) circle (1.4pt);
\filldraw[black] (\wthree,\hseven) circle (1.2pt);
\filldraw[black] (\wthree,\height) circle (1pt);

\filldraw[black] (\wfour,\hone) circle (1.5pt);
\filldraw[black] (\wfour,\htwo) circle (1.5pt);
\filldraw[black] (\wfour,\hthree) circle (1.5pt);
\filldraw[black] (\wfour,\hfour) circle (2pt) node[below right] {Crystalline};
\filldraw[black] (\wfour,\hfive) circle (1.5pt);
\filldraw[black] (\wfour,\hsix) circle (1.4pt);
\filldraw[black] (\wfour,\hseven) circle (1.2pt);
\filldraw[black] (\wfour,\height) circle (1pt);

\filldraw[black] (\wfive,\hone) circle (1.5pt);
\filldraw[black] (\wfive,\htwo) circle (1.5pt);
\filldraw[black] (\wfive,\hthree) circle (1.5pt);
\filldraw[black] (\wfive,\hfour) circle (1.5pt);
\filldraw[black] (\wfive,\hfive) circle (1.5pt);

\filldraw[black] (\wsix,\hone) circle (1.4pt);
\filldraw[black] (\wsix,\htwo) circle (1.4pt);
\filldraw[black] (\wsix,\hthree) circle (1.4pt);
\filldraw[black] (\wsix,\hfour) circle (1.4pt);
\filldraw[black] (\wsix,\hsix) circle (1.4pt);

\filldraw[black] (\wseven,\hone) circle (1.2pt);
\filldraw[black] (\wseven,\htwo) circle (1.2pt);
\filldraw[black] (\wseven,\hthree) circle (1.2pt);
\filldraw[black] (\wseven,\hfour) circle (1.2pt);
\filldraw[black] (\wseven,\hseven) circle (1.2pt);

\filldraw[black] (\weight,\hone) circle (1pt);
\filldraw[black] (\weight,\htwo) circle (1pt);
\filldraw[black] (\weight,\hthree) circle (1pt);
\filldraw[black] (\weight,\hfour) circle (1pt);
\filldraw[black] (\weight,\height) circle (1pt);
\filldraw[black] (\winf,\hinf) circle (2pt) node[below right] {Hodge Type};
\node[red] at (1,0.8){\begin{rotate}{31}{Arithmetic Fontaine-Faltings module}\end{rotate}}; 
\draw[green, thick,->] (9.4,3)--(\wfive,3);
\node[green] at (12.3,3){$\ell$ to $\ell'$(or $p$) companions(Deligne)}; 
\end{tikzpicture}
\end{center}
\end{adjustwidth}

\section{\bf The torsor map induced by Higgs-de Rham flow}\label{sec_appendix_B}

To compute the torsor map induces by Higgs-de Rham flow, we recall the explicit construction of the inverse Cartier functor in curve case and give some notations used in the computation. For the general case, see the appendix of \cite{LSZ19}.

\subsubsection{setup}
Let $k$ be a perfect field of characteristic $p\geq 3$. Let $W=W(k)$ be the ring of Witt vectors and $W_n=W/p^n$ for all $n\geq1$ and $\sigma:W\rightarrow W$ be the Frobenius map on $W$. Let $X$ be a smooth algebraic curve over $W$ and $\overline{D}$ be a relative simple normal crossing divisor.

For a sufficiently small open affine subset $U$ of $X$, by \cite[Proposition 9.7 and Proposition 9.9]{EsVi92}, one gets the existence of log Frobenius lifting over the $p$-adic completion $\widehat{U}$ of $U$, respecting the divisor $\widehat{D}$. We choose a covering of affine open subsets $\{U_i\}_{i\in I}$ of $X$ together with a log Frobenius lifting $F_i:\widehat{U}_i\rightarrow \widehat{U}_i$, respecting the divisor $\widehat{D}\cap \widehat{U}_i$ for each $i\in I$. Denote $R_i=\mO_{X}(U_i)$, $R_{ij}=\mO_{X_2}(U_{ij})$. Then $\Phi_i:=F_i^\#$ is an continuous ring endomorphism of the $p$-adic completion of $R_i$
\[\Phi_i=F_i^\#: \widehat{R}_i\rightarrow \widehat{R}_i.\]
For any object $\aleph$ (e.g. open subsets, divisors, sheaves, etc.) over $X_n=X\otimes_WW_n$, over $X$ or over $\widehat{X}$, we denote by $\overline{\aleph}$ its reduction on $X_1$. Denote by $\Phi$ the $p$-th power map on all rings of characteristic $p$. Thus $\overline{\Phi}_i=\Phi$ on $\overline{R}_{ij}$.

Since $F_i$ is a log Frobenius lifting, $\mathrm{d}\Phi_i$ is divisible by $p$ and which induces a map
\begin{equation*}
\frac{\mathrm{d}\Phi_i}{p}:\Omega_{\widehat{X}}^1(\mathrm{log} \widehat{D})(\widehat{U}_i)\otimes_{\Phi_i} \widehat{R}_i \rightarrow \Omega_{\widehat{X}}^1(\mathrm{log} \widehat{D})(\widehat{U}_i). \eqno{(\frac{\mathrm{d}\Phi_i}{p})}
\end{equation*}

\subsection{Principle of the calculation}

Let $(E,\theta,\acute{V},\acute{\nabla},\acute{\Fil},\psi)$ be an object in the category $\sH((X_{n+1},D_{n+1}))$. In other words, the pair $(E,\theta)$ is a logarithmic graded Higgs bundle with nilpotent Higgs field over $X_n$ of exponent$\leq p-1$ and $\psi$ is an isomorphism of graded Higgs bundles
\[\psi\colon \mathrm{Gr}(\acute{V},\acute{\nabla},\acute{\Fil}) \rightarrow (E,\theta)\pmod{p^n}.\]
Now, we give the construction of the de Rham bundle defined by inverse Cartier functor
\[C^{-1}_{X_n\subset X_{n+1}}((E,\theta,\acute{V},\acute{\nabla},\acute{\Fil},\psi)).\]
Explicitly, the inverse Cartier functor $C^{-1}_{X_n\subset X_{n+1}}$ is a composition of $\mT$-functor and the Frobenius pullback.

\subsubsection{$\mT$-functor}
By the functor $\mT_{n+1}$, one constructs a filtered $p$-connection over $X_{n+1}$
\[(\widetilde{V},\widetilde{\nabla}) = \mT_{n+1}((E,\theta,\acute{V},\acute{\nabla},\acute{\Fil},\psi))\]
\begin{lemma}
Let $(E^\epsilon,\theta^\epsilon)$ be another lifting of $(E,\theta)\pmod{p^n}$ over $X_{n+1}$ with difference
\[\epsilon:=((E^\epsilon,\theta^\epsilon)-(E,\theta)) \in H^1_{Hig}((E,\theta)).\]
Denote by
\[(\widetilde{V}^\epsilon,\widetilde{\nabla}^\epsilon):= \mT_{n+1}((E,\theta,\acute{V},\acute{\nabla},\acute{\Fil},\psi))\]
Then
\[(\widetilde{V}^\epsilon,\widetilde{\nabla}^\epsilon)-(\widetilde{V},\widetilde{\nabla}) = \epsilon.\]
\end{lemma}
\begin{proof}
This follows the direct computation.
\end{proof}

\subsubsection{Frobenius pullback}
Locally we set
\begin{equation*}
\begin{split}
&V_i =\widetilde{V}(U_i)\otimes_{\Phi_i} \widehat{R}_i,\\
&\nabla_i = \mathrm{d} + \frac{\mathrm{d}\Phi_i} {p}(\widetilde{\nabla} \otimes_\Phi1): V_i\rightarrow V_i\otimes_{R_i} \Omega_{X}^1(\mathrm{log} D)(U_i),\\
\end{split}
\end{equation*}
and the gluing isomorphism
\[G_{ij} \colon V_i\mid_{U_{ij}}\rightarrow V_j\mid_{U_{ij}}\]
is given by
\[G_{ij}(e\otimes_{\Phi_i}1) = \sum_{J=0}^{\infty} \frac{\widetilde{\nabla}^J(\partial_{t_{ij}})}{J!}(e)\otimes_{\Phi_j}\left(\frac{\Phi_i(t_{ij})-\Phi_j(t_{ij})}{p}\right)^J\]
for any $e\in\widetilde{V}$ where $t_{ij}$ is a local parameter over $U_{ij}$.

Those local data $(V_i,\nabla_i)$'s can be glued into a global sheaf $V$ with an integrable connection $\nabla$ via the transition maps $\{ G_{ij} \}$. The inverse Cartier functor on $(E,\theta)$ is defined by
\[C^{-1}_{X_n\subset X_{n+1}}((E,\theta,\acute{V},\acute{\nabla},\acute{\Fil},\psi)):=(V,\nabla).\]

Let $\widetilde{v}_{i,\cdot}=\{\widetilde{v}_{i,1},\widetilde{v}_{i,2},\cdots,\widetilde{v}_{i,r}\}$ be a basis of $\widetilde{V}(\overline{U}_i)$ and denote by $e_{i,\cdot}$ the associated basis of the graded Higgs bundle $(E,\theta)$. Then
\[\Phi_i^*\widetilde{v}_{i,\cdot}:=\{\widetilde{v}_{i,1}\otimes_{\Phi_i} 1,\widetilde{v}_{i,2}\otimes_{\Phi_i} 1,\cdots,\widetilde{v}_{i,r}\otimes_{\Phi_i} 1\}\]
forms a basis of $V_i$. Now under those basis, there are $r\times r$-matrices $\omega_{\theta,i}$, $\omega_{\widetilde{\nabla},i}$, $\omega_{\nabla,i}$ with coefficients in $\Omega_{X_1}^1(\mathrm{log}\overline{D})(\overline{U}_i)$, and matrices $\mF_{ij}$, $\mG_{ij}$ over $\overline{R}_{ij}$, such that
\begin{equation*}
(\widetilde{v}_{i,\cdot})=(\widetilde{v}_{j,\cdot}) \cdot \mF_{ij} \eqno{(\mF_{ij})}
\end{equation*}
\begin{equation*}
\widetilde{\nabla}(\widetilde{v}_{i,\cdot})=(\widetilde{v}_{i,\cdot}) \cdot \omega_{\widetilde{\nabla},i} \eqno{(\omega_{\widetilde{\nabla},i})}
\end{equation*}
\begin{equation*}
\nabla_i(\Phi_i^*\widetilde{v}_{i,\cdot})=(\Phi_i^*\widetilde{v}_{i,\cdot})\cdot \omega_{\nabla,i} \eqno{(\omega_{\nabla,i})}
\end{equation*}
\begin{equation*}
G_{ij}(\Phi^*\widetilde{v}_{i,\cdot})=(\Phi^*\widetilde{v}_{j,\cdot})\cdot \mG_{ij} \eqno{(\mG_{ij})}
\end{equation*}

Similarly for any other tuple $(E^\epsilon,\theta^\epsilon,\acute{V},\acute{\nabla},\acute{\Fil},\psi)$, one can similarly define
\[V_i^\epsilon,\nabla_i^\epsilon,G_{ij}^\epsilon,\widetilde{v}_{i,\cdot}^\epsilon,e_{i,\cdot}^\epsilon,\Phi_i^*\widetilde{v}_{i,\cdot}^\epsilon,\mF_{ij}^\epsilon,\omega_{\widetilde{\nabla}^\epsilon,i},\omega_{\nabla^\epsilon,i},\mG_{ij}^\epsilon,\cdots.\]
\begin{lemma} \label{lem_diff_p_conn}
Suppose $(E,\theta)$ and $(E^\epsilon,\theta^\epsilon)$ has the same underlying bundle $E=E^\epsilon$. Then
\begin{enumerate}[$(1).$]
\item $\epsilon$ can be represented by $\frac{\theta^\epsilon-\theta}{p^n}$;
\item $\widetilde{V}=\widetilde{V}^\epsilon$, in this case, we can pick basis such that $\widetilde{v}_{i,\cdot} = \widetilde{v}_{i,\cdot}^\epsilon$;
\item $\theta^\epsilon-\theta = \widetilde{\nabla}^\epsilon - \widetilde{\nabla}$.
\item If $\theta(\partial)^2=0$, then $\mG_{ij}^\epsilon - \mG_{ij} =\Phi_j(\mF_{12}))\cdot \Phi_j(\frac{\theta^\epsilon-\theta}{\mathrm{d}t_{ij}})\cdot\frac{\Phi_i(t_{ij})-\Phi_j(t_{ij})}{p}$
\end{enumerate}
\begin{proof}
This follows the definition of the $\mT$-functor.
\end{proof}

\end{lemma}

\subsubsection*{Computation of our example:} Let $\lambda\in W_2(k)$ with $\lambda\not\equiv 0,1\pmod{p}$ and let $X_2=\mathrm{Proj}\,W_2[T_0,T_1]$. Let $D$ be the divisor of $X_2$ associated to the homogeneous ideal $(T_0T_1(T_1-T_0)(T_1-\lambda T_0))$. By using $t=T_0^{-1}T_1$ as a parameter, we can simply write $D=\{0,1,\lambda,\infty\}$. Denote $U_1=X_2 \setminus \{0,\infty\}$,$U_2=X_2 \setminus \{1,\lambda\}$,$D_1=\{1,\lambda\}$ and $D_2=\{0,\infty\}$. Then $\{U_1,U_2\}$ forms a covering of $X_2$,
\[\begin{split}
& R_1=\mO(U_1)=W_2[t,\frac1t],\\
& R_2=\mO(U_2)=W_2[\frac{t-\lambda}{t-1},\frac{t-1}{t-\lambda}],\\
& R_{12}=\mO(U_1\cap U_2)=W_2[t,\frac1t,\frac{t-\lambda}{t-1},\frac{t-1}{t-\lambda}],\\
& \Omega_{X_2}^1(\log D)(U_1)=W_2[t,\frac1t]\cdot \mathrm{d}\log\left(\frac{t-\lambda}{t-1}\right),\\
& \Omega_{X_2}^1(\log D)(U_2)=W_2[\frac{t-\lambda}{t-1},\frac{t-1}{t-\lambda}]\cdot \mathrm{d}\log t.\\
\end{split}\]
Over $U_{12}$, one has
\[\mathrm{d}\log \left(\frac{t-\lambda}{t-1}\right)=\frac{(\lambda-1)t}{(t-\lambda)(t-1)}\cdot \mathrm{d}\log t.\]
Denote $\Phi_1(\frac{t-\lambda}{t-1})=\left(\frac{t-\lambda}{t-1}\right)^p$ and $\Phi_2(t)=t^p$, which induce two Frobenius liftings on $R_{12}$. One checks that $\Phi_i$ can be restricted on $R_i$ and forms a log Frobenius lifting respecting the divisor $D_i$. Moreover
\begin{equation}\label{equ:dF_1/p}
\frac{\mathrm{d} \Phi_1}{p}\left(\mathrm{d}\log \frac{t-\lambda}{t-1} \otimes_{\Phi} 1\right) = \mathrm{d}\log \frac{t-\lambda}{t-1},
\end{equation}
and
\begin{equation}\label{equ:dF_2/p}
\frac{\mathrm{d} \Phi_2}{p}\left(\mathrm{d}\log t \otimes_{\Phi} 1\right) = \mathrm{d}\log t.
\end{equation}

\paragraph{\emph{Local expressions of the Higgs field and the de Rham bundle.}}
Let $(E,\theta)$ be a logarithmic graded semistable Higgs bundle over $X_{n+1}=\bP^1_{W_{n+1}}$ with $E=\mO\oplus \mO(1)$ and the modulo $p$ reduction of
\[\theta: \mO(1) \rightarrow \mO\otimes \Omega_{X}^1(\log D) \]
is nontrivial. Then the cokernel of
\[\theta: \mO(1) \rightarrow \mO\otimes \Omega_{X}^1(\log D)\]
is supported at one point $a\in\bP_{W_{n+1}}^1(W_{n+1})$, which is called the zero of the Higgs field. Conversely, for any given $a\in\bP_{W_{n+1}}^1(W_{n+1})$, up to isomorphic, there is a unique graded semistable logarithmic Higgs field on $\mO\oplus \mO(1)$ such that its zero equals to $a$. Assume $a\neq \infty$, we may choose and fix a basis $\widetilde{e}_{i,j}$ of $\mO(j-1)$ over $U_i$ for $1\leq i,j\leq 2$ such that
\begin{equation}
\mF_{12}=\left(\begin{array}{cc}
1 & 0\\0 & \frac{t}{t-1}\\
\end{array}\right),
\end{equation}
\begin{equation}
\omega_{\theta,1}=\left(\begin{array}{cc}
0 & \frac{t-a}{\lambda-1}\\0 & 0\\
\end{array}\right)\cdot\mathrm{d}\log\frac{t-\lambda}{t-1},\qquad \omega_{\theta^\epsilon,1}=\left(\begin{array}{cc}
0 & \frac{t-(a+p^n\delta)}{\lambda-1}\\0 & 0\\
\end{array}\right)\cdot\mathrm{d}\log\frac{t-\lambda}{t-1},
\end{equation}
Then lift these basis to a basis of $\widetilde{V}$. Choose $t_{ij}=t$, then by \autoref{lem_diff_p_conn}
\begin{equation} \label{equ_g12_g12epsilon}
\begin{split}
\frac{\mG_{ij}^\epsilon-\mG_{ij}}{p^n} & =
\Phi_j(\mF_{12}))\cdot \Phi_j(\frac{\theta^\epsilon-\theta}{\mathrm{d}t_{ij}})\cdot\frac{\Phi_i(t_{ij})-\Phi_j(t_{ij})}{p}\\
&= \Phi\left(\overline{\mF_{12}}\cdot \left(\begin{array}{cc}
0 & \frac{-\delta}{(t-\lambda)(t-1)}\\0 & 0\\
\end{array}\right)\right)\cdot\frac{\Phi_i(t_{ij})-\Phi_j(t_{ij})}{p}\\
&= \left(\begin{array}{cc}
0 & \frac{-\delta^p}{(t-\lambda)^p(t-1)^p}\cdot\frac{\Phi_i(t_{ij})-\Phi_j(t_{ij})}{p} \\0 & 0\\
\end{array}\right)\\
\end{split}
\end{equation}

\paragraph{\emph{Hodge filtration.}}
Since $X_1= \mathbb{P}^1_k$ and $(V,\nabla)$ is semi-stable of degree $p$, the bundle $V$ is isomorphic to $\mO(m) \oplus \mO(m+1)$ with $p=2m+1$. So the filtration on $(V,\nabla)$
\[
0\subset \mO(m+1) \subset V
\]
is the graded semi-stable Hodge filtration on $V$. Choose a basis $e_i$ of $\mO(m+1)$ on $U_i$ such that $e_1=\left(\frac{t}{t-1}\right)^{m+1}e_2$ on $U_{12}$.

\begin{lemma}\label{mainlem: f h}
$\mathrm{i)}$. Let $f,h$ be two elements in $R_1$. Then the $R_1$-linear map from $R_1\cdot e$ to $V(U_1)$, which maps $e_1$ to $\widetilde{v}_{11}\otimes_{\Phi_1} h+\widetilde{v}_{12}\otimes_{\Phi_1} f$, can be extended to a global map of vector bundles $\mO(m+1)\rightarrow V$ if and only if
\begin{equation}
\mG_{12}\left(h\atop f\right) \in \left(\frac{t}{t-1}\right)^{m+1}\cdot (R_2/p^{n+1}R_2)^2.
\end{equation}
\end{lemma}
\begin{proof} i). Over $U_{12}$, one has
\begin{equation}
\iota(e_2)=(\widetilde{v}_{21}\otimes_{\Phi_2} 1,\widetilde{v}_{22}\otimes_{\Phi_2} 1)\left(\frac{t-1}{t}\right)^{m+1}\cdot\mG_{12}\left(h\atop f\right)
\end{equation}
\end{proof}

\begin{lemma} Up to multiplying a unit, there exist a unique non-zero morphism
\[\mO_{X_{n+1}}(m+1)\rightarrow V.\]
In particular, up to a unit in $W_{n+1}^\times$, there exists a unique $h(a,t)$ and $f(a,t)$ such that
\begin{equation} \label{equ_g12_hf}
\mG_{12}\left(h(a,t)\atop f(a,t)\right) \in \left(\frac{t}{t-1}\right)^{m+1}\cdot (R_2/p^{n+1}R_2)^2
\end{equation}
Denote $\overline{h}(a,t)$ and $\overline{f}(a,t)$ the modulo $p$ reduction of $h(a,t)$ and $f(a,t)$ respectively.
\end{lemma}
\begin{proof}
Since the inverse Cartier functor is an equivalence and $(E,\theta)\pmod{p}$ is stable of degree $1$, the $(V,\nabla)\pmod{p}$ is also stable of degree $p$. Thus the underlying vector bundle $V\pmod{p}$ is of form $\mO_{X_1}(m+1)\oplus\mO_{X_1}(m)$. Since $\mO_{X_1}(m+1)\oplus\mO_{X_1}(m)$ has unique lifting $\mO_{X_{n+1}}(m+1)\oplus\mO_{X_{n+1}}(m)$ over $X_{n+1}$, thus $V\cong \mO_{X_{n+1}}(m+1)\oplus\mO_{X_{n+1}}(m)$. Then the lemma follows.
\end{proof}

\begin{lemma} \label{lem_diff_G12hf}
\begin{enumerate}[$(1).$]
\item one can choose $h(a,t)$, $f(a,t)$, $h(a+p^n\delta,t)$ and $f(a+p^n\delta,t)$ in $R_1/p^{n+1}R_1$ such that
\[\mG_{12}\left(h(a,t)\atop f(a,t)\right),\mG_{12}^\epsilon \left(h(a+p^n\delta,t)\atop f(a+p^n\delta,t)\right) \in \left(\frac{t}{t-1}\right)^{m+1}\cdot (R_2/p^{n+1}R_2)^2\]
and $h(a+p^n\delta,t)\equiv h(a,t)\pmod{p^n}$, $f(a+p^n\delta,t)\equiv f(a,t)\pmod{p^n}$.
\item Denote $\Delta h = \frac{h(a+p^n\delta)-h(a)}{p^n}\pmod{p}$, $\Delta f = \frac{f(a+p^n\delta)-f(a)}{p^n}\pmod{p}$. Then
\begin{equation} \label{equ_diff_g12_hf}
\overline{\mG}_{12}\left(\Delta h \atop \Delta f\right) + \left(\begin{array}{cc}
0 & \frac{-\delta^p}{(t-\lambda)^p(t-1)^p}\cdot\frac{\Phi_i(t_{ij})-\Phi_j(t_{ij})}{p} \\0 & 0\\
\end{array}\right)\left(\overline{h} \atop \overline{f}\right) \in \left(\frac{t}{t-1}\right)^{m+1}(R_2/pR_2)^2.
\end{equation}
\end{enumerate}
\end{lemma}

\begin{proof}
By \eqref{equ_g12_hf}, one has
\[\mG_{12}\left(h(a,t)\atop f(a,t)\right) \in \left(\frac{t}{t-1}\right)^{m+1}\cdot (R_2/p^{n+1}R_2)^2\]
and
\[\mG^\epsilon_{12}\left(h(a,t)+p^n\Delta h \atop f(a,t) + p^n\Delta f \right) \in \left(\frac{t}{t-1}\right)^{m+1}\cdot (R_2/p^{n+1}R_2)^2\]
Dividing the difference between two equations above by $p^n$ and considering its modulo $p$ reduction, one will get \eqref{equ_diff_g12_hf} by \eqref{equ_g12_g12epsilon}.
\end{proof}

\begin{corollary} There exists $u(\delta)$ such that
\[\Delta \overline{h} \equiv \overline{h}(a+\delta,t) -\overline{h}(a,t) + u(\delta)\overline{h}(a,t) \pmod{\delta^{2p}}\]
and
\[\Delta \overline{f} \equiv \overline{f}(a+\delta,t) - \overline{f}(a,t) + u(\delta)\overline{f}(a,t) \pmod{\delta^{2p}}.\]
\end{corollary}
\begin{proof}
 Denote by $(\overline{E}^\delta,\overline{\theta}^\delta)$ the graded Higgs bundle, which has the same underlying graded vector bundle with $(\overline{E},\overline{\theta})$ such that the zero of its Higgs field is $a+\delta$. By direct computation, one has
\[\overline{\mG}_{12}^\delta-\overline{\mG}_{12} = \left(\begin{array}{cc}
0 & \frac{-\delta^p}{(t-\lambda)^p(t-1)^p}\cdot\frac{\Phi_i(t_{ij})-\Phi_j(t_{ij})}{p} \\0 & 0\\
\end{array}\right)=\frac{\mG_{12}^\epsilon-\mG_{12}}{p^n} \pmod{p}.\]
\eqref{equ_g12_hf} implies that
\[\overline{\mG}_{12}(a+\delta,t)\left(\overline{h}(a+\delta,t)\atop \overline{f}(a+\delta,t)\right) \in \left(\frac{t}{t-1}\right)^{m+1}\cdot (R_2/pR_2)^2.\]
Then
\begin{equation}
\begin{split}
\mG_{12}&\left(\overline{h}(a+\delta,t)-\Delta \overline{h}\atop \overline{f}(a+\delta,t)-\Delta \overline{f}\right)\\
=\quad & \mG_{12}^\delta \left(\overline{h}(a+\delta,t)\atop \overline{f}(a+\delta,t)\right) -\left(\begin{array}{cc}
0 & \frac{-\delta^p}{(t-\lambda)^p(t-1)^p}\cdot\frac{\Phi_i(t_{ij})-\Phi_j(t_{ij})}{p} \\0 & 0\\
\end{array}\right)\left(\overline{h}(a+\delta,t)\atop \overline{f}(a+\delta,t)\right) -\mG_{12}\left(\Delta \overline{h}\atop \Delta \overline{f}\right) \\
\equiv \quad &
-\left(\begin{array}{cc} 0 & \frac{-\delta^p}{(t-\lambda)^p(t-1)^p}\cdot\frac{\Phi_i(t_{ij})-\Phi_j(t_{ij})}{p} \\0 & 0\\\end{array}\right)
\left[\left(\overline{h}(a+\delta,t)\atop \overline{f}(a+\delta,t)\right) - \left(\overline{h}(a,t)\atop \overline{f}(a,t)\right)\right] \\
& \mod \left(\frac{t}{t-1}\right)^{m+1}\cdot (R_2/pR_2)^2
\end{split}
\end{equation}
Since $\overline{f}$ and $\overline{h}$ are both inseparable in the first variable, $\left[\left(\overline{h}(a+\delta,t)\atop \overline{f}(a+\delta,t)\right) - \left(\overline{h}(a,t)\atop \overline{f}(a,t)\right)\right]$ is divided by $\delta^p$. Thus
\[\left(\overline{h}(a+\delta,t)-\Delta \overline{h}\atop \overline{f}(a+\delta,t)-\Delta \overline{f}\right) \in \delta^{2p}\cdot (R_{12}/pR_{12})^2 + \overline{\mG}_{12}^{-1}\left(\frac{t}{t-1}\right)^{m+1}\cdot (R_2/pR_2)^2.\]
Thus there exists $u(\delta)$ such that
\[\left(\overline{h}(a+\delta,t)-\Delta \overline{h}\atop \overline{f}(a+\delta,t)-\Delta \overline{f}\right) \equiv (1-u(\delta))\cdot \left(\overline{h}(a,t)\atop \overline{f}(a,t)\right) \pmod{\delta^{2p}}. \qedhere\]
\end{proof}

\paragraph{\emph{The Higgs field of the graded Higgs bundle.}}
We extend the local basis $v_{12}:=e_1$ of the Hodge filtration in $V$ over $U_1$ to a basis $\{v_{11},v_{12}\}$ of $V(U_1)$. Assume $v_{11}=\widetilde{v}_{11}\otimes_\Phi h_1 + \widetilde{v}_{12}\otimes_\Phi f_1$ and denote $P=\left(\begin{array}{cc}
h_1& h(a,t)\\ f_1 & f(a,t)\\
\end{array}\right)$, which is an invertible matrix over $\overline{R}_1$ with determinant $d:=\det(P)\in \overline{R}_1^\times$. One has
\begin{equation}
(v_{11},v_{12})=(\widetilde{v}_{11}\otimes_\Phi 1,\widetilde{v}_{12}\otimes_\Phi 1)\left(\begin{array}{cc} h_1 & h(a,t)\\f_1 & f(a,t)\\
\end{array}\right)
\end{equation}
and
\begin{equation}
\nabla(v_{11},v_{12})=
(v_{11},v_{12})\cdot \upsilon_{\nabla,1}
\end{equation}
where $\upsilon_{\nabla,1}=\left(P^{-1}\cdot \mathrm{d}P+ P^{-1}\cdot \omega_{\nabla,1} \cdot P\right)$.

Taking the associated graded Higgs bundle, the Higgs field $\theta'$ on $\mathrm{Gr}(V,\nabla,\Fil)(\overline{U}_1)=V(\overline{U}_1)/(\overline{R}_1\cdot v_{12}) \oplus \overline{R}_1\cdot v_{12}$ is given by
\begin{equation}\label{equ:gradingHiggsField}
\theta'(e_{12}')=
\frac1d\left(\frac{f(a,t)\mathrm{d}h(a,t)-h(a,t)\mathrm{d}f(a,t)}{\mathrm{d}\log \frac{t-\lambda}{t-1}}
+f(a,t)^2 \Phi_1\left(\frac{t-a}{\lambda-1}\right)\right)\cdot \left(e'_{11}\otimes \mathrm{d}\log \frac{t-\lambda}{t-1}\right)
\end{equation}
over $\overline{U}_1$, where $e'_{11}$ is the image of $v_{11}$ in $V(\overline{U}_1)/(\overline{R}_1\cdot v_{12})$ and $e'_{12}=v_{12}$ in $\overline{R}_1v_{12}$.
Thus the zero of the graded Higgs bundle $\mathrm{Gr}(V,\nabla,\Fil)$ is the root of polynomial
\begin{equation}
\begin{split}
P(a,t) & = \frac{f(a,t)\cdot\mathrm{d}h(a,t)-h(a,t)\cdot \mathrm{d}f(a,t)}{\mathrm{d}\log \frac{t-\lambda}{t-1}}
+f(a,t)^2\cdot \left(\frac{t-a}{\lambda-1}\right)^p\\
& =: L(a)t -C(a).
\end{split}
\end{equation}

\begin{lemma}\label{lem_P}
\[\frac{P(a+p^n\delta,t)-(1+p^nu(\delta))^2P(a,t)}{p^n}\equiv P(a+\delta,t)-P(a,t) \pmod{(p,\delta^{2p})}.\]
\end{lemma}
\begin{proof} Denote $\xi(f) = f(a+\delta,t)-f(a,t)$ and $\xi(h) = h(a+\delta,t)-h(a,t)$. Then
\[\frac{f(a+p^n\delta,t)-(1+p^nu(\delta))f(a,t)}{p^n}\equiv f(a+\delta,t)-f(a,t)\pmod{p}\]
and
\[\frac{h(a+p^n\delta,t)-(1+p^nu(\delta))h(a,t)}{p^n} \equiv h(a+\delta,t)-h(a,t) \pmod{p}.\]
Then the lemma follows direct computation.
\end{proof}

\begin{corollary} \label{cor_explicitCalcu}
\[\frac{\frac{C(a+p^n\delta)}{L(a+p^n\delta)}-\frac{C(a)}{L(a)}}{p^n}\equiv \frac{C(a+\delta)}{L(a+\delta)}-\frac{C(a)}{L(a)} \pmod{(p,\delta^{2p})}.\]
\end{corollary}
\begin{proof}
From \autoref{lem_P}, we have
\[L(a+p^n\delta) \equiv (1+p^nu(\delta))^2L(a) + p^n(L(a+\delta)-L(a)) \pmod{p,\delta^{2p}}\]
and
\[C(a+p^n\delta) \equiv (1+p^nu(\delta))^2C(a) + p^n(C(a+\delta)-C(a))\pmod{p,\delta^{2p}}.\]
Thus
\begin{equation*}
\begin{split}
\frac{\frac{C(a+p^n\delta)}{L(a+p^n\delta)}-\frac{C(a)}{L(a)}}{p^n}
=& \frac{C(a+p^n\delta)L(a) - C(a)L(a+p^n\delta)}{p^nL(a+p^n\delta)L(a)}\\
\equiv & \frac{C(a+\delta)L(a) - C(a)L(a+\delta)}{L(a+\delta)L(a)} \\
=&\frac{C(a+\delta)}{L(a+\delta)}-\frac{C(a)}{L(a)} \quad \pmod{(p,\delta^{2p})}
\end{split}
\end{equation*}
\end{proof}

Recall the theorem in \cite{SYZ22}, the map
\[\varphi_{\lambda,p}(\overline{a}) := \frac{C(a)}{L(a)} \pmod{p}\]
is an inseparable rational polynomial of degree $p^2$ in variable $\overline{a}=a\pmod{p}$. Denote by $\widetilde{\varphi}_{\lambda,p}$ the rational polynomial such that
\[\varphi_{\lambda,p}(a) = \widetilde{\varphi}_{\lambda,p}(a^p).\]
As a consequence of \autoref{cor_explicitCalcu}, we get the main result of this appendix.
\begin{theorem}
The torsor map induces by Higgs-de Rham flow is of form
\[\widetilde{\varphi}'_{\lambda,p}(\overline{a}^p)\cdot t+\gamma.\]
\end{theorem}

\end{document}